\documentclass{article}
\usepackage{amsmath}
\usepackage{amsthm}
\usepackage{amssymb}
\usepackage[colorlinks=true, pdfstartview=FitV, linkcolor=blue,citecolor=blue, urlcolor=blue]{hyperref}
\usepackage{verbatim}
\usepackage{graphicx}

\usepackage{times}

\usepackage[usenames,dvipsnames,svgnames,table]{xcolor}
\usepackage[hmargin=3cm,vmargin=3.5cm]{geometry}
\numberwithin{equation}{section}

\def\a{\alpha}
\def\b{\beta}
\def\ga{\gamma}

\def\de{\delta}
\def\De{\Delta}
\def\ep{\epsilon}

\def\la{\lambda}
\def\La{\Lambda}
\def\si{\sigma}

\def\Om{\Omega}

\renewcommand{\th}{\theta}
\newcommand{\Th}[0]{\Theta}

\def\nab{\nabla}
\def\varep{\varepsilon}

\def\CC{{\cal C}}
\def\DD{{\cal D}}

\def\II{{\cal I}}

\def\DD{{\cal D}}

\newcommand{\N}[0]{\mathbb{N}}

\newcommand{\R}[0]{\mathbb{R}}
\newcommand{\Z}[0]{\mathbb{Z}}

\newcommand{\C}[0]{\mathbb{C}}

\newcommand{\T}[0]{\mathbb{T}}

\newcommand{\supp}[0]{\mathrm{supp} \,}

\newcommand{\fr}[2]{\frac{#1}{#2}}

\newcommand{\ALI}[1]{\begin{align*} #1 \end{align*}}
\newcommand{\tx}[1]{\mbox{#1}}

\newcommand{\pr}[0]{\partial}

\renewcommand{\div}{\mathop{div} \,}

\newcommand{\co}[1]{\| #1 \|_{C^0}}

\newcommand{\Ddt}[0]{\fr{\bar D}{\partial t}}

\newcommand{\DDdt}[0]{\fr{\bar D^2}{\partial t^2}}

\newcommand{\ali}[1]{ \begin{align} #1 \end{align} }

\def\XXint#1#2#3{{\setbox0=\hbox{$#1{#2#3}{\int}$}
     \vcenter{\hbox{$#2#3$}}\kern-.5\wd0}}

\newtheorem{thm}{Theorem}[section]
\newtheorem{lem}{Lemma}[section]

\newtheorem{prop}{Proposition}[section]
\newtheorem{cor}{Corollary}[section]

\newtheorem{Claim}{Claim}[section]

\theoremstyle{definition}
\newtheorem{defn}{Definition}[section]

\theoremstyle{remark}
\newtheorem{rem}{Remark}

\title{ H\"{o}lder Continuous Solutions of Active Scalar Equations }
\author{ Philip Isett\thanks{Department of Mathematics, MIT, Cambridge, MA 02139, USA. (\href{mailto:isett@math.mit.edu}{isett@math.mit.edu}).}\, and Vlad Vicol\thanks{Department of Mathematics, Princeton University, Princeton, NJ 08544, USA. (\href{mailto:vvicol@math.princeton.edu}{vvicol@math.princeton.edu}).}}
\date{  }

\begin{document}
\maketitle

\begin{abstract}
 We consider active scalar equations $\partial_t \theta + \nabla \cdot (u \, \theta) = 0$, where $u = T[\theta]$ is a divergence-free velocity field, and $T$ is a Fourier multiplier operator with symbol $m$. We prove that when $m$ is not an odd function of frequency, there are nontrivial, compactly supported solutions weak solutions, with H\"older regularity $C^{1/9-}_{t,x}$. In fact, every integral conserving scalar field can be approximated in $\DD'$ by such solutions, and these weak solutions may be obtained from arbitrary initial data. We also show that when the multiplier $m$ is odd, weak limits of solutions are solutions, so that the $h$-principle for odd active scalars may not be expected. 
\end{abstract}

\setcounter{tocdepth}{2}
\tableofcontents

\section{Introduction}

The present paper is concerned with existence, nonuniqueness and results of $h$-principle type for H\"{o}lder continuous, weak solutions to inviscid {\em active scalar equations} with a divergence free drift velocity.  These equations have the form
\ali{
\label{eq:activeScalar}
\begin{split}
\pr_t \th + \pr_l(\th u^l) &= 0 \\
u^l &= T^l[\th]  \\
\pr_l u^l &= 0.
\end{split}
}
The operator $T^l[ \cdot ]$ defining the drift velocity $u^l$ in \eqref{eq:activeScalar} is represented in frequency space by a multiplier
\ali{
{\hat u}^l(\xi) = {\widehat T}^l[\th](\xi) &= m^l(\xi) {\hat \th}(\xi).
}
We assume that $m^l(\xi)$ is defined on the whole frequency space as a tempered distribution and is homogeneous of degree $0$ so that $T^l$ is an operator of order $0$.  The multiplier must satisfy $m^l(-\xi) = \overline{m^l(\xi)}$ so that the drift-velocity $u^l$ is real-valued whenever the scalar $\th$ is real-valued, and we assume that $m^l(\xi)$ is smooth away from the origin.  
The requirement that $u^l$ is divergence free corresponds to the requirement that $m^l(\xi)$ takes values perpendicular to the frequency vector $\xi$, i.e. $\xi \cdot m(\xi) = 0$ for $\xi \neq 0$. 

Active scalar equations arise from the full Navier-Stokes, Euler, or magneto-hydrodynamic equations in a number of physical regimes, such as stratification, rapid rotation, hydrostatic, and geostrophic balance. Physically motivated examples include:
\begin{enumerate}
 \item The {\em surface quasi-geostrophic} (SQG) equation~\cite{ConstantinMajdaTabak94,HeldPierrehumbertGarnerSwanson95}. Here 
 \[ m(\xi) = i \langle-\xi_2,\xi_1\rangle |\xi|^{-1}\] 
 is an {\em odd} symbol, bounded and smooth on the unit sphere. The SQG equation belongs to a general class of active scalar equations (with odd constitutive law $T$) satisfied by the vorticity of a generalized two-dimensional Euler equation on a Lie algebra (\'a la Arnold~\cite{arnoldEuler}) with a specific inner product~\cite{Resnick95} (see also \cite{TaoBlog} for a more recent account).
 
 \item The {\em incompressible porous media} (IPM) equation with velocity given by Darcy's law~\cite{Bear72,CordobaGancedoOrive07}. Here 
 \[ m(\xi) = \langle \xi_1 \xi_2, - \xi_1^2 \rangle |\xi|^{-2}\] 
 is an {\em even} symbol, bounded and smooth on the unit sphere. Note that the IPM equation has a three-dimensional analogue, with symbol $m(\xi) = \langle\xi_1 \xi_3, \xi_2 \xi_3, -\xi_1^2 - \xi_2^2\rangle |\xi|^{-2}$, which is again even. Our proof applies to this three-dimensional case as well, cf.~Remark~\ref{rem:3D} below.
 
 \item The {\em magneto-geostrophic} (MG) equation~\cite{MoffattLoper94,Moffatt08,FriedlanderVicol11a}. This is a three-dimensional active scalar equation, with symbol given by 
 \[ m(\xi) = \Big\langle \xi_2 \xi_3 |\xi|^2 + \xi_1 \xi_2^2 \xi_3, - \xi_1\xi_3 |\xi|^2+\xi_2^3 \xi_3, -\xi_2^2(\xi_1^2+\xi_2^2) \Big\rangle (\xi_3^2 |\xi|^2 + \xi_2^4)^{-1} \]
 for all $\xi \in \Z^3_*$ with $\xi_3\neq 0$, and by $m(\xi_1,\xi_2,0) = 0$. The symbol of the MG equation is {\em even} and zero-order homogenous, but as opposed to the previous examples, it is not bounded. This unboundedness may be seen by evaluating the symbol on a parabola $m(\zeta^2, \zeta,1)$, and passing $|\zeta|\to \infty$.  Nonetheless, the proof in our paper still applies to the MG equations as we only require smoothness in a neighborhood of finitely many points, cf. ~Remark~\ref{rem:non-smooth} below.
\end{enumerate}

Remarkably, from the mathematical point of view these scalar equations retain some of the same essential difficulties of the full fluid equations. In particular, the global well-posedness for the 2D SQG and IPM equations remains open, in analogy to the 3D Euler equations. More relevant for this paper, the regularity class in which the conservation of the energy $\|\theta\|_{L^2}^2$ may be established for weak solutions of \eqref{eq:activeScalar}, is  H\"older continuity with exponent $1/3$, as for 3D Euler. However, due to their more rigid geometry (e.g. no known analogue for Beltrami flows), their non-local nature, and the presence of infinitely many conservation laws (the $L^p$ norms of $\theta$, for any $p\geq 1$), the construction of weak solutions that fail to conserve energy appears to be more restrictive than for 3D Euler.

The pair $(\theta, u^l)$ is called a weak solution of \eqref{eq:activeScalar} if the equations \eqref{eq:activeScalar} are satisfied on $\R \times \T^2$ in the sense of distributions.  
When $(\th, u^l)$ are continuous, it is equivalent to require the balance laws
\[ \fr{d}{dt} \int_\Om \th(t,x) dx = \int_{\pr \Om} \th~u(t,x) \cdot n ~ d\si, \qquad 
\int_{\pr \Om} u(t,x) \cdot n ~d\si = 0 \]
to be satisfied as continuous functions of time for all subdomains $\Om$ with smooth boundary and inward unit normal $n$.  The definition of weak solution implies immediately that the integral $\int_{\T^2} \th(t,x) ~dx$ is a conserved quantity, but this definition does not immediately imply the other conservation laws that hold for classical solutions (see also~\cite{BTW12,BSW13} for comparisons with other notions of non-classical solutions for the Euler equations).

The study of weak solutions in fluid dynamics, including those which fail to conserve energy, is natural in the context of turbulent flows. The power spectrum predicted by Kolmogorov~\cite{K41} implies that solutions which arise in the inviscid limit of the 3D Navier-Stokes equations have H\"older $1/3$ regularity on average, and in particular are not classical.  Such flows are expected to exhibit anomalous dissipation of energy, rather than conserving energy.  The exponent $1/3$ is the same regularity threshold conjectured by Onsager~\cite{onsag} to be critical for energy conservation in the 3D Euler equations (see~\cite{BT13,DeLellisSzekelihidiBull,ShvydkoyLectures} for recent reviews).  For power spectra in  active scalar turbulence, we refer to Kraichnan~\cite{kraichnan2003small} and Constantin~\cite{Constantin98,Constantin02}.


Our first main result, Theorem~\ref{thm:mainThm}, shows that if the symbol of the multiplier $m^l(\xi)$ is not an odd function of $\xi$ for $\xi \neq 0$, there exist nontrivial, space-periodic solutions in two dimensions with compact support in time, having any H\"{o}lder regularity $\th \in C_{t,x}^\a$ with $\a < 1/9$.  
In contrast, the energy $\int |\th|^2(t,x) dx$ is a conserved quantity for solutions with H\"{o}lder regularity above $\a > 1/3$ and for classical solutions the quantity $\th^2$ is also advected by the drift velocity $u^l = T^l[\th]$, whereas both these properties clearly fail for our solutions.  This result gives the first proof of nonuniqueness of continuous weak solutions for any active scalar equation of this type.

\begin{thm}[Weak Solutions to Active Scalar equations]\label{thm:mainThm}  
Consider the active scalar equation \eqref{eq:activeScalar} with divergence free drift velocity, and assume that the multiplier $m^l(\xi)$ defining the operator $T^l$ is not an odd function of $\xi$ for $\xi \neq 0$.  Let $\a < 1/9$ and let $I$ be an open interval.  Then there exist nontrivial solutions to \eqref{eq:activeScalar} with H\"{o}lder regularity $\th, u^l \in C_{t,x}^\a(\R \times \T^2)$ which are identically $0$ outside of $I \times \T^2$.

Moreover, if $f : \R \times \T^2 \to \R$ is a smooth scalar function with compact support on $I \times \T^2$ which satisfies the conservation law $\fr{d}{dt} \int_{\T^2} f(t,x) ~dx = 0$, then there exists a sequence of weak solutions $\th_n : \R \times \T^2 \to \R$ to \eqref{eq:activeScalar} in the above regularity class such that $\th_n$ converges to $f$ in the $L^\infty$ weak-* topology, and each $\th_n$ has compact support in  $I \times \T^2$.
\end{thm}

The above result builds upon the recent works by C\'{o}rdoba, Faraco, Gancedo~\cite{corFarGanPor}, Shvydkoy~\cite{shvConvInt}, and Sz\'{e}kelyhidi~\cite{laszloIPM} which establish the non-uniqueness of $L^\infty_{t,x}$ weak solutions to the IPM equations and $2D$ active scalar equations with even symbols $m$.  These previous works are based on a variant of the method of convex integration introduced for the Euler equations in \cite{deLSzeIncl} that provides an effective and elegant approach to producing bounded solutions, but which faces a major obstruction to producing continuous solutions.  For the Euler equations, this obstruction was overcome in \cite{deLSzeCts, deLSzeHoldCts, deLSzeCts2d} to produce continuous and $C^\a$ solutions on $\T^2$ and $\T^3$.  A crucial idea to overcome this obstruction is a key cancellation coming from the use of special families of stationary, plane wave solutions which allows for the control of interference terms between different waves in the construction.  For $3D$ Euler, these solutions are Beltrami flows (eigenfunctions of the curl operator), while for $2D$ Euler they are rotated gradients of Laplace eigenfunctions.

There is an obstruction to generalizing these ideas to obtain continuous solutions to active scalar equations, which is that analogous families of stationary, plane wave solutions do not exist in general for active scalar equations.  Furthermore, as we explain more precisely in Section~\ref{sec:stressTerm}, there is a sense in which no analogous cancellation is ever available under the assumptions of Theorem \ref{thm:mainThm}.  The same difficulty has also prohibited this approach from generalizing to the Euler equations in higher dimensions, even though similar results in principle could be expected to hold in any dimension.  (The conservation of energy for regularity above $1/3$ holds in any dimension, and the approach of \cite{deLSzeIncl} for constructing $L_{t,x}^\infty$ solutions applies in any dimension.)

The main idea that forms the starting point of our work is a new, more general, mechanism for obtaining the cancellation of interference terms in the construction, which arises without any special Ansatz in the construction.  Our observation is that the interference terms which arise when an individual wave interacts with itself must always cancel thanks to the divergence free structure of the equation, even though we lack a general method for controlling the interference between waves which oscillate in different directions.  This observation opens the door to a serial iteration scheme based on one-dimensional oscillations, as in the original scheme of Nash \cite{nashC1}.  The same observation applies to both the Euler equations and to general active scalar equations regardless of the dimension (c.f. Remark~\ref{rem:3D}).  Our proof therefore gives a new approach to constructing continuous and $C^\a$ weak solutions to these equations that is independent of the use of Beltrami flows or the analogue.    

Although the regularity obtained in Theorem~\ref{thm:mainThm} is strictly worse than the results which have been obtained for the Euler equations, the exponent $1/9$ is the best result we can hope to obtain from our method.  For the Euler equations, solutions in the class $C_{t,x}^{1/5-}$ were constructed in \cite{isett}, with another proof given by Buckmaster, De Lellis and Sz\'{e}kelyhidi \cite{deLSzeBuck}.  The construction in \cite{deLSzeBuck} has recently been refined in \cite{buckDeLSzeOnsCrit} to give continuous solutions in the class $L_t^1 C_x^{1/3 -}$, improving significantly a result of Buckmaster \cite{Buckmaster}.  A main obstruction to higher regularity faced by all of these works and also the present paper is the presence of anomalously sharp time cutoffs.  These cutoffs lead to bounds on advective derivatives which are inferior to the bounds that hold for solutions with higher regularity, cf.~\cite[Sec. 9]{isett2} and \cite[Sec. 1.1.3]{isettOhAng}.  
In our case, we face an additional loss of regularity which comes from our inability to eliminate more than one component of the error in a given stage of the iteration.  The same obstruction to regularity arises for the isometric embedding equation \cite{deLSzeC1iso}.  For active scalars, we must deal with both obstructions at the same time, and improving on either one seems to be a difficult problem.

Our approach to proving Theorem~\ref{thm:mainThm} also yields the following result, which shows that our construction can realize arbitrary smooth initial data.
\begin{thm}\label{thm:arbitraryNonuniqueness}  Let $I = (-T, T)$ be a finite open interval containing the origin, let $\a < 1/9$ and let $(\th_{(0)}, u_{(0)}^l)$ be a smooth solution to \eqref{eq:activeScalar} on $I \times \T^2$.  Then there exists a global, weak solution $(\th, u^l)$ to \eqref{eq:activeScalar} in the class $(\th, u^l) \in C_{t,x}^\a(\R \times \T^2)$ which coincides with $(\th_{(0)}, u_{(0)}^l)$ on the time interval 
\[ \th(t,x) = \th_{(0)}(t,x) \qquad (t,x) \in (-T/2, T/2) \times \T^2 \]
and which coincides with a constant
\[ \th(t,x) = \bar{\th} \]
for $(t,x) \notin (-4T/5, 4T/5) \times \T^2$.
\end{thm}
To the best of our knowledge, Theorem~\ref{thm:arbitraryNonuniqueness} gives the first proof of global existence of weak solutions for \eqref{eq:activeScalar} with multipliers $m$ which are not odd, from arbitrary smooth initial data~\cite{corFarGanPor}. The global existence of weak solutions appears to be only known for odd symbols~\cite{Resnick95,ChaeConstantinCordobaGancedoWu12}, or for patch-type initial datum in the IPM equations~\cite{CordobaGancedo07}. Thus, in view of the known existence result for odd multipliers, we show that all active scalar equations with smooth constitutive law have global in time weak solutions (see also Remark~\ref{rem:non-smooth}).

Our method of construction demonstrates not only the existence of weak solutions, but also the abundance and flexibility of solutions in the class $C_{t,x}^{1/9 - \ep}$.  This point is emphasized by the following result of ``$h$-principle'' type, which follows from Theorem~\ref{thm:mainThm}, and completely characterizes the weak-* closure of these solutions in $L^\infty$.  The result illustrates that, within this regularity class, the conservation of the integral is the only source of rigidity for solutions to the equations that is stable in the weak-* topology\footnote{One must be cautious that Corollary~\ref{cor:weakLims} below does not assert that integral-conserving $L^\infty$ functions can be approximated by a weak-* convergent {\it sequence} of solutions as in the statement of Theorem~\ref{thm:mainThm}.  Such a statement would be false, since the functions $f$ obtained as weak-* limits of sequences will also inherit {\it time regularity} of the type $\pr_t f \in L_t^\infty W_x^{-1, p}$ from the equation~\eqref{eq:activeScalar}.}.  We refer to \cite{deLSzeHFluid, choff} for more on $h$-principles for fluid equations.  

\begin{cor}[$h$-principle for Active Scalar Equations]\label{cor:weakLims}  Consider the $2D$ active scalar equation \eqref{eq:activeScalar} as in the hypotheses of Theorem \ref{thm:mainThm}, with multiplier $m$ that is not odd.  Then for any $\a < 1/9$ and for any open interval $I$, the closure in the weak-* topology on $L^\infty(I \times \T^2)$ of the set of $C_{t,x}^{\a}$ solutions to \eqref{eq:activeScalar} with compact support in $I \times \T^2$ is equal to the space of real-valued $f \in L^\infty(I \times \T^2)$ which satisfy the conservation law $\int_{\T^2} f(t,x) dx = 0$ as a distribution in time.
\end{cor}


While Theorems~\ref{thm:mainThm}-\ref{thm:arbitraryNonuniqueness}  and Corollary~\ref{cor:weakLims} illustrate an utter lack of rigidity for multipliers which are not odd, we find a much more rigid situation for weak solutions in the case of odd multipliers.  The following result implies that, when the multiplier is odd, every weak limit of solutions in $L_{t,x}^\infty$ must also be a solution to the same active scalar equation, in stark contrast to Theorem~\ref{thm:mainThm} and Corollary~\ref{cor:weakLims}.  This theorem generalizes the statement at the end of \cite{deLSzeHFluid} concerning weak rigidity for SQG, and makes precise the assumptions necessary for this rigidity.

\begin{thm}[Weak Rigidity for Active Scalars with Odd Multipliers]\label{thm:theOddCase}  Consider the active scalar equation~\eqref{eq:activeScalar} in any dimension, with divergence free drift velocity, and assume that the multiplier $m^l(\xi)$ defining the operator $T^l$ is an odd function of $\xi$ for $\xi \neq 0$.  Suppose that $f = \lim_n \th_n$ is a weak limit of solutions to \eqref{eq:activeScalar} in $L^p(I;L^2(\T^n))$, for some $p>2$.  Then $f(t,x)$ must be a weak solution to \eqref{eq:activeScalar}.   
\end{thm}
We note that the $L^p$ time integrability condition on $\theta_n$ is by no means restrictive. Indeed, due to the incompressible transport nature of \eqref{eq:activeScalar}, weak solutions constructed via smooth approximations (e.g.~vanishing viscosity) are in fact bounded, or even weakly continuous in time.  

The proof of Theorem \ref{thm:theOddCase} is based on the approach of~\cite{Resnick95}, where global $L^\infty_t L^2_x$ weak solutions of the surface quasi-geostrophic equations are constructed. The main idea is that odd multipliers $m$ induce a certain commutator structure in the nonlinear term, which yields the necessary compactness. In fact, the oddness of $m$ implies that the equations are well-posed, even if the operator $T^l$ is not of degree $0$ (see~\cite{ChaeConstantinCordobaGancedoWu12}), and in such cases the oddness appears to be necessary~\cite{FriedlanderGancedoSunVicol12,FriedlanderVicol11c}.

In addition to the weak rigidity of Theorem~\ref{thm:theOddCase}, in the following theorem we show that every active scalar equation in $2D$ with odd symbol has a Hamiltonian that is conserved for solutions in the class $L_{t,x}^3$.
\begin{thm}[Conservation of the Hamiltonian for Active Scalars with Odd Multipliers]\label{thm:odd:hamiltonian}  Consider the active scalar equation \eqref{eq:activeScalar} in two dimensions with divergence free drift velocity and odd multiplier as in Theorem~\ref{thm:theOddCase}. Define the operator
\ali{
L = (-\Delta)^{-1} (\nabla \cdot T^\perp) = (-\Delta)^{-1/2} (R_2 T^1 - R_1 T^2)
\label{eq:L:def}
} 
where $R_i$ is the $i^{th}$ Riesz transform. The fact that $m$ is odd, implies that $L$ is self-adjoint. 
Define the Hamiltonian
\begin{align}
H(t) = \int_{\T^2} \theta(t,x) L\theta(t,x) dx.
\label{eq:Hamiltonian:def}
\end{align}
Then, if $\th$ is a solution to \eqref{eq:activeScalar} in the class $\th \in L_{t,x}^3$, the function $H(t)$ is constant in time.  
\end{thm}
We note that due to the transport structure of \eqref{eq:activeScalar}, solutions which are obtained by smooth approximations, such as viscosity approximations, Galerkin truncations, etc, will automatically lie in $L^\infty_{t,x}$, and thus also in $L^3_{t,x}$. 

Theorem~\ref{thm:theOddCase} precludes any results such as Theorems~\ref{thm:mainThm}-\ref{thm:arbitraryNonuniqueness} from holding in the case of the SQG equation, in which case $L = (-\Delta)^{-1/2}$ and we obtain the conservation of the $H^{-1/2}$ norm for solutions in $L_{t,x}^3$. Note however that in general the operator $L$ need not be coercive, as is the case when $m$ vanishes somewhere on the unit sphere. We refer to \cite{Resnick95,TaoBlog} for an exposition of how the quantity $H(t)$ serves as a Hamiltonian for the equation. 

We conclude our introduction by remarking on how our method extends to higher dimensions, and to the case of multipliers which are not smooth.

\begin{rem}[Higher Dimensions] \label{rem:3D}
Our proof generalizes to active scalar equations in arbitrary dimensions (c.f. Section~\ref{sec:explain:higherD} for the relevant modifications).  In this case, however, there are two further restrictions.  First of all, the regularity we obtain becomes worse as the dimension increases.  The same type of loss (for essentially the same reason, see Section~\ref{sec:compareEulerIso} below) is also seen in the case of the isometric embedding equations \cite{deLSzeC1iso}.  Second, we cannot obtain our result for all smooth multipliers whose symbols are not odd, and we require a nondegeneracy condition on the even part of the multiplier.  
\end{rem}
The precise result we obtain is the following:
\begin{thm}[Multi-dimensional Case]\label{thm:nDims}
Consider the active scalar equation \eqref{eq:activeScalar} with divergence free drift velocity on $\T^d$.  Assume also that the image of the even part of the multiplier contains $d$ vectors 
\ali{
 A_{(i)} &= m(\xi^{(i)}) + m(-\xi^{(i)}), \qquad i = 1, 2, \ldots, d  \label{eq:vectorsInImage}
}
such that the vectors $A_{(1)}, \ldots, A_{(d)}$ span $\R^d$.  Then Theorems~\ref{thm:mainThm}-\ref{thm:arbitraryNonuniqueness}  and Corollary~\ref{cor:weakLims} hold as stated, but with the condition $\a < \fr{1}{9}$ on the H\"older exponent being replaced by
\[ \a < \fr{1}{1 + 4 d}.\]
\end{thm}
Theorem~\ref{thm:nDims} applies in particular to the $3D$ IPM equation, and in that case yields weak solutions with H\"older regularity $\a < 1/13$.  Note also that Theorem~\ref{thm:nDims} generalizes the two dimensional case of Theorem~\ref{thm:mainThm}.  Namely, if the even part $m(\xi^{(1)}) + m(-\xi^{(1)}) \neq 0$ is nonzero at a single point, it follows already from incompressibility (i.e. the condition $m(\xi) \cdot \xi = 0$) that the span of the image of the even part of $m$ has dimension at least $2$.

The assumption \eqref{eq:vectorsInImage} in Theorem~\ref{thm:nDims} arises quickly from the proof and turns out to be necessary for the conclusion of Theorem~\ref{thm:mainThm}.  That is to say, when the assumption \eqref{eq:vectorsInImage} fails, there are in general additional constraints on weak limits of solutions besides the conservation of the mean value.  
In the case where the multiplier is even, such constraints arise from the conservation of the integrals
\[ \fr{d}{dt} \int_{\T^n} \th(t,x) \Psi(x) dx = 0 \]
for functions $\Psi$ whose gradients take values perpendicular to the image of the multiplier.  More generally, we have the following theorem which can be applied to every multiplier that fails to satisfy \eqref{eq:vectorsInImage}:
\begin{thm}[Constraints on Weak Limits of Degenerate Multipliers]\label{thm:more:constraints}  
Consider the active scalar equation \eqref{eq:activeScalar} on a torus $\T^n$ of any dimension and suppose that the image of the even part of the multiplier lies in a hyperplane perpendicular to some nonzero vector $\xi_{(0)} \in \widehat{\T^n}$ in the dual lattice.  Then there exists a smooth function of compact support $f \in \CC_0^\infty(\R \times \T^n)$ which is real-valued and satisfies the conservation law $\int_{\T^n} f(t,x) dx = 0$ such that $f$ cannot be realized as a weak-* limit in $L^\infty$ of any sequence of bounded weak solutions to \eqref{eq:activeScalar}.
\end{thm}
The proof of Theorem~\ref{thm:more:constraints} draws on the proof of weak compactness in Theorem~\ref{thm:theOddCase}.  One can compare condition \eqref{eq:vectorsInImage} to criteria for having a large $\La$-convex hull in the theory of differential inclusions (e.g. \cite{kMuSve,deLSzeIncl,laszloIPM}).

\begin{rem}[Non-smooth Symbols] \label{rem:non-smooth}
In view of the example of the MG equation, it is important to remark that our proof applies also to multipliers which are not smooth.  In fact, the only regularity condition we require in our proof is that the multiplier should be smooth in a neighborhood of the points $\xi^{(1)}, \xi^{(2)}, \ldots, \xi^{(d)}$ and $- \xi^{(1)}, -\xi^{(2)}, \ldots, -\xi^{(d)}$ appearing in \eqref{eq:vectorsInImage}.  Thus Theorem~\ref{thm:nDims} applies to the MG equation, if we take for example the points $\xi^{(1)} = \langle 1, 0, 1 \rangle$, $\xi^{(2)} = \langle 0, 1, 1 \rangle$, $\xi^{(3)} = \langle 1, 1, 1 \rangle$.
\end{rem}

\subsection{Difficulties and new ideas}

The proof of Theorem~\ref{thm:nDims} contains a number of new ideas in the method of convex integration, which we summarize before we begin the proof.  

As stated earlier in the Introduction, our main idea is a new mechanism for obtaining cancellations in interference terms between overlapping waves. This allows us to get around the lack of Beltrami flows, or their analogues, as the type of cancellation given by such flows is entirely unavailable in our setting (cf.~Section~\ref{sec:highFreqInterference}). This idea gives a new and general approach to constructing continuous weak solutions\footnote{We note, however, this idea alone obtains a lesser H\"{o}lder regularity compared to the Beltrami flow approach to Euler.} which generalizes also to Euler.  The idea is based on the observation that self-interference terms vanish automatically thanks to the incompressible nature of the equation.  

The above idea opens the door to a multi-stage iteration scheme based on one-dimensional oscillations, as in the original scheme of Nash for isometric embeddings applied in \cite{nashC1, deLSzeC1iso}.  This type of scheme had previously appeared unavailable in the setting of the Euler equations (see \cite[Section 1.3, Comment 2]{deLSzeCts}).  On the other hand, while implementing a scheme exactly of this type now appears to be possible, it also appears to be relatively complicated, requiring the addition of several iterations of waves (each with their own time, length scale and frequency parameters) before the error improves in the $C^0$ norm.  We manage to avoid these complications by defining a space of approximate solutions by a {\it compound scalar stress equation}.  This concept allows us to obtain a $C^0$ improvement after only one iteration, which simplifies the iteration and gives estimates which are much closer to the bounds familiar from the case of Euler.

The main new technical difficulty in obtaining continuous solutions to active scalar equations
lies in how to deal with the integral operator in the equation which determines the drift velocity $u^l = T^l[\th]$.  The whole construction is based on high frequency, plane-wave type corrections of the form $e^{i \la \xi_I(t,x)} \th_I(t,x)$, and it is necessary to understand very precisely how adding such waves will affect the drift velocity.  Furthermore, the convex integration schemes for producing H\"{o}lder continuous Euler flows all use heavily $C^0$ type estimates on all error terms.  From this point of view, the failure of $C^0$ boundedness of $T^l$ suggests some serious trouble.  

Our main technical device for addressing this difficulty is a ``Microlocal Lemma'' (Lemma~\ref{lem:microlocal}).  This lemma makes precise how a convolution operator behaves to leading order like a multiplication operator when given a high-frequency plane wave input, allowing for the use of nonlinear phase functions.  In the case of the operator $T^l$, represented on the Fourier side by the multiplier $m^l(\xi)$, our lemma gives a statement of the form
\[ u^l = T^l[e^{i \la \xi(x)} \th(x) ] = e^{i \la \xi(x)}( \th m^l(\nab \xi(x)) + \de u^l) \]
and gives an explicit formula for the error term $\de u^l$ (which also allows us to estimate its spatial and advective derivatives).  We expect that this technique should be of independent interest for other applications.  

To address the lack of $C^0$ boundedness of $T^l$, our proof makes additional use of the frequency localization in the construction, which allows for the effective application of the Microlocal Lemma.  A number of other simplifications in the argument arise from the use of frequency localized waves.  For instance, many error terms can be estimated in a simpler way than in previous works, and we remove the need for nonstationary phase arguments in solving the relevant elliptic equations.

In connection with our space of approximate solutions, we introduce a family of estimates we call {\it compound frequency energy levels}.  These estimates generalize to active scalars the frequency energy levels introduced in \cite{isett}.  These bounds have the key feature that they carry $C^0$ type estimates for derivatives of the drift velocity along the iteration.  Otherwise, the lack of $C^0$ boundedness of $T^l$ would prohibit us from deducing these estimates from the bounds on the scalar field. 

\subsection{Outline of the Paper}
The overall strategy for the construction is outlined in Section~\ref{sec:basicTechOutline}.  The bulk of the paper then consists of proving the ``Main Lemma'', Lemma~\ref{sec:mainLem}, which is stated in Section~\ref{sec:mainLem}.  After the statement of the Main Lemma, Section~\ref{sec:microlocalLem} is devoted to the proof of a ``Microlocal Lemma'', which is one of the main technical tools in the paper.  Sections~\ref{sec:construction}-\ref{sec:boundNewStress} are then devoted to proving Lemma~\ref{sec:mainLem}.

In Section~\ref{sec:mainLemWorks}, we explain how the Main Lemma implies the results stated in Theorem~\ref{thm:mainThm} and Corollary~\ref{cor:weakLims}.  Section~\ref{sec:arbitNonuniqueProof} provides an outline of how Theorem~\ref{thm:arbitraryNonuniqueness} also follows from the same Lemma.  The modifications used to prove Theorem~\ref{thm:nDims} regarding higher dimensions are explained in Section~\ref{sec:explain:higherD}. 

Sections \ref{sec:rididityOdd} and \ref{sec:hamiltonianOdd} are devoted to the rigidity properties of weak solutions in the case of odd multipliers.  In Section~\ref{sec:rididityOdd}, we give a proof of Theorem~\ref{thm:theOddCase} on the rigidity of solutions under weak limits when the multiplier is odd.  Section~\ref{sec:hamiltonianOdd} is then devoted to the proof of Theorem~\ref{thm:odd:hamiltonian} on the conservation of the Hamiltonian for active scalars with odd multipliers in dimension $2$.

The last Section~\ref{sec:constraints} is devoted to proving Theorem~\ref{thm:more:constraints}, which shows that the nondegeneracy condition in Theorem~\ref{thm:nDims} is necessary in general for the weak limit statement of Theorem~\ref{thm:mainThm} to apply in higher dimensions.  In Section~\ref{sec:conclusions} we give a conclusion to the paper and state some open questions.


\subsection{Notation}

We use the Einstein summation convention of summing over indices which are repeated.  We take the convention that vectors are written with upper indices, whereas covectors are written with lower indices; thus, for a vector field $u^l$ and function $\xi$, we write $u \cdot \nab \xi = u^l \pr_l \xi$ and $\tx{div } u = \pr_l u^l$.  

We use the notation $X \unlhd Y$ to indicate an inequalities $X \leq Y$ which have not been proven, but will be proven later on in the course of the argument.  We sometimes refer to such inequalities as ``goals''.  


%
%
%
%

\section{Basic Technical Outline} \label{sec:basicTechOutline}
In this Section, we give a technical outline of the main ideas of the construction which includes a list of the important error terms and provides a comparison to the cases of the Euler and isometric embedding equations.  This section provides the basic ideas to motivate the statement of the Main Lemma of Section~\ref{sec:mainLem}.

We will perform the construction in a space of approximate solutions to the active scalar equation which we now define.

We say that $(\th, u^l, R^l)$ satisfy the {\bf scalar-stress} equation if
\ali{
\label{eq:scalStress}
\left \{
\begin{aligned}
\pr_t \th + \pr_l(\th u^l) &= \pr_l R^l  \\
u^l &= T^l(\th) 
\end{aligned}
\right.
}
This system is the analogue for active scalar equations of the Euler-Reynolds system introduced in \cite{deLSzeCts} for the Euler equations.  Here $R^l$ is a vector field on $\T^2$ that we call the ``stress field'' (by analogy with the stress tensor $R^{jl}$ in the Euler-Reynolds equations) which measures the error by which $\th$ fails to solve the active scalar equation.

Recall that the operator 
\[ T^l[\th] = \int_{\R^2} K^l(h) \th(x-h) dh  \]
is a convolution operator with a real-valued kernel $K^l$ which is homogenous of degree $-2$ as a distribution.  The corresponding Fourier multiplier 
\ali{
m^l(\xi) &= {\hat K}^l(\xi) 
}
is homogeneous of degree $0$, satisfies $m^l(-\xi) = \overline{m^l(\xi)}$, and we assume that $m^l(\xi)$ is smooth on $|\xi| = 1$ (and therefore smooth away from the origin).  To ensure that  $u^l = T^l[\th]$ satisfies the divergence free condition $\pr_l u^l = 0$, we require that 
\ali{
 m(\xi)\cdot \xi &= 0
}
At a high level, the basic idea of the convex integration construction is to start with a given solution $(\th, u^l, R^l)$ to \eqref{eq:scalStress}, and proceed to add a (high-frequency) correction $\Th$ to the scalar field $\th$, so that the corrected scalar field and drift velocity 
\ali{
\th_1 = \th + \Th, \qquad u_1^l = u^l + U^l, \qquad U^l = T^l[\Th]
} 
satisfy the scalar stress equation \eqref{eq:scalStress} with a new stress field $R_1^l$ that is significantly smaller than the original stress field $R^l$.  These corrections are added in an iteration to obtain a sequence of solutions to \eqref{eq:scalStress}
\[ (\th_{(k)}, u_{(k)}^l, R_{(k)}^l) \]
such that $R_{(k)}^l \to 0$ as the number of iterations $k$ tends to infinity.  From dimensional analysis and experience with the isometric embedding and Euler equations, we expect an estimate $\| \Th_{(k)} \|_{C^0} \leq C \|R_{(k)} \|_{C^0}^{1/2}$ for the size of the corrections, so that we will obtain continuous solutions in the limit provided $\|R_{(k)} \|_{C^0}$ tends to $0$ at a reasonable rate\footnote{ In our case, the error will converge to zero exponentially fast: $ \| R_{(k)} \|_{C^0} \leq C_1 e^{- C_2 k}$ for some constants $C_1, C_2 > 0$.}.  On the other hand, the $C^1$ norms of the corrections $\| \nab \Th_{(k)} \|_{C^0}$ will diverge as the frequencies in the iteration grow to infinity, and we prove convergence of the iteration in H\"{o}lder spaces by interpolating between the bounds for $\| \Th_{(k)} \|_{C^0}$ and $\| \nab \Th_{(k)} \|_{C^0}$ after the construction has been optimized to reduce the stress field $\| R_{(k)} \|_{C^0}$ at the most efficient rate possible.  Although this description explains how the scheme works at a high level, we must study the equation and the scheme in much more detail before it is clear that there is any hope of reducing the stress field $R^l$ in this manner.

As in \cite{isett}, we will consider corrections built from rapidly oscillating ``plane waves'' where we allow for phase functions $\xi_I$ and amplitudes $\th_I$ which depend on space and time
\ali{
\Th &= \sum_I \Th_I \label{eq:correctionForm} \\
\Th_I &= e^{i \la \xi_I} (\th_I + \de \th_I) \label{eq:correctionsLookLike1}
}
The amplitude $\th_I$ and the phase functions $\xi_I$ are scalar functions of our choice, which vary slowly compared to the frequency parameter $\la$.  The term $\de \th_I$ is a small correction term which will be made precise later.  Each wave $\Th_I$ has a conjugate wave $\Th_{\bar{I}} = \overline{\Th}_I$ with opposite phase function $\xi_{\bar{I}} = - \xi_I$ and amplitude $\th_{\bar{I}} = \bar{\th}_I$ so that the overall correction is real valued.

We now proceed to calculate the equation satisfied by the corrected scalar field $\th_1 = \th + \Th$.  This requires us to calculate the new drift velocity $u_1^l = T^l[\th_1] = u^l + U^l$, where $U^l = T^l[\Th]$.  Our main tool for this calculation is a Microlocal Lemma, which in this case guarantees that each wave $\Th_I$ gives rise to a velocity field
\ali{
U_I^l = T^l[\Th_I] &= e^{i \la \xi_I} (u_I^l + \de u_I^l) \label{eq:microlocalLemmaInAction} \\
 u_I^l &= m^l(\nab \xi) \th_I \label{eq:microlocalLemMultiplier}
}
with amplitude determined by the Fourier multiplier $m^l(\xi)$ in the definition of $T^l$.  

The amplitude $u_I^l$ thus has the size comparable to $\th_I$, while the term $\de u_I$ is a small correction of the same order as $\de \th_I$.  Thus, given a highly oscillatory input such as $\Th_I = e^{i \la \xi_I} \th_I$, the operator $T^l$ behaves to leading order like a multiplication operator on the amplitude.  (For our purposes, the simplest way to achieve equation \eqref{eq:microlocalLemmaInAction} will be to use phase functions defined on the whole torus $\T^2$, but this will not be a serious restriction.)

From the Ansatz \eqref{eq:correctionForm} and equation \eqref{eq:scalStress}, we see that the corrected scalar field $\th_1 = \th + \Th$ satisfies the equation
\ali{
\pr_t \th_1 + \pr_l(u_1^l \th_1) &= \pr_t \Th + \pr_l( u^l \Th) + \pr_l(U^l \th) + \pr_l( U^l \Th + R^l)  \label{eq:newTh1}
}
We now expand $\Th$ and $U$ into individual waves using \eqref{eq:correctionForm} to derive
\ali{
\pr_t \th_1 + \pr_l(u_1^l \th_1) =& \pr_t \Th + \pr_l( u^l \Th) + \pr_l(U^l \th) \label{eq:transportTermAppears} \\
&+ \sum_{J \neq \bar{I}} \pr_l(U_J^l \Th_I) + \pr_l( \sum_I U_I^l \Th_{\bar{I}} + R^l )  \label{eq:highAndStressTerm}
}
Our goal is to design the correction $\Th$ so that the forcing terms on the right hand side of \eqref{eq:transportTermAppears}-\eqref{eq:highAndStressTerm} can be represented in divergence form $\pr_l R_1^l$ for a vector field $R_1^l$ which is significantly smaller in $C^0$ than the previous error $R^l$.

\subsection{The Stress Term } \label{sec:stressTerm}
Our first goal is to cancel out the term $R^l$ appearing in the rightmost term of \eqref{eq:highAndStressTerm}, which is the only term in equations \eqref{eq:transportTermAppears}-\eqref{eq:highAndStressTerm} that has low frequency.  We expand this term using \eqref{eq:microlocalLemmaInAction}-\eqref{eq:microlocalLemMultiplier} as
\ali{
 \sum_I U_I^l \Th_{\bar{I}} + R^l &= \fr{1}{2} \sum_I(U_I^l \Th_{\bar{I}} + U_{\bar{I}}^l \Th_I) + R^l \notag \\
&\approx \fr{1}{2} \sum_I ( u_I^l \th_{\bar{I}} + u_{\bar{I}}^l \th_I ) + R^l \notag \\
&= \fr{1}{2} \sum_I |\th_I|^2 (m^l(\nab \xi_I) + m^l(- \nab \xi_I)) + R^l \label{eq:theAlgebraAppears} \\
&= \fr{1}{2} \sum_I |\th_I|^2 (m^l(\nab \xi_I) + \overline{m^l(\nab \xi_I)} ) + R^l 
}
where the error terms are lower order, involving $\de \th_I$ and $\de u_I^l$.  Here we can see already why we are restricted to multipliers $m^l( \cdot )$ which are not odd.  Namely, for an odd multiplier $m^l(-\xi) = -m^l(\xi)$, the high frequency interactions fail to leave a nontrivial low frequency part.  In other words, the obstruction is that we lack a high-low frequency cascade.  

We therefore assume now that the multiplier $m^l$ is not odd.  Together with the divergence free property $\xi_l m^l(\xi) = 0$ and the degree zero homogeneity of the symbol $m^l(\cdot)$, this condition implies that there are linearly independent vectors in the image of the even part of the multiplier
\ali{
 A^l = m^l(\xi^{(1)}) + m^l(-\xi^{(1)}), \qquad B^l = m^l(\xi^{(2)}) + m^l(-\xi^{(2)})  \label{eq:ourAlandBl}
}
where $\xi^{(1)}, \xi^{(2)} \in \Z^2 = {\widehat \T}^2$ are nonzero frequencies with integer entries.  

At this point, since we now have two vectors $A^l$ and $B^l$ in the image of the even part of $m^l$ that are linearly independent, there is some hope to get the terms in \eqref{eq:theAlgebraAppears} to cancel out.  Namely, one should first make sure that the phase gradients $\nab \xi_I$ are perturbations of the directions $\xi^{(1)}, \xi^{(2)}$ so that each wave yields a velocity field taking values in the direction $(m^l(\nab \xi_I) + m^l(- \nab \xi_I) \approx A^l$ or in the direction $(m^l(\nab \xi_I) + m^l(- \nab \xi_I) \approx B^l$.  One would then like to choose coefficients $\th_I$ so that terms $|\th_I|^2 (m^l(\nab \xi_I) + m^l(- \nab \xi_I))$ in \eqref{eq:theAlgebraAppears} form the appropriate linear combinations of $A^l$ and $B^l$ needed to cancel out $R^l$.  

However, there is an immediate difficulty in implementing the above approach.  Namely, although we know that $A^l$ and $B^l$ are linearly independent, it may not be case that $R^l$ can be written as a linear combination of $A^l$ and $B^l$ with {\bf non-negative} coefficients $|\th_I|^2$.  To get around this difficulty, we take advantage of a degree of freedom which already played an important role in the arguments of \cite{corFarGanPor} and \cite{shvConvInt}.  Namely, observe that we do not need to solve the equation \eqref{eq:theAlgebraAppears} exactly, but need only ensure that \eqref{eq:theAlgebraAppears} is divergence free.  This freedom allows us to subtract from \eqref{eq:highAndStressTerm} any vector field $e(t) \de^l$ which is constant in space and depends only on time.  Therefore the equation we actually solve is more similar to
\ali{
\fr{1}{2} \sum_I |\th_I|^2 (m^l(\nab \xi_I) + \overline{m^l(\nab \xi_I)} ) &= e(t) \de^l - R_\ep^l \label{eq:shiftedAlgebraEquation}
}
where $R_\ep^l$ is a regularized version of $R^l$ and $\de^l$ is a constant vector field.  If we choose $\de^l = A^l + B^l$ and make sure that $e(t)$ is bounded below by, say, $e(t) \geq 100 \| R_\ep \|_{C^0}$ on the support of $R_\ep$, then the coefficients $|\th_I|^2$ solving \eqref{eq:shiftedAlgebraEquation} can be guaranteed to be non-negative.  Observe also that the equation \eqref{eq:shiftedAlgebraEquation} leads to the bounds $\|\th_I\|_{C^0} \leq C \| R_\ep \|_{C^0}^{1/2}$ for the amplitudes.

The role played by the function $e(t) \de^l$ is the same as the role played by the low frequency part of the pressure correction in the scheme for Euler~\cite[Section 7.3]{isett}.  This device in some way appears to limit our proof to the periodic setting.


\subsection{The High Frequency Interference Terms} \label{sec:highFreqInterference}
Controlling the interference terms between high frequency waves is a fundamental difficulty in convex integration.  In our case, the interference terms require solving the elliptic equation
\ali{
\pr_l R_H^l &= \sum_{J \neq \bar{I}}  \pr_l(U_J^l \Th_I) = \sum_{J \neq \bar{I}}  U_J^l \pr_l\Th_I   \\
&= \fr{1}{2} \sum_{J \neq \bar{I}} (U_J^l \pr_l \Th_I + U_I^l \pr_l \Th_J) 
}
To leading order, these terms have the form
\ali{
\pr_l R_H^l &= \fr{1}{2} (i \la) \sum_{J \neq \bar{I}} e^{i \la(\xi_I + \xi_J)}(u_J^l \pr_l \xi_I \th_I + u_I^l \pr_l \xi_J \th_J) + \ldots \label{eq:theHighHighTerms} \\
&= \fr{1}{2} (i \la) \sum_{J \neq \bar{I}} e^{i \la(\xi_I + \xi_J)}\th_I \th_J (m^l(\nab \xi_J) \pr_l \xi_I + m^l(\nab \xi_I) \pr_l \xi_J) + \ldots \label{eq:theHighHighTerms1}
}
We expect to a gain a factor of $\la^{-1}$ while inverting the divergence in \eqref{eq:theHighHighTerms}; however, solving \eqref{eq:theHighHighTerms} leads in principle to a solution $R_H^l$ of size $\co{R_H} \leq  \co{\sum_I |\th_I|^2} \leq \co{ R }$, which is not even an improvement on the size of the previous error $R^l$.  These terms therefore seem to already prohibit the construction of continuous solutions by convex integration.  The same difficulty also arises for the Euler equations.

For the Euler equations, the key idea introduced in \cite{deLSzeCts} which made it possible to handle high frequency interference terms similar to \eqref{eq:theHighHighTerms} was to construct the high frequency building blocks using a family of stationary solutions to the Euler equations known as Beltrami flows.  Specifically, the basic building blocks in the construction \cite{deLSzeCts} are constructed using vector fields of the form $B^l e^{i k \cdot x}$ where $B^l$ is a constant vector amplitude, $k \cdot x$ is a linear phase function, and we have $(i k )\times B^l = |k| B^l$ so that the experession $B^l e^{i k \cdot x}$ is an eigenfunction of curl and hence a stationary solution to Euler.  The idea of using Beltrami flows was adapted in \cite{isett} to building blocks $V_I = e^{i \la \xi_I} v_I$ with nonlinear phase functions $\xi_I$ by imposing a ``microlocal Beltrami flow'' condition that $(i \nab \xi_I) \times v_I = |\nab \xi_I| v_I$ pointwise.  Viewed from this latter approach, the role of the Beltrami flow condition is to ensure that the leading term in \eqref{eq:theHighHighTerms1} cancels out.

For the active scalar equations we consider here, such a family of stationary solutions is not available, and moreover we do not have any method to control interference terms between waves which oscillate in distinct directions.  For instance, suppose that the multiplier $m^l(\xi)$ is even, and suppose that $\xi_1, \xi_2 \in {\hat \R}^2$ are linearly independent frequencies for which the terms in \eqref{eq:theHighHighTerms1} cancel
\[ m(\xi_1) \cdot \xi_2 \pm m(\xi_2) \cdot \xi_1 = 0 \]
It then follows from the conditions $m(\xi_1) \cdot \xi_1 = 0, m(\xi_2) \cdot \xi_2 = 0$ that both $m(\xi_1)$ and $m(\xi_2)$ must be equal to $0$.  More generally, one can show that the even part of the multiplier must vanish when applied to both frequencies
\[ m^l(\xi_1) + m^l(-\xi_1) = m^l(\xi_2) + m^l(-\xi_2) = 0 \]
if we assume that all of the interference terms in \eqref{eq:theHighHighTerms1} cancel.  This vanishing of the even part would prohibit any nontrivial contribution to \eqref{eq:theAlgebraAppears}.  In contrast, in the case of the surface quasigeostraphic equation where the drift velocity is given by $u = \nab^\perp (- \De)^{-1/2} \th$, the set of Laplace eigenfunctions provides a large family of high frequency, stationary solutions.  However, in this case the multiplier $m(\xi) = i \langle-\xi_2,\xi_1\rangle |\xi|^{-1}$ is odd and we have already seen that such multipliers are out of reach of our method.

Our main observation which allows us to handle these terms is the fact that the interference terms which arise when an individual wave interacts with {\bf itself} always vanish to leading order from the structure of the equations.  Namely, if we look at a single index $J = I$, then from the divergence free condition for the symbol $m(\xi) \cdot \xi = 0$ we see that the leading term in \eqref{eq:theHighHighTerms1} gives no contribution
\[  (m^l(\nab \xi_I) \pr_l \xi_I + m^l(\nab \xi_I) \pr_l \xi_I) = 0 \]
Therefore, while we lack a method to control interference terms between waves which oscillate in different directions, we can still pursue an approach where in each step of the iteration we use corrections $\Th$ containing waves waves which oscillate in only a single direction and thus do not interfere with each other.  

\subsubsection{Comparison with the Euler and Isometric Embedding equations} \label{sec:compareEulerIso}

In this Section, we remark on how our observation also gives a new approach to building weak solutions to the Euler equations which is independent of Beltrami flows, and explain why we expect a loss of regularity by comparing to analogous considerations in the case of the isometric embedding equations.

Our observation of vanishing self-interference terms applies in the case of the Euler equations as well.  For the Euler equations, an individual wave is a velocity field which takes the form $V_I = e^{i \la \xi_I}( v_I + \de v_I )$, and we require that the amplitude takes values in $v_I \in \langle \nab \xi_I \rangle^\perp$ in order to ensure the divergence free condition for $V_I$.  In this case, the high frequency interference terms between an individual wave and itself have the form
\ali{
 V_I^j \pr_j V_I^l &= (i \la) e^{2 i \la \xi_I} v_I^j \pr_j \xi_I v_I^l + \tx{ lower order terms }  \label{eq:noSelfTermsEuler}
}
Observe that the requirement $v_I \cdot \nab \xi_I = 0$ forces the the main contribution to cancel.  Thus, the method we apply here in principle generalizes to give a new approach to producing H\"{o}lder continuous weak solutions to the Euler equations which entirely avoids the use of Beltrami flows and applies in arbitrary dimensions.  Our observation appears to be quite natural in that the key cancellation we exploit comes immediately from the structure of the equations themselves without imposing any particular Ansatz in the construction.  On the other hand, in contrast to the use of Beltrami flows for Euler, we are restricted here to removing one component of the error at a time during the iteration, which ultimately results in a loss of regularity in the solutions obtained from the construction.

The reason we expect to lose regularity from the restriction of removing one component of the error each stage comes from experience with the isometric embedding equations from the work of Conti, De Lellis and Sz\'{e}kelyhidi \cite{deLSzeC1iso}.  For these equations, there is currently no method available for controlling the relevant interference terms between high frequency waves for embeddings of codimension $1$, and this obstruction leads to a loss of regularity for the solutions obtained through convex integration.  Namely, without a method to control interference terms between distinct waves, it is only possible to eliminate a single, rank one component of the metric error in each step of the iteration from the addition of a single wave.  Consequently, it is necessary to increase the frequencies of the waves multiple times before any $C^0$ improvement in the metric error can be realized, which leads to a loss of regularity.  
In contrast, the use of Beltrami flows for the Euler equations allows for the addition of waves which oscillate at the same frequency level in several different directions, and the stress error can be made smaller in $C^0$ after only one step of the iteration.  Since our scheme suffers from the same deficiency as in the case of isometric embeddings (that is, we cannot use waves at equal frequency levels which oscillate in multiple directions), it turns out that our scheme is limited to a H\"{o}lder exponent which is inferior to the exponent $1/5$ achieved for the Euler equations.

The restriction to eliminating a single component of the error in each step of the iteration also threatens to make our proof considerably more complicated than the scheme used for Euler.  While we are unable to avoid the loss of regularity, we are at least able to keep the overall complexity of the argument to be essentially no more complicated than the scheme used for Euler.  This simplification is accomplished by introducing a new technique, which we explain in the following Section.



\subsection{Reducing the Steps in the Iteration} \label{sec:reduceSteps}

From the discussion in Section \ref{sec:compareEulerIso}, we can now consider a {\bf serial convex integration scheme} wherein we cannot reduce the size of the error term $R^l$ until we have added a series of two corrections
\ali{
\th_1 &= \th + \Th_{(1)} + \Th_{(2)} \label{eq:need2Corrections}
}
Following the original scheme of Nash \cite{nashC1} in the isometric embedding problem, we should first decompose $R^l$ into components as
\[ R^l = c_A A^l + c_B B^l \]
where $A^l$ and $B^l$ are linearly independent vectors in the image of $m^l(\xi) + m^l(-\xi)$ defined in \eqref{eq:ourAlandBl}.  The first correction $\Th_{(1)}$ to $\th$ should oscillate in the $\xi^{(1)}$ direction in order to eliminate the $A^l$ component of the error $R^l$ by the method described in Section \ref{sec:stressTerm}.  Then, the second correction $\Th_{(2)}$ should have an even larger frequency than $\Th_{(1)}$, but the same amplitude $| \Th_{(1)}| \sim |\Th_{(2)}| \sim |R|^{1/2}$, since its purpose is to eliminate the $B^l$ component of the error $R^l$.  Thus, one stage of the convex integration is completed after two steps, where each step involves eliminating one component of the error, and the error $R^l$ is smaller in $C^0$ only at the end of the stage.

It appears that such a serial convex integration scheme should be possible for active scalar equations and should lead to the same H\"{o}lder exponent $1/9$ that we achieve here.  On the other hand, such a serial proof seems to be somewhat complicated compared to the ``one-step'' scheme used for Euler or to the case of the isometric embedding equations.  In our case, a serial proof would involve treating a larger number of error terms having unfamiliar estimates, and optimizing a larger number of time, frequency and length scale parameters.  We avoid these additional complexities by making a simple observation that allows us to reduce the $C^0$ norm of the error in a single step of the iteration rather than several.  It turns out that this idea also causes most of the terms in the construction to obey estimates which are familiar from experience with the Euler equations, amounting to an overall more transparent proof.


Our observation which allows us to reduce the error in every stage of the iteration and thereby simplify our proof is the following.  First, note that the addition of the first correction $\Th_{(1)}$ results in a remaining error $R_{(1)}^l$ of the form
\ali{
R_{(1)}^l &= c_B B^l + R_{E}^l \label{eq:compScalStressBAppears}
}
where $R_{E}^l$ is much smaller than the original error $R^l$, whereas the term $|c_B B^l| \sim |R^l|$ has the same size.  Rather than using the second correction $\Th_{(2)}$ to eliminate the term $c_B B^l$ as discussed previously, we observe that we can simultaneously get rid of the $B^l$ component of the small term $R_{E}^l$, thus leaving an error of the form
\ali{
R_{1}^l &= c_A A^l + R_J^l \label{eq:compScalStressAppears}
}
where $c_A A^l$ is the remaining $A^l$ component of $R_E^l$, and the term $R_J^l$ is an even smaller error term.  For our next correction, we can repeat the same idea and eliminate the $A^l$ component of \eqref{eq:compScalStressAppears}, leaving an error of the form \eqref{eq:compScalStressBAppears}.  Continuing in this way, we see that each correction now causes an improvement in the size of the error in the $C^0$ topology, just as in the situation for Euler.  


The above discussion has been based on the hope that we can really eliminate the $A^l$ and $B^l$ components of the error, which is not entirely justified at this point.  In fact, there are some further difficulties which stand in our way before this task can be accomplished which will become more clear as we specify the construction.  One such difficulty is the appearance of low frequency interference terms.

\subsection{Low Frequency Interference Terms}

It turns out that the most straightforward approach to the construction based on the ideas Section \ref{sec:highFreqInterference} gives rise to certain interference terms of low to intermediate frequency which apparently prohibit the success of our scheme.  Thus, while the idea introduced in Section \ref{sec:highFreqInterference} allows us to control the high frequency interference terms in a sactifactory manner, we must incorporate one additional idea into the construction before our scheme can handle every type of error term which arises.  

The ideas in Section \ref{sec:highFreqInterference} suggest that a natural approach to the construction is to use waves of the form $\Th_I = e^{i \la \xi_I}(\th_I + \de \th_I)$ where the phase functions $\xi_I$ oscillate in the direction $\pm \xi^{(1)}$ (or $\pm \xi^{(2)}$) in the sense that the gradients remain close to their common initial values
\ali{
\nab \xi_I \approx \pm \xi^{(1)} \label{eq:plusMinusBadIdea}
}
For an index $I$, let us write $f(I) \in \{ \pm \}$ to denote the sign appearing in \eqref{eq:plusMinusBadIdea}. 

According to Section \ref{sec:highFreqInterference}, we have a method to ensure that high frequency nonlinear interference terms obey good bounds.  Thus, every interaction term of the form
\ali{
 \pr_l( \Th_I U_J^l + \Th_J U_I^l) \label{eq:aHighFreqErrorExmpl}
}
which arises between waves of the same sign $f(I) = f(J) \in \{ \pm \}$ can be handled by our method, as these terms are all of high frequency.

A new difficulty arises when we consider interference terms between waves of opposite signs $f(I) = - f(J)$, which we call ``Low-Frequency Interference Terms''.  In this case, the terms of the form $\Th_I U_J^l + \Th_J U_I^l$ as in \eqref{eq:aHighFreqErrorExmpl} can be expressed to leading order as
\ali{
\Th_I U_J^l + \Th_J U_I^l &\approx e^{i \la ( \xi_I + \xi_J) } ( \th_I u_J^l + \th_J u_I^l) \label{eq:LowFreqIntTerms}
}
When we consider indices with opposite signs $f(I) = - f(J)$, the term \eqref{eq:LowFreqIntTerms} cannot be viewed as a high frequency error term.  In the worst case it may even be true that $\nab (\xi_I + \xi_J) = 0$ thanks to the initial conditions satisfying \eqref{eq:plusMinusBadIdea}.

It turns out that having low frequency interference terms of the form \eqref{eq:LowFreqIntTerms} prevents us from solving the quadratic equation to determine the amplitudes $\th_I$.  To see this difficulty, note that the left hand side of the equation analogous to \eqref{eq:shiftedAlgebraEquation}, which includes all low frequency interactions, would have to include terms of the form
\ali{
\sum_{\substack{ I,J \\ f(I) = +, ~f(J) = -}} \Th_I U_J^l + \Th_J U_I^l &=  \sum_{\substack{I, J \\f(I) = f(J) = +}} e^{i \la ( \xi_I - \xi_J) } (\th_I \bar{\th}_J + \th_J \bar{\th}_I) A^l + \ldots \label{eq:weTryLowFreqIntTerms}
}
Remarkably, the right hand side of \eqref{eq:weTryLowFreqIntTerms} appears to obey all the estimates we would require for obtaining solutions with H\"{o}lder regularity $1/9-$, despite the appearance of the parameter $\la$.  The problem is that the right hand side of \eqref{eq:weTryLowFreqIntTerms} must remain bounded from $0$ in order to solve the quadratic equation for the amplitudes.  On the other hand, there is no way to preclude the possibility that the series \eqref{eq:weTryLowFreqIntTerms} cancels completely at points $(t,x)$ on which the amplitudes $\th_I(t,x)$ and $\th_J(t,x)$ have essentially the same size, due to the presence of the oscillating factors $e^{i \la ( \xi_I - \xi_J) }$ in the cross terms arising from distinct indices $J \neq I$.


At first sight, this difficulty would seem to completely prevent us even from achieving continuous solutions, as we are left with no way to obtain a $C^0$ improvement in the size of the error on the regions where distinct indices interact.  We overcome this obstruction by making one more adjustment to the construction.  Roughly speaking, our idea is to allow the condition \eqref{eq:plusMinusBadIdea} to be satisfied by ``half'' the waves in our construction, whereas the other ``half'' of the waves in the construction involve phase functions with initial data satisfying
\ali{
\nab \xi_I \approx \pm 10 \xi^{(1)} \label{eq:theTenTrick}
}
Furthermore, we ensure that every nonlinear interaction which takes place between nonconjugate waves involves one wave satisfying \eqref{eq:plusMinusBadIdea}, and a second wave satisfying \eqref{eq:theTenTrick}.  In this way, every interference term of the form \eqref{eq:LowFreqIntTerms} is actually a high frequency error term. 
Moreover, every wave oscillates in a direction essentially parallel to $\xi^{(1)}$, so that the idea of Section \ref{sec:highFreqInterference} still applies to treat these high frequency interference terms.

With these ideas in hand, we are now ready to proceed with the formal construction in detail, beginning with the statement of the Main Lemma.

\section{The Main Lemma} \label{sec:mainLem}

In order to state the main lemma, let us recall that we have fixed once and for all a choice of linearly independent vectors
\ali{
 A^l = m^l(\xi^{(1)}) + m^l(-\xi^{(1)}), \qquad B^l = m^l(\xi^{(2)}) + m^l(-\xi^{(2)}) \label{eq:AlBlare}
}
where $\xi^{(1)}, \xi^{(2)} \in \Z^2 = {\widehat \T}^2$ are nonzero (integral) frequencies.  The existence of these vectors is guaranteed by the condition that $m^l(\xi)$ is not odd, and the orthogonality condition $\xi_l m^l(\xi) = 0$.

\begin{defn}  For a constant vector $A^l$, we say that $(\th, u^l, c_A, R_J^l)$ satisfy the {\bf Compound Scalar-Stress} equation (with vector $A^l$) if 
\ali{
\label{eq:compScalStress}
\left \{
\begin{aligned}
\pr_t \th + \pr_l(\th u^l) &= \pr_l( c_A A^l + R_J^l)  \\
u^l &= T^l(\th) 
\end{aligned}
\right.
}
In this case, we will refer to the tuple $(\th, u^l, c_A, R_J^l)$ as a {\bf compound scalar-stress field}.
\end{defn}

For a solution to the compound scalar-stress equation \eqref{eq:scalStress}, we define compound frequency-energy levels to be the following
\begin{defn} \label{def:freqEnDef}
Let $L \geq 1$ be a fixed integer.  Let $\Xi \geq 2$, and let $e_v$, $e_R$ and $e_J$ be positive numbers with $e_J \leq e_R \leq e_v$.  We say that $(\th, u^l, c_A, R^l)$ have frequency and energy levels below $(\Xi, e_v, e_R, e_J)$ to order $L$ in $C^0$ if $(\th, u^l, c_A, R^l)$ solve the system \eqref{eq:compScalStress} and satisfy the bounds 
\begin{align}
|| \nab^k u ||_{C^0} + \co{ \nab^k \th } &\leq \Xi^k e_v^{1/2} &k = 1, \ldots, L \label{bound:nabkth} \\
\| \nab^k (\pr_t + u \cdot \nab) u \|_{C^0} &\leq \Xi^{k+1} e_v &k = 0, \ldots, L-1 \label{bound:nabkdtu} \\
|| \nab^k c_A ||_{C^0} &\leq \Xi^k e_R &k = 0, \ldots, L \label{bound:nabkCA} \\
|| \nab^k (\pr_t + u \cdot \nab) c_A ||_{C^0} &\leq \Xi^{k+1} e_v^{1/2} e_R  &k = 0, \ldots, L - 1 \label{bound:dtnabkCA} \\
|| \nab^k R_J ||_{C^0} &\leq \Xi^k e_J &k = 0, \ldots, L \label{bound:nabkR} \\
|| \nab^k (\pr_t + u \cdot \nab) R_J ||_{C^0} &\leq \Xi^{k+1} e_v^{1/2} e_J &k = 0, \ldots, L - 1 \label{bound:dtnabkR}
\end{align}
Here $\nab$ refers only to derivatives in the spatial variables.
\end{defn}
Note that we assume bounds \eqref{bound:nabkth}-\eqref{bound:nabkdtu} on the drift velocity $u^l$ which do not in general follow from the corresponding bounds on $(\th, c_A, R^l)$ and the transport equation \eqref{eq:compScalStress}.  We assume these bounds on $u^l$ in order to to avoid logarithmic losses in our estimates which would arise otherwise from the lack of $C^0$ boundedness of the operator $u^l = T^l \th$ defining the velocity.

We now state the Main Lemma of the paper, which summarizes the result of one step of the convex integration procedure.  The statement of this lemma involves two constants: $K_0 \geq 1$ (specified in Line \eqref{eq:whichK0weneed} of the construction) and $K_1 \geq 1$ (determined in Line \eqref{eq:decomposeRepComponents} of the construction, see also Section~\ref{sec:coBounds}).  These constants $K_0$ and $K_1$ depend only on the operator $T^l$ in the statement of the Main Theorem.

\begin{lem}[The Main Lemma]\label{lem:mainLemma}
Suppose that $L \geq 2$ and let $K, M \geq 4$ be non-negative numbers such that $K \geq K_0$.  
There is a constant $C_0$ depending only on $L$, $K$, $M$ and the operator $T^l$ such that the following holds:

Let $(\th,u^l, c_A, R_J^l)$ be any solution of the compound scalar-stress system whose compound frequency and energy levels are below $(\Xi, e_v, e_R, e_J)$
 to order $L$ in $C^0$, and let $I \subseteq \R$ be a nonempty closed interval such that 
\ali{
 \supp R_J \cup \supp c_A &\subseteq I \times \T^2  \label{sub:suppAssumption}
}

Define the time-scale ${\hat \tau} = \Xi^{-1} e_v^{-1/2}$, and let
\[ e(t) : \R \to \R_{\geq 0} \]
be any non-negative function for which the lower bound
\begin{align}
 e(t) \geq K e_R \quad \quad \mbox{ for all } t \in I \pm {\hat \tau}  \label{eq:lowBoundEoftx}
\end{align}
is satisfied in a ${\hat \tau}$-neighborhood of the interval $I$, and whose square root satisfies the estimates
\begin{align}
|| \fr{d^r}{dt^r} e^{1/2} ||_{C^0} &\leq M (\Xi e_v^{1/2})^r e_R^{1/2} ,\qquad 0 \leq r \leq 2 \label{ineq:goodEnergy}
\end{align}

Now let $N$ be any positive number obeying the bound
\begin{align} 
 N &\geq \left(\fr{e_v}{e_R} \right)^{3/2}\label{eq:conditionsOnN2}
\end{align}
and define the dimensionless parameter ${\bf b} = \left(\fr{e_v^{1/2}}{e_R^{1/2}N} \right)^{1/2}$.

Then there exists a solution $(\th_1,u^l, c_B, R_1^l)$ of the form $\th_1 = \th + \Th$, $u_1 = u + U$ to the Compound Scalar-Stress Equation \eqref{eq:compScalStress} with vector $B^l$ whose frequency and energy levels are below
\begin{align}
\begin{split}
 (\Xi', e_{v}', e_{R}', e_J') &= \left( C N \Xi, e_R, K_1 e_J, \fr{e_v^{1/4} e_R^{3/4}}{N^{1/2}}   \right) \\
 &= \left( C N \Xi, e_R, K_1 e_J, \left(\fr{e_v^{1/2}}{e_R^{1/2}N} \right)^{1/2} e_R   \right) \\
&= \left( C N \Xi, e_R, K_1 e_J, {\bf b}^{-1} \fr{e_v^{1/2} e_R^{1/2}}{N}\right) 
\end{split} \label{eq:theNewEnergyLevel}
\end{align}
to order $L$ in $C^0$, and whose stress fields $R_1$ and $c_B$ are supported in
\begin{align}
\supp c_B \cup \supp R_1 \subseteq \supp e \times \T^2
\end{align}
The correction $\Th = \th_1 - \th$ is of the form $\Th = \nab \cdot W$.  This correction and the correction to the velocity field $U^l = T^l[\Th]$ can be guaranteed to obey the bounds
\begin{align}
||\Th||_{C^0} + \|U^l \|_{C^0} +  &\leq C e_R^{1/2} \label{eq:Vco} \\
\|\nab \Th \|_{C^0} + \|\nab U \|_{C^0} &\leq C N \Xi e_R^{1/2} \label{eq:nabVco} \\
\|(\pr_t + u^j \pr_j) \Th \|_{C^0} + \|(\pr_t + u^j \pr_j) U \|_{C^0} &\leq C {\bf b}^{-1} \Xi e_v^{1/2} e_R^{1/2} \label{eq:matVco} \\
||W||_{C^0} &\leq C \Xi^{-1} N^{-1} e_R^{1/2} \label{eq:Wco}\\
\co{ \nab W } &\leq C e_R^{1/2} \label{eq:nabWco} \\
||(\pr_t + u^j \pr_j) W ||_{C^0} &\leq C {\bf b}^{-1} N^{-1} e_v^{1/2} e_R^{1/2} \label{eq:matWco}
\end{align}
The energy increment from the correction is prescribed up to errors bounded by
\begin{align}
\left| \int_{\T^2} \fr{|\Th|^2}{2}(t,x) dx - \int_{\T^2} e(t) dx \right| &\leq \fr{1}{2} \int_{\T^2} e(t) dx + \fr{e_R}{N}  \label{eq:energyPrescribed}
\end{align}
and the incremental energy variation satisfies an estimate
\ali{
\left| \fr{d}{dt} \int_{\T^2} |\Th|^2(t,x) dx \right| \leq C {\bf b}^{-1} \Xi e_v^{1/2} e_R \label{eq:dtenergyPrescribed}
}
uniformly in time.  Finally, the space-time support of the correction $\Th$ is contained in $\supp e \times \T^2$.
\end{lem}

\subsection{Remarks about the Main Lemma}
The overall structure of Lemma~\ref{lem:mainLemma} is based on the Main Lemma of \cite[Lemma 10.1]{isett}.  The most important difference in our Lemma lies in the difference in the definition of the compound frequency energy levels.  The bounds implicit in \eqref{eq:theNewEnergyLevel}, which state the rate at which we are able to reduce the stress error, are the most essential point the main lemma and dictate the regularity of the solutions we obtain.  Another noticeable difference between Lemma~\ref{lem:mainLemma} compared to the Lemmas \cite[Lemma 10.1]{isett} and \cite[Lemma 4.1]{isettOhAng} is that the estimate \eqref{eq:energyPrescribed} gives us worse control over the increment of energy.  In those Lemmas, the term $\fr{1}{2} \int_{\T^2} e(t) dx$ is not present, and the error in prescribing the energy increment is of size $O(N^{-1})$.  

This weaker estimate on the energy increment is still sufficient for the applications considered in those papers.  In \cite{isett} and \cite{isettOhAng}, the same estimate is applied to prove the nontriviality of solutions, by proving that the energy strictly increases during the iteration at each fixed time slice on which the corrections are nontrivial.  The same statement can be obtained here, although in our case the nontriviality of solutions follows already from the weak-* approximation statement in Theorem~\ref{thm:mainThm}.  In \cite{isettOhAng}, it was shown that a localized version of the estimate \eqref{eq:energyPrescribed} can be combined with the bounds \eqref{eq:Wco}-\eqref{eq:nabWco} to prove that that the construction necessarily results in solutions which fail to have any kind of improved $C_x^{1/5 + \ep}(B)$ local regularity (or even local $W^{1/5+\ep, 1}$ regularity) on every open ball $B$ and every time slice contained in the support of the iteration (see \cite[Theorem 1.2]{isettOhAng} for a precise statement).  This lack of higher regularity is an automatic consequence of the construction, as the same proof shows the failure of local regularity above $C_x^{1/5 - \ep}$ regularity for the earlier constructions of $C_{t,x}^{1/5 - \ep}$ solutions in \cite{isett, deLSzeHoldCts}.  The same result applies in our setting by the same proof, using the estimates \eqref{eq:Wco}-\eqref{eq:nabWco} and the localized version of \eqref{eq:energyPrescribed}.  Namely, our solutions in dimension $2$ fail to belong to $C_x^{1/9+\ep}(B)$ on every open ball $B$ and every time slice contained in the support of the iteration, and in dimension $d$ fail to have any local regularity $C_x^{1/(1 + 4 d)+\ep}(B)$ in a similar way.

\subsection{Modifications for the higher dimensional case} \label{sec:explain:higherD}

In this subsection we make some remarks about how to modify our proof to apply in higher dimensions.

In order to prove Theorem~\ref{thm:nDims} regarding the case of higher dimensions, the relevant Main Lemma has a slightly different formulation, as one must modify the definitions of the compound scalar stress equation and the compound frequency energy levels.  In the case of dimension $d$, we assume given a linearly independent set of vectors $A_{(1)}, \ldots, A_{(d)}$ in the image of the even part of the multiplier.  A typical solution to the Compound Scalar Stress equation will then be a solution to the equation
\ali{
\begin{split}
\pr_t \th + \pr_l(\th u^l) &= \pr_l( c_{A,(1)} A_{(1)}^l + \ldots + c_{A,(d-1)} A_{(d-1)}^l + R_J^l)  \\
u^l &= T^l(\th)
\end{split} \label{eq:higherDCompoundEqn} 
}
A single step of the iteration will remove the $A_{(1)}$ component of the error, giving a solution $\th_1$ and a new error of the form
\ali{
\pr_t \th_1 + \pr_l(\th_1 u_1^l) &= \pr_l( c_{A,(2)} A_{(2)}^l + \ldots + c_{A,(d-1)} A_{(d-1)}^l + c_{A,(d)} A_{(d)}^l + R_{J,1}^l)  \\
u_1^l &= T^l(\th_1) 
}
At the step above (or even earlier when writing \eqref{eq:higherDCompoundEqn}) we can absorb the $A_{(2)}, \ldots, A_{(d-1)}$ components of $R_{J,1}$ into the other terms.  (To say it in a slightly different way, one can assume from the start in writing \eqref{eq:higherDCompoundEqn} that $R_J$ is a multiple of $A_{(d)}$ by absorbing the other components of $R_J$ into the other terms.)

The Definition~\ref{def:freqEnDef} of compound frequency energy levels now should include $d + 1$ different energy levels $e_v \geq e_{R,[1]} \geq \ldots \geq e_{R,[d-1]} \geq e_J$.  The Main Lemma then takes as an input a compound scalar stress field with given frequency energy levels and outputs another scalar stress field with compound frequency energy levels
\ali{
(\Xi, e_v, e_{R,[1]}, \ldots, e_{R,[d-1]}, e_J) &\mapsto (C N \Xi, e_{R,[1]}, K_1 e_{R,[2]}, \ldots, K_1 e_{R,[d-1]}, K_1 e_J, e_J') \\
e_J' &= \left(\fr{e_v^{1/2}}{e_{R,[1]}^{1/2}N} \right)^{1/2} e_{R,[1]}
}
as in \eqref{eq:theNewEnergyLevel}.  All the bounds of the Main Lemma then hold with $e_R$ replaced by $e_{R,[1]}$, since we are eliminating the first and largest component of the error, and leaving the other terms for the next stages.  

The proof of the Main Lemma is then performed similarly as below, but naturally involves more terms and notation.  The Main Lemma is applied to prove Theorem~\ref{thm:nDims} in a similar way as is done in Section~\ref{sec:mainLemWorks} below, where one maintains a constant ratio of the consecutive energy levels with size bounded by $\fr{e_v}{e_{R,[1]}} , \fr{e_{R,[i]}}{e_{R,[i+1]}} \leq \fr{K_1}{Z}$.  The difference in the iteration then is the choice of $N_{(k)} \sim Z^{(4 d + 1)/2}$ instead of \eqref{eq:Nkchoice} at later stages $k$.  Comparing the growth of frequencies $\Xi_{(k)} \sim Z^{(4 d + 1)k/2}$ to the decay in energy levels $e_{R,[1],(k)}^{1/2} \sim Z^{-k/2}$ as in \eqref{eq:correctCtxaBd}, we obtain H\"{o}lder regularity up to $\fr{1}{(4d + 1)}$ as stated in Theorem~\ref{thm:nDims}.

$ $

In the next Sections~\ref{sec:microlocalLem}-\ref{sec:boundNewStress}, we give the proof of the Main Lemma.  In the following Sections~\ref{sec:mainLemWorks}-\ref{sec:arbitNonuniqueProof}, we then explain how the Main Lemma can be used to deduce Theorems~\ref{thm:mainThm}-\ref{thm:arbitraryNonuniqueness}.






\section{The Microlocal Lemma} \label{sec:microlocalLem}
The following Lemma will be used heavily in the construction in order to control the output of a convolution operator applied to a highly oscillatory input.  The Lemma allows us to show that, to leading order, a convolution operator simply behaves like a multiplication operator when it is applied to a high frequency input with a nonlinear phase function.

In all of our applications, the kernel $K(h)$ below will be a Schwartz function essentially supported on length scales of order $|h| \sim \la^{-1}$ for large $\la$.  We normalize the Fourier transform of a function $K : \R^2 \to \C$ to be
\[ \hat{K}(\xi) = \int_{\R^2} e^{- i \xi \cdot h} K(h) dh \]
\begin{lem}[Microlocal Lemma]\label{lem:microlocal}
Suppose that 
\[ T[\Th](x) = \int_{\R^2} \Th(x - h) K(h) dh \]
is a convolution operator acting on functions $\Th : \T^2 \to \C$, with a kernel $K : \R^2 \to \C$ in the Schwartz class.  Let $\xi : \T^2 \to \R$ and $\th : \T^2 \to \C$ be smooth, periodic functions and $\la \in \Z$ be an integer.  Then for any input of the form
\[ \Th = e^{i \la \xi(x)} \th(x) \]
we have the formula
\ali{
T[\Th](x) &= e^{i \la \xi(x)} \left( \th(x) {\hat K}(\la \nab \xi(x))  + \de[T\Th](x) \right) \label{eq:microLocal}
}
where the error in the amplitude term has the explicit form
\ali{
\label{eq:deThExpress}
\begin{split}
\de[T\Th](x) &= \int_0^1 dr \fr{d}{dr} \int_{\R^2} e^{- i \la \nab \xi(x) \cdot h} e^{i Z(r,x, h)} \th(x - rh) K(h) dh \\
Z(r,x,h) &= r \la \int_0^1 h^a h^b \pr_a \pr_b \xi(x - s h) (1-s) ds 
\end{split}
}
\end{lem}
\begin{proof}
Observe that
\ali{
e^{- i \la \xi(x)} T[\Th](x) &= \int_{\R^2} e^{i \la(\xi(x-h) - \xi(x))} \th(x-h) K(h) dh \label{eq:conjugatedTOp}
}
By Taylor expanding, we express
\ali{
 \xi(x-h) - \xi(x) &= - \nab \xi(x) \cdot h + \int_0^1 h^a h^b  \pr_a \pr_b \xi(x - s h) (1-s) ds \label{eq:taylorXiExpand}
}
In our applications, the kernel $K$ is localized to small values of $|h| \sim \la^{-1}$ for large $\la$, so we view the second term in \eqref{eq:taylorXiExpand} as a small error.  Similarly, we think of $\th(x - h)$ as a perturbation of $\th(x)$, which motivates us to express the right hand side of \eqref{eq:conjugatedTOp} as
\ali{
e^{- i \la \xi(x)} T[\Th](x) &= \int_{\R^2} e^{- i \la \nab \xi(x) \cdot h} \th(x) K(h) dh + \de[T \Th](x), \label{eq:mainAndDeTermsTTh}
}
where $\de[T \Th](x)$ is expressed in \eqref{eq:deThExpress}.  The proof concludes by recognizing that $\th(x)$ can be factored out of the integral in \eqref{eq:mainAndDeTermsTTh}, which gives formula \eqref{eq:microLocal}.
\end{proof}

\begin{rem} We remark that the same method applied here to prove Lemma \ref{lem:microlocal} can also be iterated to obtain a higher order expansion of $T[\Th](x)$ involving only the functions $\th(x)$, $\nab \xi(x)$ and their derivatives evaluated at the point $x$
\ali{
 \de[T\th](x) = -i ~\pr_a \th(x) \pr^a \hat{K}(\la \nab \xi(x)) - \fr{1}{2} i \la~ \th(x) \pr_a \pr_b \xi(x)  \pr^a \pr^b \hat{K} (\la \nab \xi(x)) + \ldots  \label{eq:higherOrderMicrolocal}
}
To obtain this further expansion, one modifies the function $Z$ defined in \eqref{eq:deThExpress} to have an additional factor of $r$ in the argument of the phase function
\[ Z(r,x,h) = r \la \int_0^1 h^a h^b \pr_a \pr_b \xi(x - r s h) (1-s) ds \]
The expansion \eqref{eq:higherOrderMicrolocal} is then obtained by Taylor expansion in the variable $r$ via integration by parts.  We do not take this approach here because it does not improve our estimates, and results in some more complicated formulas.
\end{rem}

\section{The Construction}\label{sec:construction}
We now give a detailed description of the construction.  We start by obtaining a complete list of the error terms.

Suppose that we are in the setting of Lemma \ref{lem:mainLemma}.  Thus, we have a solution $(\th, u, c_A, R_J)$ to the compound scalar-stress equation with vector $A^l = m^l(\xi^{(1)}) + m^l(-\xi^{(1)})$ as in \eqref{eq:ourAlandBl}
\ali{
\left \{
\begin{aligned}
\pr_t \th + \pr_l(\th u^l) &= \pr_l( c_A A^l + R_J^l)  \\
u^l &= T^l(\th) 
\end{aligned}
\right.
}
whose frequency-energy levels are below $(\Xi, e_v, e_R, e_J)$.  After adding a correction $\Th$ to the scalar field, the corrected scalar $\th_1 = \th + \Th$ and drift velocity $u_1^l = u^l + U^l$, $U^l = T^l[\Th]$ satisfy the system
\ali{
\pr_t \th_1 + \pr_l(\th_1 u_1^l) &= \pr_t \Th + \pr_l( u^l \Th) + \pr_l(\th U^l) + \pr_l( \Th U^l + c_A A^l + R_J^l) \label{eq:correctedSystem1}
}
As a preliminary step, it is necessary to define suitable regularizations $(\th_\ep, u_\ep, \tilde{c}_A, R_\ep)$ of the given $(\th, u, c_A, R_J)$.  The purpose of these regularizations is to ensure that only the ``low frequency parts'' of the given solutions $(\th, u, c_A, R_J)$ will influence the building blocks of the construction.  These mollifications give rise to an error term
\ali{
R_M^l &= (u^l - u_\ep^l) \Th + (\th - \th_\ep)U^l + (c_A - \tilde{c}_A) A^l + (R_J^l - R_\ep^l) \label{eq:mollifyRappears}
}
Our goal is to design a correction $\Th$ for the scalar field $\th$ so that the corrected scalar $\th_1 = \th + \Th$ and drift velocity $u_1^l = u^l + U^l$ satisfy the compound scalar-stress equation with vector $B^l = m^l(\xi^{(2)}) + m^l(-\xi^{(2)})$ as in \eqref{eq:ourAlandBl}
\ali{
\label{eq:scalarStressB}
\left \{
\begin{aligned}
\pr_t \th_1 + \pr_l(\th_1 u_1^l) &= \pr_l( c_B B^l + R_1^l)  \\
u_1^l &= T^l[\th_1] 
\end{aligned}
\right.
}
whose compound frequency energy levels are bounded as in Lemma \ref{lem:mainLemma}.

\subsection{The shape of the corrections}\label{sec:shapeOfCorrections}
Our correction is a sum of individual waves
\ali{
\Th &= \sum_I \Th_I \\
\Th_I &= e^{i \la \xi_I}(\th_I + \de \th_I) \label{eq:aThILooksLike}
}
where we are free to specify the amplitudes $\th_I$ and the phase function $\xi_I$.  The parameter $\la$ is a large frequency parameter of the form
\ali{
\la &= B_\la N \Xi \label{eq:lambdaDef}
}
where $B_\la$ is a very large constant associated to $\la$ which is chosen at the end of the argument.  (For technical reasons, we will require that $\la \in \Z_+$ is a positive integer, so $B_\la$ will really have some dependence on $N \Xi$, but will nonetheless be bounded, and should be thought of as a constant.)  The term $\de \th_I$ in \eqref{eq:aThILooksLike} is a small correction term which is present to ensure that the wave $\Th_I$ has compact support in frequency space.  We will specify $\de \th_I$ later, but it is important to remark that 
\[ \co{\de \th_I} \to 0, \tx{ as } \la \to \infty \]
Each wave $\Th_I$ has a conjugate wave $\Th_{\bar{I}} = \overline{\Th}_I$ with an opposite phase function $\xi_{\bar{I}} = - \xi_I$ and amplitude $\th_{\bar{I}} = \bar{\th}_I$ so that the overall correction is real-valued.  We will choose the amplitudes $\th_I = \bar{\th}_I$ to be real-valued as well.

The index $I$ for the wave $\Th_I$ consists of two parts $I = (k,f) \in \Z \times \{ \pm \}$.  The discrete index $k \in \Z$ specifies the support of the wave $\Th_I = \Th_{(k,f)}$ in time.  Specifically, the support of $\Th_{(k,f)}$ will be contained in the time interval $[(k - \fr{2}{3}) \tau, (k + \fr{2}{3})\tau]$ where $\tau$ is a time scale parameter that will be chosen during the iteration.  The index $f \in \{ \pm \}$ is a sign which specifies the direction of oscillation of the wave $\Th_{(k,f)}$.

The phase functions $\xi_I$ are solutions to the transport equation
\ali{
\label{eq:transportEqn} 
\begin{split}
(\pr_t + u_\ep^l \pr_l) \xi_I &= 0 \\
\xi_I(t(I), x) &= \hat{\xi}_I(x)
\end{split}
}
The amplitudes $\th_I$ will be supported on a small time interval during which the phase functions remain close to their initial data.  The initial data $\hat{\xi}_I$ for the phase function $\xi_I = \xi_{(k,f)}$ is chosen at the time $t(I) = k \tau$ depending on the index $I = (k,f)$
\ali{
\hat{\xi}_I(x) = \hat{\xi}_{(k, \pm)}(x) &= \pm 10^{[k]} \xi^{(1)} \cdot x 
}
where $[k] \in \{ 0, 1 \}$ is equal to $0$ when $k$ is even and is equal to $1$ when $k$ is odd.  In particular, we have
\[ \nab \hat{\xi}_I = \pm 10^{[k]} \xi^{(1)}, \qquad [k] \in \{ 0, 1 \} \]
Our individual waves are localized in frequency and take the form
\ali{
\Th_I &= P_{\approx \la}^I [ e^{i \la \xi_I} \th_I ] \label{eq:ThIisLike}
}
The operators $P_{\approx \la}^I$ in \eqref{eq:ThIisLike} restrict to frequencies of order $\la$ in a neighborhood of $\la \nab \hat{\xi}_I$.  To be explicit, let ${\hat \eta}_{\approx 1}(\xi)$ be a bump function supported on frequencies 
\[ {\hat \eta}_{\approx 1}(\xi) \in C_c^\infty\left(B_{|\xi^{(1)}|/2}(\xi^{(1)}) \right) \]
which has the property that 
\[ {\hat \eta}_{\approx 1}(\xi) = 1, \qquad ~\tx{if } |\xi - \xi^{(1)}| \leq \fr{1}{4} |\xi^{(1)}| \] 
We then define a frequency cutoff supported on high frequencies of order $\la$ by rescaling and reflection
\[ {\hat \eta}^{I}_{\approx \la}(\xi) = {\hat \eta}_{\approx 1}(\pm 10^{-[k]} \la^{-1} \xi). \] 
Then $P_{\approx \la}^{I}$ is given explicitly by a Fourier multiplier 
\[ \widehat{P_{\approx \la}^{I} F}(\xi) = {\hat \eta}_{\approx \la}^I(\xi) {\hat F}(\xi). \]   
Including this ``projection operator'' $P_{\approx \la}^{I}$ guarantees that all the corrections \eqref{eq:ThIisLike} have frequency support in the ball $| \xi - ( \la \nab \hat{\xi}_I ) | \leq \la \fr{|\nab \hat{\xi}_I|}{2} $, and in particular have integral $0$.  Having compact support in frequency space will allow us to easily control the resulting increment to the velocity field, which is obtained by applying another Fourier multiplier.

By the Microlocal Lemma \ref{lem:microlocal}, it is possible to write the wave \eqref{eq:ThIisLike} in the form \eqref{eq:aThILooksLike} with an explicit remainder $\de \th_I$, since we have
\ali{
\Th_I = P_{\approx \la}^I [ e^{i \la \xi_I} \th_I ] &= e^{i \la \xi_I}( \th_I {\hat \eta}_{\approx \la}^I(\la \nab \xi_I) + \de \th_I) \notag \\
&= e^{i \la \xi_I}( \th_I + \de \th_I)
}
provided that the phase gradient is sufficiently close to its initial value
\ali{
| \nab \xi_I - \nab \hat{\xi}_I | &\unlhd \fr{|\nab \hat{\xi}_I|}{4} \label{eq:keepPhaseGradientClose}
}
We will verify that inequality \eqref{eq:keepPhaseGradientClose} is satisfied when the parameter lifespan parameter $\tau$ is chosen.

Applying the Microlocal Lemma \ref{lem:microlocal} again, we can also calculate the resulting correction to the drift velocity.
\ali{
U_I^l &:= T^l \Th_I \\
T^l P_{\approx \la}^I [ e^{i \la \xi_I} \th_I ] &= e^{i \la \xi_I}(\th_I {\hat K}^l(\la \nab \xi_I) {\hat \eta}_{\approx \la}^I(\la \nab \xi_I) + \de u_I^l ) \\ 
U_I^l &= e^{i \la \xi_I}( \th_I m^l(\nab \xi_I) + \de u_I^l )
}
Therefore, once we have verified \eqref{eq:keepPhaseGradientClose}, we have
\ali{
U_I^l &= e^{i \la \xi_I}( u_I^l + \de u_I^l) \\
u_I^l &= \th_I m^l(\nab \xi_I) \label{eq:uIexpression}
}
with an explicit error term $\de u_I^l$ given by Lemma \ref{lem:microlocal}.

\subsection{Choosing the amplitudes }\label{sec:chooseAmplitudes}
According to Section \ref{sec:shapeOfCorrections}, we can now decompose the remaining error terms in Equation \eqref{eq:correctedSystem1} as follows
\ali{
\pr_t \th_1 + \pr_l(\th_1 u_1^l) &= \pr_t \Th + \pr_l( u_\ep^l \Th) + \pr_l(\th_\ep U^l) \label{eq:transAndHighLow}\\
&+ \pr_l[  \sum_I \Th_I U_{\bar{I}}^l + \tilde{c}_A A^l + R_\ep^l ] \label{eq:theStressTerm} \\
&+  \sum_{J \neq \bar{I}} \pr_l(\Th_I U_J^l) \label{eq:highFreqTerms} \\
&+ \pr_l R_M^l
}
The term $R_M^l$ comes from the regularizations in Equation \eqref{eq:mollifyRappears}.

The first objective of the correction is to eliminate the term \eqref{eq:theStressTerm}, which is the only low frequency term that arises.  However, since we consider oscillations in essentially only one direction $\nab \xi_I \approx \pm 10^{[k]} \xi^{(1)}$, we will only able to eliminate the $A^l$ component of \eqref{eq:theStressTerm}.   

We begin by expanding the low frequency part of the interactions in line \eqref{eq:theStressTerm} as
\ali{
\sum_I (\Th_I U_{\bar{I}}^l) &= \fr{1}{2} \sum_I (\Th_I U_{\bar{I}}^l + \Th_{\bar{I}} U_I^l) \notag \\
&= \sum_{I \in \II_+} \th_I \bar{\th}_I (m^l(- \nab \xi_I) + m^l(\nab \xi_I)) + \tx{Lower Order Terms} \label{eqref:mainNonlinearTerm} \\
&= \sum_{I \in \II_+} |\th_I|^2 (m^l(- \nab \hat{\xi}_I) + m^l(\nab \hat{\xi}_I)) + \tx{Lower Order Terms} \\
&= \sum_{I \in \II_+} |\th_I|^2 A^l + \tx{Lower Order Terms} \label{eq:mainPartIsLikeAl}
}
We will give a complete list of the lower order terms below after we have chosen the amplitudes $\th_I$.

We wish to choose the amplitudes $\th_I$ so that the main term in \eqref{eq:mainPartIsLikeAl} cancels with the $A^l$ component of the other terms in line \eqref{eq:theStressTerm}.  We achieve this cancellation in two steps.  First, we decompose $R_\ep$ into components
\ali{
R_\ep^l &= c_J A^l + c_B B^l \label{eq:decomposeRepComponents}
}
We also subtract a constant vector field $\pr_l( e(t) A^l) = 0$ from line \eqref{eq:theStressTerm}, which leads us to impose to an equation
\ali{
 \sum_{I \in \II_+} |\th_I|^2 A^l &= e(t) A^l +\tilde{c}_A A^l + c_J A^l \label{eq:equationForAmplitudes}  \\
&= e(t)(1 + \varep) A^l \label{eq:eqForThIs}\\
\varep &= \fr{\tilde{c}_A + c_J}{e(t)} \label{eq:varepTerm}
}
for the amplitudes $\th_I$.  In this way, the amplitudes $\th_I$ are chosen to eliminate the $A^l$ component of the low frequency part of the stress $R_\ep^l$.

It will be important for our construction that the term $\varep$ is smaller than the constant $1$ in the $(1 + \varep)$ term in \eqref{eq:eqForThIs}.  From the lower bound $e(t) \geq K e_R$ assumed in \eqref{eq:lowBoundEoftx}, we can obtain an upper bound 
\ali{
\co{\varep} &\leq \fr{Z}{K} \label{eq:boundOnVarep}
}
on the size of the term \eqref{eq:varepTerm}, where $Z$ is a constant depending only on the vectors $A^l$ and $B^l$.  Now, provided $K \geq K_0 = 2 Z$, we have
\ali{
\co{\varep} &\leq \fr{1}{2} \label{eq:whichK0weneed}
}
A subtle point here is that the bound \eqref{eq:boundOnVarep} does not follow immediately from \eqref{eq:lowBoundEoftx}.  Namely, we must also check that the same lower bound remains true on the set
\ali{
e(t) \geq K e_R \mbox{ for all } (t,x) \in \supp ( \tilde{c}_A + c_J ), \label{eq:stillNeedToCheckLowBound}
}
which is slightly larger than the supports of the given $R_J$ and $c_A$ due to a regularization in time in the definitions of $\tilde{c}_A$, $c_J$.  Thus, the estimates \eqref{eq:boundOnVarep}-\eqref{eq:whichK0weneed} are guaranteed only after \eqref{eq:stillNeedToCheckLowBound} has been verified, which is accomplished in Line \eqref{eq:pickingLengthTimeScales} below when we choose the mollifying parameters.  We now assume that \eqref{eq:boundOnVarep}-\eqref{eq:whichK0weneed} hold in order to finish defining the construction.

From Equation \eqref{eq:eqForThIs}, we are led to choose amplitudes of the form
\ali{
\th_I &= e^{1/2}(t) \eta_k(t) \ga, \qquad I = (k,f) \label{eq:formForThi} \\
\ga &=  (1 + \varep)^{1/2}
}
The functions
\[ \eta_k(t) = \eta\left( \fr{t - k \tau}{\tau} \right) \]
are elements of a rescaled partition of unity in time
\[ \sum_{u \in \Z} \eta^2(t - u) = 1 \]
which we use to patch together local solutions of Equation \eqref{eq:eqForThIs}.  Our choice of $\eta_k$ ensures that each amplitude $\th_{(k,f)}$ has support in a time interval $[k \tau - \fr{2 \tau}{3},k \tau + \fr{2 \tau}{3}]$ of duration $\fr{4 \tau}{3}$.  The coefficient $\ga$ ensures that \eqref{eq:eqForThIs} is satisfied, and $\ga$ is assured to be well-defined by the bound \eqref{eq:whichK0weneed}.

To express the remaining error terms in a compact way, let us introduce the notation
\[ \tilde{\th}_I = \th_I + \de \th_I, \quad \tilde{u}_I^l = u_I^l + \de u_I^l \]
Thus, $\Th_I = e^{i \la \xi_I} \tilde{\th}_I$ and $U_I^l = e^{i \la \xi_I} \tilde{u}_I^l$.

Having chosen $\th_I$, we can now expand the error term in \eqref{eq:theStressTerm} as follows
\ali{
 \sum_I \Th_I U_{\bar{I}}^l &+ \tilde{c}_A A^l + R_\ep^l = c_B B^l + R_S^l \\
R_S^l &= \sum_I (\Th_I U_{\bar{I}}^l) - \sum_{I \in \II_+} |\th_I|^2 A^l 
}
We now expand
\ALI{
(\Th_I U_{\bar{I}}^l + \Th_{\bar{I}} U_I^l) &= \tilde{\th}_I \tilde{u}_{\bar{I}}^l + \tilde{\th}_{\bar{I}} \tilde{u}_I^l \\
&= |\th_I|^2 (m^l(-\nab \hat{\xi}_I) + m^l(\nab \hat{\xi}_I) ) + R_{S,1}^l + R_{S, 2}^l  \\
R_{SI,1}^l &= |\th_I|^2 [ (m^l(-\nab \xi_I) - m^l(-\nab \hat{\xi}_I)) + ( m^l(\nab \xi_I) - m^l(\nab \hat{\xi}_I) ) ] \\
\begin{split}
R_{SI,2}^l &= \de \th_I \tilde{u}_{\bar{I}}^l + \tilde{\th}_I \de u_{\bar{I}}^l - \de \th_I \de u_{\bar{I}}^l  + \de \th_{\bar{I}} \tilde{u}_{I}^l + \tilde{\th}_{\bar{I}} \de u_{I}^l - \de \th_{\bar{I}} \de u_{I}^l
\end{split}
}
which gives
\ali{
R_S^l &= \sum_I (R_{SI,1}^l + R_{SI,2}^l ) \label{eq:stressErrorInPieces}
}
Note that, at any given time $t$, at most four indices $I$ contribute to the sum in \eqref{eq:stressErrorInPieces}.

\subsection{The remaining error terms }\label{sec:theErrorTerms}

In Sections \ref{sec:shapeOfCorrections}-\ref{sec:chooseAmplitudes} we have defined the construction up to the specification of a few parameters.  Our result is that the corrected field $\th_1 = \th + \Th$ and drift velocity $u_1^l = u^l + U^l$ satisfy Equation \eqref{eq:scalarStressB} with $c_B$ defined in line \eqref{eq:decomposeRepComponents}, and
\ali{
 R_1 &= R_T + R_L + R_H + R_M + R_S \label{eq:allTheErrors}
}
The terms $R_M$ $R_S$ are defined in \eqref{eq:mollifyRappears} and \eqref{eq:stressErrorInPieces}.  We now rewrite the remaining terms in Equations \eqref{eq:transAndHighLow}-\eqref{eq:highFreqTerms} using the fact that the velocity fields appearing in these equations are divergence free.

The {\bf transport term} $R_T$ is obtained by solving
\ali{
\pr_l R_T^l &= (\pr_t + u_\ep^l \pr_l) \Th \label{eq:transportTerm} \\
&= \sum_I e^{i \la \xi_I} (\pr_t + u_\ep^l \pr_l){\tilde \th}_I \label{eq:explicitTransport}
}
Here the term where the derivative hits the phase functions vanishes according to equation \eqref{eq:transportEqn}.  Formula \eqref{eq:explicitTransport} suggests that the transport term has frequency $\la$, so we expect to gain a factor $\la^{-1}$ in solving equation \eqref{eq:transportTerm}.  In fact, we will choose our mollification $u_\ep$ to be a frequency-localized version of $u$ so that together with \eqref{eq:ThIisLike}, the term \eqref{eq:transportTerm} is literally supported on frequencies of size $\fr{\la}{3} \leq |\xi| \leq 20 \la$.  Hence, there is a frequency localizing operator $P_{\approx \la}$ satisfying
\[ (\pr_t + u_\ep^l \pr_l) \Th = P_{\approx \la}[ (\pr_t + u_\ep \cdot \nab) \Th ]  \]
This frequency localization property allows us to simply define
\ali{
R_T^l &= \pr^l \De^{-1} P_{\approx \la} [ (\pr_t + u_\ep \cdot \nab) \Th ] \label{eq:RTdef}
}
In particular, we obtain the bound
\ali{
 \co{R_T} &\leq \la^{-1} \co{(\pr_t + u_\ep \cdot \nab) \Th} \label{eq:transportBound1}
}

The terms remaining from \eqref{eq:transAndHighLow} and \eqref{eq:highFreqTerms} are the {\bf High-Low term}
\ali{
\pr_l R_L^l &= U^l \pr_l \th_\ep \label{eq:highLowTerm} \\
&= \sum_I e^{i \la \xi_I} {\tilde u}_I^j \pr_j \th_\ep
}
and the {\bf high frequency interference terms} 
\ali{
\pr_l R_H^l &= \sum_{J \neq \bar{I}} U_J^l \pr_l \Th_I
}
The frequency cutoffs in our definitions of $\th_\ep, U_I$ and $\Th_I$ ensure both of these terms have Fourier support in frequencies $\fr{\la}{3} \leq |\xi| \leq 40 \la$.  Here it is important that we have localized the frequency support each $\Th_I$ and $U_I^l$ to a limited range of angles.  As a consequence, 
\[ U^l \pr_l \th_\ep = P_{\approx \la} [ U^l \pr_l \th_\ep ] \] 
\[ U_J^l \pr_l \Th_I = P_{\approx \la} [U_J^l \pr_l \Th_I] \]
for some frequency projection operator $P_{\approx \la}$, and we can define
\ali{
R_L^l &= \pr^l \De^{-1} P_{\approx \la} [ U^j \pr_j \th_\ep ] \label{eq:highLowTermIs}\\
R_H^l &= \sum_{J \neq \bar{I}} \pr^l \De^{-1} P_{\approx \la} [ U_J^l \pr_l \Th_I ] \label{eq:highFreqPartOfStressIs}
}
Now that we have written down the error terms \eqref{eq:allTheErrors}, we must observe that each of these terms can be made small.  For the transport term $R_T^l$, the estimate \eqref{eq:transportBound1} ensures that $R_T$ is small once $\la$ is chosen sufficiently large, and the same type of estimate can be used to control the High-Low term $R_L$.  The high-frequency interference terms require a more careful treatment.

Let us focus on an individual term in the sum.
\ali{
U_J^l \pr_l \Th_I &=  ( i \la) e^{i \la(\xi_I + \xi_J)} \tilde{u}_J^l \pr_l {\tilde \th}_I \label{eq:mainHighFreq1}
}
We expand this term as
\ali{
 U_J^l \pr_l \Th_I &=  ( i \la) e^{i \la(\xi_I + \xi_J)} (\th_J m^l(\nab \xi_J) + \de u_J^l) \pr_l \xi_I {\tilde \th}_I \label{eq:mainInterferenceTerm}
} 
If we regard the phase gradients $\nab \xi_I \approx \nab \hat{\xi}_I$ as perturbations of their initial values, the main term in \eqref{eq:mainInterferenceTerm} vanishes
\[ m^l(\nab \hat{\xi}_J) \pr_l \hat{\xi}_I = m^l(\pm \nab \hat{\xi}_I) \pr_l \hat{\xi}_I = 0\] 
from the degree zero homogeneity of $m(\xi)$ and the identity $m(\xi) \cdot \xi = 0$.

The terms which remain are all lower order
\ali{
\fr{1}{(i \la)} U_J^l \pr_l \Th_I &= e^{i \la (\xi_I + \xi_J)} \th_J {\tilde \th}_I ( m^l(\nab \xi_J) - m^l(\nab \hat{\xi}_J) ) \pr_l \xi_I \label{eq:MIJfirst} \\
&+ e^{i \la (\xi_I + \xi_J)} \th_I {\tilde \th}_J m^l(\nab \hat{\xi}_J)( \pr_l \xi_I - \pr_l \hat{\xi}_I) \label{eq:MIJfirst2} \\
&+ e^{i \la (\xi_I + \xi_J)} \de u_J^l \pr_l \xi_I {\tilde \th}_I \label{eq:MIJlast}
}
The terms \eqref{eq:MIJfirst}, \eqref{eq:MIJfirst2} are made small by choosing the lifespan parameter $\tau$ to be small, while the term \eqref{eq:MIJlast} is made small once $\la$ is chosen sufficiently large (see Section~\ref{sec:lifespanChoice}). The high-frequency term $R_H$ itself is then controlled by the estimate
\[ \co{R_H} \leq \fr{C}{\la} \co{ \sum_{J \neq \bar{I}} U_J^l \pr_l \Th_I } \]
from the formula \eqref{eq:highFreqPartOfStressIs}.  This calculation concludes our list of the error terms \eqref{eq:allTheErrors}.  What remains is to specify the parameters in the construction, prove estimates for the elements of the construction and finally to check that the estimates stated in Lemma \ref{lem:mainLemma} are satisfied.

\section{Specifying Parameters and the Mollification Term} \label{sec:specifyingParams}
To initialize the argument, we must specify how we regularize the given solution $(\th, u, c_A, R^l)$ to the compound scalar-stress equation.  In this section, we specify how these regularizations are defined.  Because the flow map of the regularized velocity is used to define the regularizations of $c_A$ and $R^l$, it is necessary to start with the defininition of the regularized velocity.  After the regularizations of $c_A$ and $R^l$ are defined, we are able to verify the lower bound \eqref{eq:stillNeedToCheckLowBound} which had been assumed previously to guarantee a well-defined construction.

To obtain the regularized scalar field $\th_\ep$ and drift velocity $u_\ep$, we take low frequency projections in the spatial variables with length scale parameters $\ep_\th$ and $\ep_u$
\ali{
\th_\ep &= P_{\leq q}^2 \th, \qquad \mbox{where}\qquad 2^{-q} = \ep_\th \\
u_\ep &= P_{\leq q}^2 u, \qquad \mbox{where} \qquad 2^{-q} = \ep_u \label{eq:mollifiedu}
}
The reason for the double mollification in equation \eqref{eq:mollifiedu} will become apparent during the commutator estimates of Section~\ref{sec:estimatesCorrections}.  
The operator is given by rescaling a Fourier multiplier
\[ \widehat{P_{\leq \chi} F}(\xi) = \hat{\eta}\left(\fr{\xi}{2^\chi}\right) \hat{F}(\xi) \]
where $\hat{\eta}(\xi)$ is a smooth function with compact support in $|\xi| \leq 2$ that is equal to $1$ on $|\xi| \leq 1$.

By well-known estimates for convolutions with mollifiers satisfying vanishing moment conditions (see \cite{isett} Section 5.1), we have
\ali{
\co{ \th - \th_\ep } &\leq C_L \ep_\th^L \co{ \nab^L \th } \label{eq:basicConvEstth}\\
 \co{ u - u_\ep } &\leq C_L \ep_u^L \co{ \nab^L u } \label{eq:basicConvEstu}
}
We want to choose the length scales $\ep_\th$ and $\ep_u$ as large as possible while ensuring that the mollification term $R_M^l$ in \eqref{eq:mollifyRappears} is acceptably small.  The main terms in \eqref{eq:mollifyRappears} where these mollification errors appear are
\ali{
R_{M,\th}^l &= \sum_I (\th - \th_\ep) e^{i \la \xi_I} u_I^l \label{eq:mollThTerm}\\
R_{M,u}^l &= \sum_I e^{i \la \xi_I} \th_I (u^l - u_\ep^l) \label{eq:molluTerm}
}
Logically, the terms \eqref{eq:mollThTerm} and \eqref{eq:molluTerm} are not well-defined until we have specified how to define $u_\ep$.  However, from the expressions \eqref{eq:formForThi} and \eqref{eq:uIexpression} and the bound \eqref{ineq:goodEnergy} we have an a priori estimate
\ali{
\co{R_{M,\th}} + \co{R_{M,u}} &\leq A e_R^{1/2} (\co{\th - \th_\ep} + \co{u - u_\ep}) \label{eq:aPrioriMollTerms}
}
for $A$ a constant depending only on the parameter $M$ in Lemma \ref{lem:mainLemma}.

Using \eqref{eq:basicConvEstth}-\eqref{eq:basicConvEstu} for $a = L$ and the bound \eqref{bound:nabkth}, we can choose parameters of the form
\ali{
 \ep_\th &= \ep_u = \fr{1}{B} N^{1/L} \Xi \label{eq:epParamChoice}
}
Here $B$ is some large constant depending on $A$ in \eqref{eq:aPrioriMollTerms} chosen to assure that
\ali{
\co{R_{M,\th}} + \co{R_{M,u}} &\leq \fr{e_v^{1/2}e_R^{1/2}}{1000 N} \label{eq:ourFirstMollifyGoal}
}
The estimate \eqref{eq:ourFirstMollifyGoal} is stronger than what we require for Lemma \ref{lem:mainLemma}.  Rather, estimate \eqref{eq:ourFirstMollifyGoal} is the type of bound one requires to obtain solutions with regularity $1/3 -$ (see \cite{isett} Section 4.6).  

Observe that the parameter choice \eqref{eq:epParamChoice} is exactly the choice of parameter taken in \cite{isett} Section 5.2 up to a constant.  We will therefore in many cases be able to refer to the estimates of \cite{isett} without repeating the proofs.

Having defined $\th_\ep$ and $u_\ep$, we can now define our regularizations $\tilde{c}_A$ and $R_\ep^l$ of $c_A$ and $R_J^l$.

Following \cite{isett}, we define these regularizations using the coarse scale flow $\Phi_s$ associated to $\pr_t + u_\ep \cdot \nab$, whose definition we now recall.
\begin{defn}\label{defn:coarseScaleFlow} We define the {\bf coarse scale flow} $\Phi_s(t,x) : \R \times \R \times \T^2 \to \R \times \T^2$ to be the unique solution to the ODE 
\ALI{
\fr{d}{ds} \Phi_s^0(t,x) &= 1 \\
\fr{d}{ds} \Phi_s^j(t,x) &= u_\ep^j(\Phi_s(t,x)), \qquad j = 1, 2 \\
\Phi_0(t,x) &= (t,x)
}
\end{defn}

We can now define our regularizations for $c_A$ and $R_J$.  First, we mollify both $c_A$ and $R_J^l$ in space to define
\ali{
c_{A, \ep_x} &= \eta_{\ep_x} \ast c_A \\
R_{\ep_x} &= \eta_{\ep_x} \ast R_J
}
We then use the coarse scale flow $\Phi_s$ and a smooth function $\eta_{\ep_t}(s)$ supported in $|s| \leq \ep_t$ with integral $\int \eta_{\ep_t}(s) ds = 1$ to average in time and form
\ali{
\tilde{c}_A(t,x) &= \int c_{A, \ep_x}(\Phi_s(t,x) ) \eta_{\ep_t}(s) ds \\
R_\ep^l(t,x) &= \int R_{\ep_x}(\Phi_s(t,x)) \eta_{\ep_t}(s) ds
}
Thus, the values of $R_\ep(t,x)$ and $\tilde{c}_A(t,x)$ are obtained by averaging $c_A$ and $R_J^l$ over an $\ep_x$-neighborhood of a time $2\ep_t$ flow line through $(t,x)$.

To estimate $\tilde{c}_A$ and $R_\ep^l$, we recall that both $c_A$ and $R_J$ satisfy the estimates
\ali{
|| \nab^k c_A ||_{C^0} + \left( \fr{e_R}{e_J} \right) \| \nab^k R_J \| &\leq 2 \Xi^k e_R &k = 0, \ldots, L \label{bound:nabkCAandR} \\
|| \nab^k (\pr_t + u \cdot \nab) c_A ||_{C^0} + \left( \fr{e_R}{e_J} \right) || \nab^k (\pr_t + u \cdot \nab) R_J ||_{C^0} &\leq 2 \Xi^{k+1} e_v^{1/2} e_R  &k = 0, \ldots, L - 1 \label{bound:dtnabkCAandR}
}
coming from the compound frequency-energy levels of Definition \ref{def:freqEnDef}.

Since the bounds \eqref{bound:nabkCAandR}-\eqref{bound:dtnabkCAandR} coincide with the bounds for the tensor $R^{jl}$ in \cite{isett}, we can draw from the results of that paper to control $R_\ep^l$ and $\tilde{c}_A$.

Following Section 5.5 of \cite{isett}, we choose length and time scales of the form
\ali{
 \ep_x = \fr{1}{B} N^{1/L} \Xi, \quad \ep_t = \fr{1}{B} (N \Xi)^{-1} e_R^{-1/2} \label{eq:epXandepT}
}
We choose $B \geq 1$ large enough to bound the terms
\ali{
\co{ (\tilde{c}_A - c_A) A^l } + \co{ R_J^l - R_\ep^l } &\leq \fr{e_v^{1/2}e_R^{1/2}}{100 N}  \label{eq:pickingLengthTimeScales}
}
which appear in the list of error terms from mollification of Equation \eqref{eq:mollifyRappears}.  

Note that the choice of parameters~\eqref{eq:epXandepT} is the same as the choice made in~\cite[Section 18.3]{isett}, and therefore leads to the same bounds
\ali{
\begin{split}
\co{ \nabla^k \left(\Ddt \right)^r \tilde{c}_A } &\leq C_k \Xi^k e_R (\Xi e_v^{1/2} )^{(r\geq 1)} (N \Xi e_R^{1/2})^{(r\geq 2)} N^{(k+1-L)_+/L} \\
\co{\nabla^k \left( \Ddt \right)^r R_J} +  \co{\nabla^k \left( \Ddt \right)^r c_J}  &\leq C_k \Xi^k e_J (\Xi e_v^{1/2} )^{(r\geq 1)} (N \Xi e_R^{1/2})^{(r\geq 2)} N^{(k+1-L)_+/L}
\end{split} \label{eq:tilde:cA:bounds}
}
where we use the notation $(r\geq m) = \chi_{[m,\infty)}(r)$.  The fact we are using here is that $c_A$ obeys the same estimates as the stress $R^{jl}$ in \cite{isett}, and the terms $R_J$ and $c_J$ satisfy even better bounds.  The details of the proof are carried out in \cite[Section 18]{isett}.

A crucial point here is that \eqref{eq:tilde:cA:bounds} contains estimates on {\it second order} advective derivatives, even though our assumed bounds \eqref{bound:dtnabkCA}, \eqref{bound:dtnabkR} on $c_A$ and $R_J $ contain only information regarding first order advective derivatives.  The ability to obtain this estimate comes from the fact that the advective derivative $\Ddt$ commutes with its own flow $\Phi_s$, and thus commutes with the averaging along the flow.  This observation allows us to integrate by parts in
\ali{
\Ddt R_\ep^l(t,x) &= \int \Ddt R_{\ep_x}(\Phi_s(t,x)) \eta_{\ep_t}(s) ds \\
&= - \int R_{\ep_x}(\Phi_s(t,x)) \eta_{\ep_t}'(s) ds
}
This computation explains why the cost of the second advective derivative in \eqref{eq:tilde:cA:bounds} is exactly a factor of $\ep_t^{-1}$ for the choice of parameter \eqref{eq:epXandepT}.  We refer to \cite[Section 18.6.1]{isett} and to \cite[Section 12.1]{IsettOhHeat13} for two different proofs of this identity.

Having defined $\tilde{c}_A$ and $R_\ep$, we are now able to verify the lower bound \eqref{eq:stillNeedToCheckLowBound} on the energy profile, which had been assumed previously in many of the formulas in our construction.  From the assumption \eqref{sub:suppAssumption} that $\supp c_A \cup \supp R_J \subseteq I \times \T^2$, we have by construction that
\[ \supp \tilde{c}_A \cup \supp R_\ep \cup \supp c_J \subseteq I \pm \ep_t \times \T^2 \]
Since we assumed the lower bound \eqref{eq:lowBoundEoftx} for $e(t)$ on the interval $I \pm \Xi^{-1} e_v^{-1/2}$, it suffices to check that $\ep_t < \Xi^{-1} e_v^{-1/2}$.  This inequality follows from the definition \eqref{eq:epXandepT} of $\ep_t$ and the inequality $N \geq \left( \fr{e_v}{e_R} \right)^{1/2}$, which follows from condition \eqref{eq:conditionsOnN2}.

At this point, the only term that remains to be estimated in the mollification term
\ali{
 R_M^l &= R_{M,\th}^l + R_{M,u}^l + (c_A - \tilde{c}_A) A^l + (R_J^l - R_\ep^l) + R_{M'}^l
 \label{eq:RMl:mollified}
}
is given by
\ali{
R_{M'}^l &= \sum_I e^{i \la \xi_I} (u^l - u_\ep^l) \de \th_I + \sum_I e^{i \la \xi_I} (\th - \th_\ep) \de u_I^l
}
This term will be estimated when we choose the parameter $\la$ at the end of the argument.

\subsection{The choice of lifespan parameter and the limiting error terms}\label{sec:lifespanChoice}

The present Section is devoted to choosing the lifespan parameter $\tau$.  Here we motivate the choice of $\tau$ by comparing the estimates that will be satisfied by the main error terms and optimizing.  However, we warn the reader that the estimates stated in this Section have not yet been established, but will follow from the bounds of Section \ref{sec:estimatesCorrections} below.

The lifespan parameter $\tau$ determines the length of time during which an amplitude
\ali{
 \th_I &= e^{1/2}(t) \eta\left( \fr{t - t(I)}{\tau} \right) \ga  \label{eq:recallAmplitudes}
}
is allowed to remain nonzero.  The parameter $\tau$ is chosen to be small so that the gradients of the phase functions, which satisfy the transport equation
\ali{
(\pr_t + u_\ep^j \pr_j) \pr^l \xi_I &= - \pr^l u_\ep^j \pr_j \xi_I, \label{eq:transportForPhaseGrad}
}
remain close to their initial values $\nab \xi_I \approx \nab \hat{\xi}_I$.  More precisely, equation \eqref{eq:transportForPhaseGrad} with initial data $\xi_I(t(I), x) = \hat{\xi}_I(x)$ leads to a bound of the form
\ali{
|\nab \xi_I(\Phi_s(t(I),x)) - \nab {\hat \xi}_I(x)| &\leq A e^{A \Xi e_v^{1/2} \tau} (\Xi e_v^{1/2}) \tau, \qquad |s|\leq \tau \label{eq:generalPhaseGradChange}
}
where $\Xi e_v^{1/2}$ is an upper bound on $\co{ \nab u_\ep } \leq \Xi e_v^{1/2}$, cf. Lemma~\ref{lem:transport:estimates} below.

In our case, we require that $\tau \leq \Xi^{-1} e_v^{-1/2}$, so that the estimate \eqref{eq:generalPhaseGradChange} becomes 
\ali{
\co{ \nab \xi_I - \nab \hat{\xi}_I } &\leq A (\Xi e_v^{1/2}) \tau \label{eq:estForPhaseGradDiscrep}
}

Here, there are two main error terms which require the choice of a sharp time cutoff in order to control.  The first such term, which is familiar from the case of the Euler equations, is the set of high-frequency interference terms in \eqref{eq:highFreqPartOfStressIs}
\ali{
R_H^l &= \sum_{J \neq \bar{I}} \pr^l \De^{-1} P_{\approx \la} [U_J^j \pr_j \Th_I] \label{eq:leadingHighFreqTermAgain} 
}
Recall from \eqref{eq:MIJfirst}-\eqref{eq:MIJlast} that each term in the series~\eqref{eq:leadingHighFreqTermAgain} can be expressed to leading order as 
\ali{
\fr{1}{(i \la)} U_J^l \pr_l \Th_I &= e^{i \la (\xi_I + \xi_J)} \th_J {\tilde \th}_I ( m^l(\nab \xi_J) - m^l(\nab \hat{\xi}_J) ) \pr_l \xi_I \label{eq:MIJfirstAgain} \\
&+ e^{i \la (\xi_I + \xi_J)} \th_I {\tilde \th}_J m^l(\nab \hat{\xi}_J)( \pr_l \xi_I - \pr_l \hat{\xi}_I) \label{eq:MIJfirst2Again} \\
&+ \tx{lower order terms} \label{eq:MIJlastAgain}
}

Formula \eqref{eq:leadingHighFreqTermAgain} leads to the bound
\ali{
\co{ R_H } &\leq \fr{A}{\la} \co{ \sum_{J \neq \bar{I}} U_J^j \pr_j \Th_I } \notag \\
&\leq A \max_I \co{\th_I}^2 (\co{m^l(\nab \xi_I) - m^l(\nab \hat{\xi}_I)} + \co{ \nab \xi_I - \nab \hat{\xi}_I } ) + \tx{Lower order terms} \notag \\
&\leq A e_R \max_I \co{ \nab \xi_I - \nab \hat{\xi}_I } + \tx{Lower order terms}  \\
&\leq A e_R (\Xi e_v^{1/2} \tau) + \tx{Lower order terms} \label{eq:highFreqEstimate}
}
where the constant $A$ changes from line to line.  The error term \eqref{eq:highFreqEstimate} is made small by choosing the lifespan parameter $\tau$ to rhe small compared to the natural time scale $\Xi^{-1} e_v^{-1/2}$ of the coarse scale velocity $u_\ep$.  The other terms in \eqref{eq:highFreqEstimate} are lower order in the sense that they can be made small by a suitable choice of $\la$.

The price we pay by introducing sharp cutoffs is a worse bound on the transport term.  
\ali{
R_T^l &= \pr^l \De^{-1} P_{\approx \la} [ (\pr_t + u_\ep \cdot \nab) \Th ] \\
(\pr_t + u_\ep \cdot \nab) \Th &= \sum_I e^{i \la \xi_I} (\pr_t + u_\ep \cdot \nab) \th_I + \tx{Lower order terms}
}
The time cutoffs appear in the formula \eqref{eq:recallAmplitudes} for the amplitude, and give rise to a term of size 
\ali{
\co{ (\pr_t + u_\ep \cdot \nab) \Th } &= A \tau^{-1} e_R^{1/2} + \tx{Lower order terms} 
}
which leads in turn from the definition \eqref{eq:lambdaDef} of $\la$. to a bound on the transport term
\ali{
\co{R_T} &\leq A \la^{-1} \co{ (\pr_t + u_\ep \cdot \nab) \Th } \\
&\leq A B_\la^{-1} (N \Xi)^{-1} \tau^{-1}e_R^{1/2} + \tx{Lower order terms} \label{eq:transTermEstimate}
}
We therefore choose
\ali{
\tau = B_\la^{-1/2} \left( \fr{e_v^{1/2}}{e_R^{1/2} N} \right)^{1/2} \Xi^{-1} e_v^{-1/2}
\label{eq:tau:def}
}
in order to optimize between the estimates for the leading term in \eqref{eq:highFreqEstimate} and \eqref{eq:transTermEstimate}.  This choice leads to the $C^0$ estimate 
\ali{
 \co{R_1} &\unlhd \left( \fr{e_v^{1/2}}{e_R^{1/2} N} \right)^{1/2} \fr{e_v^{1/2} e_R^{1/2}}{N}  \label{eq:estimateWellProve}
}
stated in Lemma \ref{lem:mainLemma}, and ultimately to the regularity $1/9-$.

Unlike the case of the Euler equations, there is also a second error term which requires sharp time cutoffs to make small in our present scheme, namely the Stress term $R_S$ appearing in \eqref{eq:stressErrorInPieces}.  It turns out that this term also satisfies the same estimate \eqref{eq:highFreqEstimate}, and consequently will be among the largest error terms, having size \eqref{eq:estimateWellProve} after the above choice of $\tau$.  The reason that we see this extra term compared to the case of Euler is that the method we have used here to solve the quadratic equation \eqref{eq:equationForAmplitudes} requires the phase gradients $\nab \xi_I$ to remain very close to their initial values $\nab \xi_I \approx \nab \hat{\xi}_I$ to within an error much smaller than $O(1)$.  In the case of Euler, the equation analogous to \eqref{eq:shiftedAlgebraEquation} can be solved using nonlinear phase functions in a way which allows for the phase gradients to depart from their initial values by an error of size $\co{ \nab \xi_I - \nab \hat{\xi}_I } = O(1)$ (see \cite[Section 7.3]{isett}).  Ideally, one would hope to solve equation \eqref{eq:shiftedAlgebraEquation} in a similar manner to avoid generating error terms such as these which require sharp time cutoffs to treat as above.  




We now turn our attention to obtaining estimates for the terms in the construction.  In particular, we need to establish the estimates \eqref{eq:highFreqEstimate} and \eqref{eq:transTermEstimate} precisely, and also to estimate the other error terms.  The proof is concluded by choosing the constant $B_\la$ in \eqref{eq:lambdaDef} to be sufficiently large so that the inequality \eqref{eq:estimateWellProve} holds as stated, without any implicit constant factor.

\section{Basic Estimates for the Construction} \label{sec:estimatesCorrections}

\begin{lem}[Coarse Scale Flow Estimates]
\label{lem:coarse:scale}
Let $L\geq 2$ be an integeras in Lemma~\ref{lem:mainLemma}. The mollified velocity field $u_\ep$ defined in \eqref{eq:mollifiedu} obeys the estimates
\ali{
\| \nabla^k u_{\epsilon} \|_{C^0} & \leq C_k \Xi^k e_v^{1/2} N^{(k-L)_+/L}, \quad k \geq 1, \label{eq:mollified:u:pure:gradients}\\
\| \nabla^k (\partial_t + u_\ep \cdot \nabla) u_{\ep} \|_{C^0} & \leq C_k \Xi^{k+1} e_v N^{(k+1-L)_+/L}, \quad k \geq 0 \label{eq:mollified:u:convective:gradients}
}
for some universal positive constants $C_k$.
\end{lem}
\begin{proof}
For $k \leq L$, we see that \eqref{eq:mollified:u:pure:gradients} holds in view of the iterative assumption \eqref{bound:nabkth}. For $k>L$, there is an additional cost of $\ep_u^{(k-L)_+} = (B^{-1} N^{1/L} \Xi)^{(k-L)_+}$, where we have used the choice of $\ep_u$ in~\eqref{eq:epParamChoice}.

In order to prove \eqref{eq:mollified:u:convective:gradients}, we recall that $u_\ep^j  = P_{\leq q}^2 u^j$, where $2^{-q} = \ep_u=B^{-1} N^{1/L} \Xi$. We then have
\ali{
P_{\leq q}^2 (\pr_t u + u \cdot \nabla u) = ( \pr_t u_{\ep} + u_{\ep} \cdot \nabla u_{\ep}) - Q_{\ep}(u,u), \label{eq:commuteAdvec}
}
where
\ali{
Q_{\ep}^j(u,u)
&= u_\ep^i \pr_i u_\ep^j - P_{\leq q}^2(u^i \pr_i u^j) \\
&= [P_{\leq q}^2 u^i \pr_i, P_{\leq q}] (P_{\leq q} u^j) + P_{\leq q} \Big( [u^i \pr_i, P_{\leq q}] u^j \Big) -   P_{\leq q} \Big( (u^i - P_{\leq q}u^i) \pr_i (P_{\leq q} u^j) \Big). 
\label{eq:Q:u:u}
}
The estimate
\ali{
\|Q_{\ep}(u,u) \|_{C^0} \leq C \ep_u \Xi^2 e_v \leq C N^{-1/L} \Xi e_v
}
follows from \eqref{eq:Q:u:u} precisely as in \cite[Section~16]{isett}, by appealing to \eqref{bound:nabkth}. The decomposition \eqref{eq:Q:u:u} of the quadratic commutator term is convenient since it allows one to estimate without additional complications the higher order derivatives $\nabla^k Q_{\ep}(u,u)$. Derivatives up to order $L-1$ each introduce a factor of $\Xi$, while past that order the derivatives fall on the mollifier $P_{\leq q}$ and the cost per derivative is a constant multiple of $\Xi N^{1/L}$. Combining with
\ali{
\| \nabla^k P_{\leq q}^2 (\partial_t u^j + u^i \pr_i u^j) \|_{C^0} \leq C_k \Xi^{k+1} e_v N^{(k+1-L)_+/L}
}
which follows from the definition of $q$ and \eqref{bound:nabkdtu}, the proof of \eqref{eq:mollified:u:convective:gradients} is completed.
\end{proof}

\begin{lem}[Commutator Estimates]\label{lem:commutatorEstimates}
Let $D \geq 1$ and let $Q$ be a convolution operator
\[ Q f(x) = \int_{\R^2} f(x-h) q(h) dh \]
whose kernel $q$ satisfies the estimates
\ALI{
\| ~|\nab^k q|(h)~ \|_{L^1(\R^2)} + \| |h| ~| \nab^{1+k} q|(h)~ \|_{L^1(\R^2)} \leq \la^k
}
for some $\lambda \geq N \Xi$, and all $0 \leq |k| \leq D$.  Then the commutator $\left[\Ddt, Q \right] = \left[\pr_t + u_\epsilon \cdot \nab,Q\right]$ satisfies the estimates
\ALI{
 \left\| \nab^k \left[ \Ddt, Q\right] \right\|  &\leq C_k \Xi e_v^{1/2} \la^k, \qquad 0 \leq k \leq D - 1
}
as a bounded operator on $C^0(\R \times \T^2)$.
\end{lem}
In fact, the above lemma will only be applied to operators $Q$ for which $\lambda$ is given as in \eqref{eq:lambdaDef}.

\begin{lem}[Transport Estimates]
\label{lem:transport:estimates}
Let $L\geq 2$, and denote by $\Ddt$ the convective derivative associated to the flow $u_\ep$. The phase gradients $\nabla \xi_I$ obey the bound
\ali{
\|\nabla^k \left( \Ddt \right)^{r} \nabla \xi_I\|_{C^0} \leq C_k \Xi^k (\Xi e_v^{1/2})^r N^{(k+(r-1)_+ + 1 - L)_+/L} 
\label{eq:transport:estimates}
}
for all $k \geq 1$ and $r \in \{0,1,2\}$. Moreover, the same bound holds if $\nabla^k (\Ddt)^r$ is replaced by $D^{(k,r)}$, where the latter is defined by
\ali{
D^{(k,r)} = \nabla^{k_1} \left(\Ddt\right)^{r_1} \nabla^{k_2} \left(\Ddt\right)^{r_2} \nabla^{k_3},
}
with $k_1+k_2+k_3=k$, $k_i \geq 0$, $r_1+r_2=r$, and $r_i \geq 0$.

We also have the estimate
\ali{
| \nab \xi_I( &\Phi_s(t,x)) - \nab \hat{\xi}_I(x) | \leq C b \label{eq:closeToInitDat} \\
b &= B_\la^{-1/2} \left( \fr{e_v^{1/2}}{e_R^{1/2} N} \right)^{1/2} = \tau \Xi e_v^{1/2} \notag
}
where $\Phi_s$ is the coarse scale flow defined in~\ref{defn:coarseScaleFlow}, and $\tau$ is specified as in line~\eqref{eq:tau:def}.
\end{lem}

\begin{proof}
In order to establish \eqref{eq:transport:estimates} for $r=0$, one appeals to \eqref{eq:transportForPhaseGrad}, and obtains
\ali{
(\pr_t + u_\ep^j \pr_j) \nabla^k \pr^l \xi_I &= - \nabla^k(\pr^l u_\ep^j \pr_j \xi_I) + [u_\ep^j \pr_j, \nabla^k] \pr^l \xi_I.
}
The bound for $r=0$ then follows from the Gr\"onwall inequality in the above identity, estimate \eqref{eq:mollified:u:pure:gradients}, and the choice for $\tau$ in \eqref{eq:tau:def}. Similarly, from 
\ali{
\nabla^k  (\pr_t + u_\ep^j \pr_j) \pr^l \xi_I &= - \nabla^k(\pr^l u_\ep^j \pr_j \xi_I)
}
and \eqref{bound:nabkth} we obtain the estimate \eqref{eq:transport:estimates} with $r=1$.

Lastly, in order to obtain the desired estimate when $r=2$ we note that
\ali{
\left( \Ddt \right)^2 \pr^l \xi_I = (\pr_t + u_\ep^i \pr_i)^2 \pr^l \xi_I
&= - (\pr_t + u_\ep^i \pr_i) (\pr^l u_\ep^j \; \pr_j \xi_I)\\
&= - \pr_j \xi_I (\pr_t + u_\ep^i \pr_i) (\pr^l u_\ep^j) - \pr^l u_\ep^j  (\pr_t + u_\ep^i \pr_i) (\pr_j \xi_I) \\
&= - \pr_j \xi_I \pr^l (\pr_t + u_\ep^i \pr_i) u_\ep^j  +2 \pr^l u_\ep^j  \; \pr_j u_\ep^i \pr_i \xi_I.
\label{eq:second:D:grad:xi}
}
In particular, it is important that the second convective derivative of $\nabla \xi_I$ only depends on a single convective derivative of $u_\ep$. By appealing to Lemma~\ref{lem:coarse:scale}, from  \eqref{eq:second:D:grad:xi} we obtain that
\ali{
\| \left(  \Ddt \right)^2 \nabla \xi_I \|_{C^0} \leq C \Xi^2 e_v.
}
The bound \eqref{eq:transport:estimates} with $r=2$ and $k \geq 1$, similarly follows from \eqref{eq:second:D:grad:xi}, the Leibniz rule, Lemma~\ref{lem:coarse:scale}, and estimate \eqref{eq:transport:estimates} with $r=0$.

The estimate \eqref{eq:closeToInitDat} follows from the bound \eqref{eq:transport:estimates} with $k = 0$ and $r = 1$, from the calculation
\ALI{
 | \nab \xi_I( \Phi_s(t,x)) - \nab \hat{\xi}_I(x) | &\leq \int_0^s \Big| \Ddt \nab \xi_I( \Phi_\si(t,x) ) \Big | d\si \\
&\leq C |s| \Xi e_v^{1/2} \leq C b, \qquad \tx{if } |s| \leq \tau
}
\end{proof}

\begin{lem}[Principal Amplitude estimates]\label{lem:principAmpEst}
Let $L\geq 2$ and $\tau$ be defined in \eqref{eq:tau:def}. Then the principal parts of the scalar amplitude $\theta_I$, and the velocity amplitude $u_I$, obey the bounds
\ali{
\| D^{(k,r)} \theta_I \|_{C^0} + \| D^{(k,r)} u_I \|_{C^0} &\leq C_k \Xi^k e_R^{1/2} \tau^{-r} N^{(k+1-L)_+/L} \label{eq:princAmpEst}
}
for all $k\geq 0$ and $r\in\{0,1,2\}$, for some suitable universal constants $C_k >0$.
\end{lem}
\begin{proof}
First, we note that in view of \eqref{eq:uIexpression} we have
$
u_I^l = \theta_I m^l(\nabla \xi_I).
$
Since the multiplier $m$ is smooth outside the origin and in view of Lemma~\ref{lem:transport:estimates} we have bounds for the derivatives of $\nabla \xi_I$, the bound on $u_I$ follows from that on $\theta_I$, up to possibly increasing the constant $C_K$ by a constant factor.

From \eqref{eq:varepTerm} and \eqref{eq:formForThi} we recall that 
\ali{
\theta_I = \eta\left( \frac{t-k\tau}{\tau} \right) e(t)^{1/2} \gamma = \eta\left( \frac{t-k\tau}{\tau} \right) e(t)^{1/2} \left(1 + \varepsilon \right)^{1/2} \label{eq:expandThI}
} 
where $\varepsilon = (\tilde{c}_A + c_J)/e(t)$.  Using \eqref{eq:tilde:cA:bounds}, the lower bound $e(t) \geq K_0 e_R$ and \eqref{eq:whichK0weneed}, we obtain the following estimates for $\varepsilon$ and $\gamma = (1 + \varepsilon)^{1/2}$
\begin{align}
\co{ \nabla^k \left(\Ddt \right)^r \varepsilon } + \co{\nabla^k \left( \Ddt \right)^r \ga}  &\leq C_k \Xi^k e_R (\Xi e_v^{1/2} )^{(r\geq 1)} (N \Xi e_R^{1/2})^{(r\geq 2)} N^{(k+1-L)_+/L}.
\label{eq:gaVarep:bounds}
\end{align}
The bounds for spatial derivatives of $\th_I$ now follow from \eqref{eq:gaVarep:bounds} since the other terms $\eta\left( \frac{t-k\tau}{\tau} \right) $ and $e^{1/2}(t)$ do not depend on $x$.  Lemma~\ref{lem:principAmpEst} requires us also to show that that each advective derivative up to order $2$ costs at most $C \tau^{-1}$ per derivative.  For the time cutoff and the function $e^{1/2}(t)$ in \eqref{eq:expandThI}, the cost of $\tau^{-1}$ follows by definition for the time cutoff and by \eqref{ineq:goodEnergy} for $e^{1/2}(t)$ using the inequality $\Xi e_v^{1/2} \leq \tau^{-1}$ from the choice of $\tau$ in Section~\ref{sec:lifespanChoice}.  For the other terms, the estimate \eqref{eq:gaVarep:bounds} tells us that the first advective derivative costs $\Xi e_v^{1/2} \leq \tau^{-1}$, and taking two advective derivatives costs
\[ \left| \DDdt \right| \leq C (\Xi e_v^{1/2})(N \Xi e_R^{1/2}) = C N \Xi^2 e_v^{1/2} e_R^{1/2} = C \tau^{-2} \]
from the choice of $\tau$ in \eqref{eq:tau:def}.  The bounds for the spatial derivatives then follow from the pattern in \eqref{eq:gaVarep:bounds}.

\end{proof}

\begin{lem}[Amplitude correction estimates]\label{lem:ampCorrectEst}  Under the hypotheses of Lemma~\ref{lem:principAmpEst}, the corrections $\de \th_I$ and $\de u_I^l$ to the scalar amplitude and the velocity amplitude obey the bounds
\ali{
\| D^{(k,r)} \de \theta_I \|_{C^0} + \| D^{(k,r)} \de u_I \|_{C^0} &\leq C_k B_\la^{-1} N^{-1} \Xi^k e_R^{1/2} \tau^{-r} N^{(k+2-L)_+/L} \label{eq:correctionBoundsAre}
}
for all $k\geq 0$ and $r\in\{0,1,2\}$, for some suitable universal constants $C_k >0$.
\end{lem}
\begin{proof}  These estimates are obtained by explicitly differentiating the formulas for $\de \th_I$ and $\de u_I^l$ coming from the Microlocal Lemma, Lemma~\ref{lem:microlocal}.  Here we carry out the calculation for the case of $\de \th_I$, since the term $\de u_I^l$ can be treated in the same way.  Recall that
\[ \Th_I = P_{\approx \la}^I( e^{i \la \xi_I} \th_I) = e^{i \la \xi_I}( \th_I + \de \th_I ) \]
Applying Lemma~\ref{lem:microlocal} with $K(h) = \eta_{\approx \la}^I(h)$, we have the following formula for $\de \th_I$
\ali{
\de \th_I &= \de \th_{I,1} - \de \th_{I,2} \notag \\
\de \th_{I, 1} &= \int_0^1 dr \int e^{-i \la \nab \xi_I(x) \cdot h} e^{i Z(r,x,h)} (i \la) \left[\int_0^1 h^a h^b \pr_a\pr_b \xi_I(x-sh) (1-s) ds\right] \th_I(x - rh) \eta_{\approx \la}^I(h) dh \label{eq:deThIone} \\
\de \th_{I,2} &=  \int_0^1 dr \int e^{-i \la \nab \xi_I(x) \cdot h} e^{i Z(r,x,h)} \pr_a \th_I(x - rh) h^a \eta_{\approx \la}^I(h) dh
}
with $Z(r,x,h) = r\la \int_0^1 h^a h^b \pr_a\pr_b \xi(x-sh) (1-s) ds$ and where $\eta_{\approx \la}^I$ is defined after line \eqref{eq:ThIisLike}.  In particular, recall that the kernel $\eta_{\approx \la}^I(h) = 10^{2[k]}\la^{2} \eta_{\approx 1}(\pm 10^{[k]} \la h)$ is constructed by rescaling a Schwartz kernel by a factor $\la$, and therefore satisfies the estimates
\ali{
 \| |h|^m \eta_{\approx \la}^I \|_{L^1_h} &\leq C_m \la^{-m} \label{ineq:rescaledKerBounds}
}
Combining the estimate \eqref{ineq:rescaledKerBounds} with the bounds of Lemma~\ref{lem:principAmpEst} and Lemma~\ref{lem:transport:estimates} gives the $C^0$ estimate for $\nab^k \de \th_I$.

Proving estimates for advective derivatives of $\de \th_I$ is tedious, but straightforward.  To ease notation let us write $Z(r,x,h) = r \la h^a h^b Z_{ab}$ where
\[ Z_{ab} = Z_{ab}(t,x,h) =  \int_0^1 \pr_a\pr_b \xi_I(x-sh) (1-s) ds \]
  We will sketch one example and estimate the advective derivative of the term in \eqref{eq:deThIone}.
\ali{
(\pr_t &+ u_\ep \cdot \nab) \de \th_{I,1}(t,x) = - i T_{(1)} + i  T_{(2)} + T_{(3)} + T_{(4)} \\
T_{(1)} &=  \int_0^1 dr \int e^{-i \la \nab \xi_I(x) \cdot h} e^{i Z} (i \la)  h^a h^b Z_{ab} \th_I(x - rh) \Ddt \pr_c \xi_I(x) \la h^c \eta_{\approx \la}^I(h) dh \label{eq:easiestAdvecMicrolocal} \\
T_{(2)} &= \int_0^1 dr \int e^{-i \la \nab \xi_I(x) \cdot h} e^{i Z} (i \la) h^a h^b Z_{ab} \th_I(x - rh) r \left(\pr_t + u_\ep^i(x) \fr{\pr}{\pr x^i}\right) Z_{ab} \la h^a h^b \eta_{\approx \la}^I(h) dh \label{eq:secondAdvecMicroloc} \\
T_{(3)} &=  \int_0^1 dr \int e^{-i \la \nab \xi_I(x) \cdot h} e^{i Z} (i \la) h^a h^b  \left(\pr_t + u_\ep^i(x) \fr{\pr}{\pr x^i}\right)Z_{ab} \th_I(x - rh) \eta_{\approx \la}^I(h) dh \\
T_{(4)} &= \int_0^1 dr \int e^{-i \la \nab \xi_I(x) \cdot h} e^{i Z} (i \la) h^a h^b Z_{ab} \left(\pr_t + u_\ep^i(x) \fr{\pr}{\pr x^i}\right) \th_I(x - rh) \eta_{\approx \la}^I(h) dh \label{eq:lastAdvecMicrolocal}
}
The pattern we observe in \eqref{eq:easiestAdvecMicrolocal}-\eqref{eq:lastAdvecMicrolocal} is that the cost of the first advective derivative is given by $\Xi e_v^{1/2}$ for every term.  This cost is most clear for the term \eqref{eq:easiestAdvecMicrolocal}.  The advective derivative brings down one term of size
\[ \co{\Ddt \nab \xi_I } \leq C \Xi e_v^{1/2} \]
and also introduces the factor $\la h^c$.  The $\la$ and the $h$ cancel out in terms of the estimate, since we gain a $\la^{-1}$ when we apply the bound
\[ \| h^a h^b h^c \eta_{\approx \la}^I(h) \|_{L^1_h} \leq C \la^{-3}  \]
for the kernel, which comes from scaling.

The terms \eqref{eq:secondAdvecMicroloc}-\eqref{eq:lastAdvecMicrolocal} require one more trick, which is to approximate the value of $u_\ep^i(x)$ with the nearby point in the integral.  For example, for the term $ \left(\pr_t + u_\ep^i(x) \fr{\pr}{\pr x^i}\right) \th_I(x - rh)$ in \eqref{eq:lastAdvecMicrolocal} we write
\ali{
\left(\pr_t + u_\ep^i(x) \fr{\pr}{\pr x^i}\right) \th_I(x - rh) &= \Ddt \th_I(x - rh) + ( u_\ep^i(x) - u_\ep^i(x - r h) ) \pr_i \th_I(x - r h)
}
The cost of $\Xi e_v^{1/2}$ for the advective derivative on $\th_I$ follows from Lemma~\ref{lem:principAmpEst}.  For the latter term, we write
\ali{
( u_\ep^i(x) - u_\ep^i(x - r h) ) \pr_i \th_I(x - r h) &= - r \int_0^1 \pr_c u_\ep^i( x - \si r h) d\si~ \pr_i \th_I(x - r h) h^c \label{eq:advecErrorTerm}
}
The term where $\pr_c u_\ep^i$ appears accounts for the cost of  $\co{\nab u_\ep} \leq \Xi e_v^{1/2}$.  The derivative hitting $\th_I$ costs a factor of $\Xi$, but this factor is regained by the factor $h^c$ that has appeared, which gains a $\la^{-1}$ when combined with the kernel as usual.  Repeating this observation many times for each one of \eqref{eq:secondAdvecMicroloc}-\eqref{eq:lastAdvecMicrolocal}, one obtains the first advective derivative bound in \eqref{eq:correctionBoundsAre}.  We omit the details.

One also has to take a second advective derivative in order to prove \eqref{eq:correctionBoundsAre}, giving rise to another long series of terms which obey the correct bounds.  We omit the proof of this estimate also, but we remark that one can avoid using these bounds during the course of the proof.  The only applications of these bounds are in Section~\ref{sec:AdvecBoundsStress} for a lower order part of the advective derivative of the transport term, and in this case one can substitute second order commutator estimates as in Lemma~\ref{lem:commutatorEstimates}, which are somewhat less tedious to write down.

\end{proof}

\begin{cor}\label{cor:tildeBounds} The bounds \eqref{eq:princAmpEst} of Lemma~\ref{lem:principAmpEst} hold also for $\tilde{\th}_I = \th_I + \de \th_I$ and for $\tilde{u}_I^l = u_I^l + \de u_I^l$.
\end{cor}

\subsection{Estimates for the Corrections to the Scalar Field and Drift Velocity}
In this Subection, we obtain estimates for the corrections $\Th$ and $U^l = T^l[\Th]$ to the scalar field and drift velocity.  These bounds confirm that the estimates \eqref{eq:Vco}-\eqref{eq:matWco} of Lemma~\ref{lem:mainLemma} are satisfied.  As with the previous Lemmas~\ref{lem:coarse:scale}-\ref{lem:ampCorrectEst} and our choices of parameters, the results we obtain in this section are familiar from \cite[Section 22.1]{isett}.  In our setting, these estimates turn out to be a bit easier to check thanks to our use of frequency localizing projections.




\begin{prop}\label{prop:velocScalarCorrections}Under the hypotheses of Lemma~\ref{lem:principAmpEst}, the corrections $\Th_I$ and $U_I^l$ to the scalar field and the drift velocity satisfy
\ali{
\| D^{(k,r)} \Theta_I \|_{C^0} + \| D^{(k,r)} U_I \|_{C^0} &\leq C_k (B_\la N \Xi)^k \tau^{-r} e_R^{1/2} \label{eq:correctionEsts}
}
for $0 \leq r \leq 2$.
\end{prop}
\begin{proof}  We outline the proof of \eqref{eq:correctionEsts} for $\Th_I$, as the velocity field $U_I$ can be treated in the same way.  Here we recall again that 
\[ \Th_I = P_{\approx \la}^I [ e^{i \la \xi_I} \th_I ] = e^{i \la \xi_I} \tilde{\th}_I \]
For $r = 0$, the estimates for $\nab^k \Th_I$ follow from the bound $\co{\th_I} \leq C e_R^{1/2}$, and the definition of $\la$.  To estimate the advective derivatives, we write
\ali{
\Ddt \Th_I &= e^{i \la \xi_I} \Ddt \tilde{\th}_I \\
\DDdt \Th_I &= e^{i \la \xi_I} \DDdt \tilde{\th}_I 
}
The bounds \eqref{eq:correctionEsts} now follow from the bounds of Lemma~\ref{lem:transport:estimates} and Corollary~\ref{cor:tildeBounds}.  The main terms in the estimates for spatial derivatives arise in every case when the derivatives fall on $e^{i \la \xi_I}$.  Alternatively, one can obtain the same bounds using commutator estimates such as those of Lemma~\ref{lem:commutatorEstimates} extended to second order commutators.  Note that this latter approach avoids using the second advective derivative estimates proven in Lemma \ref{lem:ampCorrectEst}.
\end{proof}

Lemma~\ref{lem:mainLemma} also requires bounds on a vector field $W^l$ satisfying $\div W = \Th$.  To define $W^l$, first recall that the corrections
\[ \Th_I = P_{\approx \la}^I( e^{i \la \xi_I} \th_I ) \]
are frequency localized, which allows us to invert the divergence using the standard Helmholtz solution
\ali{
W_I^l &= \pr^l \De^{-1} P_{\approx \la}^I ( e^{i \la \xi_I} \th_I )
}
With this definition, we have $\Th = \div W$ for $W^l = \sum_I W_I^l$.  The bounds \eqref{eq:Wco}-\eqref{eq:matWco} of Lemma~\ref{lem:mainLemma} now follow from Lemma~\ref{lem:commutatorEstimates} and Lemma~\ref{lem:principAmpEst} by writing
\ali{
(\pr_t + u_\ep \cdot \nab) W_I &= \left[ \Ddt, \pr^l \De^{-1} P_{\approx \la}^I\right] (e^{i \la \xi_I} \th_I ) + \pr^l \De^{-1} P_{\approx \la}^I( e^{i \la \xi_I} \Ddt \th_I )
}
and differentiating in space.


\subsection{Prescribing the energy increment}
We conclude this Section by verifying the estimates \eqref{eq:energyPrescribed} and \eqref{eq:dtenergyPrescribed} for prescribing the energy increment.  To obtain the estimate \eqref{eq:energyPrescribed}, let $t \in \R$ and write
\ali{
\int_{\T^2} |\Th|^2(t,x) dx &= \sum_{I, J} \int \Th_I \cdot \Th_J(t,x) dx \label{eq:energyIncrement}
}
For indices $J \neq \bar{I}$ which are not conjugate to each other, the product $\Th_I \cdot \Th_J$ is localized at frequency $\approx \la$, and in particular has integral $0$.  The only remaining terms are
\ALI{
\int_{\T^2} |\Th|^2(t,x) dx &= \sum_{I} \int |\Th_I|^2(t,x) dx \\
|\Th_I|^2 &= | \th_I + \de \th_I |^2 = |\th_I|^2 + 2 \de \th_I \th_I + \de \th_I^2 
}
The terms involving $\de \th_I$ can all be estimated using Lemma~\ref{lem:ampCorrectEst} and Lemma~\ref{lem:principAmpEst}.
\ALI{
\sum_I \Big| \int_{\T^2} 2 \th_I ~ \de \th_I  + (\de \th_I)^2 dx \Big| &\leq C \fr{e_R}{B_\la N} 
}
The main terms are then given by
\ALI{
\sum_I \int_{\T^2} |\th_I|^2(t,x) dx &= \sum_I \int \eta_{k}^2(t) e(t) \ga^2 dx \\
&= 2 \int e(t) \ga^2 dx \\
&= 2 \int_{\T^2} e(t) (1 + \varepsilon) dx
}
The bound \eqref{eq:energyPrescribed} now follows from \eqref{eq:whichK0weneed} provided $B_\la$ is sufficiently large.  

In order to obtain the estimate \eqref{eq:dtenergyPrescribed}, we differentiate \eqref{eq:energyIncrement} with respect to $t$, and use the fact that the coarse scale velocity field $u_\ep$ is divergence free
\ALI{
\fr{d}{dt} \int_{\T^2} |\Th|^2(t,x) dx &= \sum_{I, J} \int_{\T^2} (\pr_t + u_\ep \cdot \nab) \Th_I \cdot \Th_J(t,x) dx 
}
At this point, we again observe that the terms $(\pr_t + u_\ep \cdot \nab) \Th_I \cdot \Th_J$ are localized in frequency space at frequencies of order $\la$ for all nonconjugate indices $J \neq \bar{I}$.  These terms therefore integrate to $0$ and we are left with
\ALI{
\fr{d}{dt} \int_{\T^2} |\Th|^2(t,x) dx &= \sum_{I} \int_{\T^2} (\pr_t + u_\ep \cdot \nab) |\Th_I|^2(t,x) dx \\
&= \sum_I \int_{\T^2}  (\pr_t + u_\ep \cdot \nab) | \tilde{\th}_I |^2 dx
}
The bound \eqref{eq:dtenergyPrescribed} now follows from Corollary \ref{cor:tildeBounds}.

\subsection{Checking frequency energy levels for the scalar field and drift velocity}

The statement \eqref{eq:theNewEnergyLevel} in Lemma~\ref{lem:mainLemma} requires us to prove that the new scalar field and drift velocity $\th_1 = \th + \Th$, $u_1^l = u^l + U^l$ satisfy the bounds \eqref{bound:nabkth}-\eqref{bound:nabkdtu} for the new compound frequency energy levels $(\Xi', e_v', e_R', e_J') = (C N \Xi, e_R, K_1 e_J, e_J')$ with
\[ e_J' = \left(\fr{e_v^{1/2}}{e_R^{1/2}N} \right)^{1/2} e_R \]
The bounds in \eqref{bound:nabkth} already follow from the arguments in \cite[Section 22]{isett}, as the scalar field $\th$ and drift velocity $u^l$ both share the same estimates as the coarse scale velocity $v^l$ in that paper, and the corrections $\Th$ and $U^l$ both share the same estimates at the corrections $V^l$ in that paper.  The only new point here is how we establish the bound
\ali{
 \co{ (\pr_t u_1 + u_1 \cdot \nab) u_1} &\unlhd (\Xi' e_v') = C N \Xi e_R \label{eq:needNewAdvec}
}
This estimate, which is quadratic in the velocity, is analogous to the bound for the pressure gradient in the case of Euler.

The idea is to use the assumed bound \eqref{bound:nabkdtu} and write
\ali{
(\pr_t u_1 + u_1 \cdot \nab) u_1 &= (\pr_t u + u \cdot \nab u) + U \cdot \nab u + (\pr_t + u \cdot \nab) U \label{eq:}
}
In the case of Euler, the first term $(\pr_t u + u \cdot \nab u)$ can be bounded using the Euler-Reynolds equations.  In our case, though, the bound \eqref{eq:needNewAdvec} on the advective derivative cannot be obtained from commuting the operator $T^l$ with the compound scalar stress equation due to the lack of $C^0$ boundedness of $T^l$, and arguments involving frequency truncations still give logarithmic losses.

The idea is that we have already assumed the bound $\co{(\pr_t u + u \cdot \nab u)} \leq \Xi e_v$, so that \eqref{eq:needNewAdvec} follows from the condition $N \geq \left(\fr{e_v}{e_R}\right)$.  Also, further advective derivatives can be estimated at a cost smaller than $N \Xi$ per derivative up to order $L - 1$, giving \eqref{bound:nabkdtu} for this term.  The proof of \eqref{bound:nabkdtu} for the other two terms is the same as in \cite[Section 22]{isett}.  The main idea is to write $(\pr_t + u \cdot \nab) = (\pr_t + u_\ep \cdot \nab) + (u - u_\ep) \cdot \nab$, and then to apply the relevant bounds established earlier on in Sections~\ref{sec:specifyingParams}-\ref{sec:estimatesCorrections}.

\section{Estimates for the New Stress} \label{sec:boundNewStress}
In this Section, we conclude the proof of Lemma~\ref{lem:mainLemma} by establishing estimates for the error terms contributing to the new stress field which were derived in Section~\ref{sec:theErrorTerms}.  Recall from that section that the new scalar field $\th_1 = \th + \Th$ and the new drift velocity $u_1^l = u^l + U^l$ satisfy the compound scalar stress equation
\ali{
\label{eq:scalarStressB2}
\left \{
\begin{aligned}
\pr_t \th_1 + \pr_l(\th_1 u_1^l) &= \pr_l( c_B B^l + R_1^l)  \\
u_1^l &= T^l[\th_1] 
\end{aligned}
\right.
}
The function $c_B$ is defined in \eqref{eq:decomposeRepComponents}, and the new stress field has the form
\ali{
R_1^l &= R_T + R_L + R_H + R_M + R_S \label{eq:recallNewStress}
}
as in \eqref{eq:allTheErrors}.  For these error terms, the Main Lemma requires us to show that the bounds of Definition~\ref{def:freqEnDef} are satisfied for the compound frequency energy levels $(\Xi', e_v', e_R', e_J' )$ specified in \eqref{eq:theNewEnergyLevel}.  Our starting point will be to prove the $C^0$ estimates
\ali{
\co{ c_B } &\unlhd K_1 e_J \label{eq:coBoundcB}\\
\co{R_1} &\unlhd e_J' \label{eq:allTheStressC0} \\
e_J' &= \left(\fr{e_v^{1/2}}{e_R^{1/2}N} \right)^{1/2} e_R
}
We will obtain these bounds in Section~\ref{sec:coBounds}, at which point we will finally specify the large constant $B_\la$ appearing in line \eqref{eq:lambdaDef} where $\la$ is defined.

Once the $C^0$ estimates are established and $B_\la$ is chosen, the bounds on spatial derivatives
\ali{
\co{ \nab^k c_B } &\unlhd C (N \Xi)^k K_1 e_J \qquad k = 0, \ldots, L \\
\co{ \nab^k R_1 } &\unlhd C (N \Xi)^k e_J', \qquad k = 0, \ldots, L
}
will be clear, and we will also need to verify the estimates for the advective derivatives
\ali{
\co{ \nab^k (\pr_t + u_1 \cdot \nab) c_B } &\unlhd C (N \Xi)^k (N \Xi e_R^{1/2}) K_1 e_J \label{eq:wantAdvecCb} \\
\co{ \nab^k (\pr_t + u_1 \cdot \nab) R_1 } &\unlhd C (N \Xi)^k (N \Xi e_R^{1/2}) e_J' \label{eq:wantAdvecR1} \\
k &= 0, \ldots, L-1 \notag
}
These bounds will be checked in Sections~\ref{sec:spatialBoundsStress} and ~\ref{sec:AdvecBoundsStress}, which will conclude the proof of Lemma~\ref{lem:mainLemma}.

\subsection{The \texorpdfstring{$C^0$}{C0} bounds} \label{sec:coBounds}

In this Section, we establish the $C^0$ bounds \eqref{eq:coBoundcB}-\eqref{eq:allTheStressC0}.  The bound \eqref{eq:allTheStressC0} will be obtained only after the constant $B_\la$ of line \eqref{eq:lambdaDef} is chosen sufficiently large.

First, observe that the estimate \eqref{eq:coBoundcB} for $c_B$ follows immediately from line \eqref{eq:decomposeRepComponents} where $c_B$ is defined, and the bound $\co{ R_\ep } \leq \co{ R_J }$.  Note that the constant $K_1$ depends only on the operator $T^l$.

It now remains to estimate the stress terms appearing in \eqref{eq:recallNewStress}.  We estimate each of these in turn.

\paragraph{The mollification term $R_M^l$.}

We recall from~\eqref{eq:mollifyRappears} and \eqref{eq:RMl:mollified} that
\begin{align}
R_M^l &= (u^l - u_\ep^l) \Th + (\th - \th_\ep)U^l + (c_A - \tilde{c}_A) A^l + (R_J^l - R_\ep^l) \\
&= R_{M,\th}^l + R_{M,u}^l + (c_A - \tilde{c}_A) A^l + (R_J^l - R_\ep^l) + R_{M'}^l. \label{eq:listOfMollTermsFinal}
\end{align}
Note that by the choice of $B$, from~\eqref{eq:ourFirstMollifyGoal} and \eqref{eq:pickingLengthTimeScales} we have that 
\ali{
\co{R_{M,\theta}} + \co{R_{M,u}} + \co{(c_A - \tilde{c}_A) A^l} + \co{(R_J^l - R_\ep^l) } \leq \frac{e_v^{1/2} e_R^{1/2}}{50 N} = \fr{1}{50} \left( \fr{e_v^{1/2}}{e_R^{1/2} N} \right) e_R \label{eq:weAlreadyBoundedMoll}
}
The factor $\left( \fr{e_v^{1/2}}{e_R^{1/2} N} \right)$ is less than $1$, so this estimate is more than enough to achieve the bound \eqref{eq:allTheStressC0}.  

To estimate the remaining term
\ali{
R_{M'}^l &= \sum_I e^{i \la \xi_I} (u^l - u_\ep^l) \de \th_I + \sum_I e^{i \la \xi_I} (\th - \th_\ep) \de u_I^l \label{eq:lastMollTerm}
}
recall the estimates
\ali{
\co{  (u^l - u_\ep^l)  } + \co{ (\th - \th_\ep ) } &\leq \fr{e_v^{1/2}}{N} \leq e_v^{1/2} \\
\co{ \de \th_I } + \co{ \de u_I^l } &\leq C \fr{e_R^{1/2}}{B_\la N}
}
from Section~\ref{sec:specifyingParams} and Lemma~\ref{lem:ampCorrectEst} (in fact the estimates for the terms $(\th - \th_\ep)$ and $(u - u_\ep)$ are even better).  Note also that, at any given time $t$, at most four indices $I$ contribute to the sum in \eqref{eq:lastMollTerm}.

For sufficiently large values of $B_\la$, we therefore obtain
\[ \co{ R_{M'}^l } \leq \fr{1}{50} \fr{e_v^{1/2} e_R^{1/2}}{N} \]
which is sufficient for \eqref{eq:allTheStressC0}.

\paragraph{The Stress term $R_S$.}
To estimate $R_S$, let us recall from \eqref{eq:stressErrorInPieces} that we can express
\ali{
R_S^l &= \sum_I (R_{SI,1}^l + R_{SI,2}^l ) \label{eq:rememberStressTerm} \\
R_{SI,1}^l &= |\th_I|^2 [ (m^l(-\nab \xi_I) - m^l(-\nab \hat{\xi}_I)) + ( m^l(\nab \xi_I) - m^l(\nab \hat{\xi}_I) ) ] \label{eq:newLowFreqTerm} \\
R_{SI,2}^l &= \de \th_I \tilde{u}_{\bar{I}}^l + \tilde{\th}_I \de u_{\bar{I}}^l - \de \th_I \de u_{\bar{I}}^l 
+ \de \th_{\bar{I}} \tilde{u}_{I}^l + \tilde{\th}_{\bar{I}} \de u_{I}^l - \de \th_{\bar{I}} \de u_{I}^l
}
The estimates of Lemma~\ref{lem:principAmpEst}, Lemma~\ref{lem:ampCorrectEst} and Corollary~\ref{cor:tildeBounds} give
\ALI{
\co{ R_{SI, 2} } &\leq \fr{C}{B_\la N} e_R
}
At any given time $t$, as most $4$ terms of the form $R_{SI, 2}$ are nonzero, which allows us to obtain the estimate
\ALI{
\co{ \sum_I | R_{SI,2} | } &\leq \fr{e_v^{1/2} e_R^{1/2}}{500 N} 
}
which is sufficient for \eqref{eq:allTheStressC0}, by taking the value of $B_\la$ sufficiently large.

We estimate the terms in \eqref{eq:newLowFreqTerm}  using \eqref{eq:closeToInitDat} and Lemma~\ref{lem:principAmpEst}, in order to obtain the bound
\ALI{
\co{ R_{SI,1}^l } &\leq \fr{C}{B_\la^{1/2}} \left(\fr{e_v^{1/2}}{e_R^{1/2}N} \right)^{1/2} e_R
}
By choosing the constant $B_\la$ sufficiently large, we obtain the bound
\ali{
\co{ \sum_I | R_{SI,1}^l | } &\leq \fr{1}{1000} e_J' 
}
where $e_J'$, as defined in \eqref{eq:allTheStressC0}, is our goal for the size of the new stress term $R_1^l$.

For the next stress terms, $R_L$ and $R_T$, we use that they are frequency localized between two constant multiples of $\lambda$, and thus we can appeal to the estimate
\ali{
\| \nab \De^{-1} P_{\approx \la} \|_{C^0 \to C^0} \leq C \la^{-1} = \fr{C}{N B_\la}.
\label{eq:solving:divergence:bound}
}

\paragraph{The High-Low term $R_L$.} We recall from \eqref{eq:highLowTermIs}
that
\[
R_L^l = \pr^l \Delta^{-1} P_{\approx \la} [ U^j \pr_j \theta_\epsilon]
\]
and thus
\[
\co{R_L} \leq \| \nab \De^{-1} P_{\approx \la} \|_{C^0 \to C^0} \co{U^j} \co{\pr_j \theta_\epsilon} \leq \fr{C}{N B_\la}  e_R^{1/2} e_v^{1/2}
\]
holds, upon appealing to \eqref{eq:solving:divergence:bound}. Choosing $B_\la$ sufficiently large, we see that
\[ \co{R_L} \leq \fr{1}{1000} \fr{e_v^{1/2} e_R^{1/2}}{N} \]
which is sufficient for \eqref{eq:allTheStressC0} to be satisfied.

\paragraph{The Transport term $R_T$.} We use \eqref{eq:RTdef} to recall that 
\[
R_T = \pr^l \Delta^{-1} P_{\approx \la}\left[ \Ddt \Theta\right].
\]
Thus, from \eqref{eq:solving:divergence:bound} and \eqref{eq:correctionEsts} we obtain
\ALI{ 
\co{ R_T } 
&\leq \frac{C}{\lambda} \tau^{-1} e_R^{1/2} \notag\\
&=\frac{C}{B_\lambda N \Xi } B_\la^{1/2} \left( \fr{e_v^{1/2}}{e_R^{1/2} N} \right)^{-1/2} \Xi e_v^{1/2} e_R^{1/2} \\
&= \fr{C}{B_\la^{1/2}} \left(\fr{e_v^{1/2}}{e_R^{1/2}N} \right)^{1/2} e_R 
}
in view of the choice of $\tau$ in \eqref{eq:tau:def}. Choosing $B_\la$ sufficiently large immediately shows  that
\[ \co{R_T } \leq \fr{1}{1000} e_J' \]
holds.

\paragraph{The High-Frequency Interference term $R_H$.} To conclude the $C^0$ stress estimates we recall from \eqref{eq:highFreqPartOfStressIs} and \eqref{eq:MIJfirst}--\eqref{eq:MIJlast} that
\ALI{
R_H &= \sum_{J \neq \bar I} i\lambda \pr^l \Delta^{-1}  P_{\approx \lambda} \\
&\qquad \Big[ e^{i \la (\xi_I + \xi_J)} \left( \th_J {\tilde \th}_I ( m^l(\nab \xi_J) - m^l(\nab \hat{\xi}_J) ) \pr_l \xi_I   +  \th_I {\tilde \th}_J m^l(\nab \hat{\xi}_J)( \pr_l \xi_I - \pr_l \hat{\xi}_I)+  \de u_J^l \pr_l \xi_I {\tilde \th}_I \right) \Big].
}
From \eqref{eq:solving:divergence:bound} we thus obtain
\ALI{
\co{R_H} 
&\leq C A e_R (\Xi e_v^{1/2} \tau) +  \frac{C}{B_\lambda N } e_R \\
&\leq \fr{C A }{B_\la^{1/2}} \left(\fr{e_v^{1/2}}{e_R^{1/2}N} \right)^{1/2} e_R + \frac{C}{B_\lambda N } e_R.
}
For all sufficiently large choices of $B_\la$, we finally have the estimate
\[ \co{R_H } \leq \fr{1}{1000} e_J'  \]
This error term is the last one, so the estimate \eqref{eq:allTheStressC0} will finally be satisfied for any sufficiently large choice of $B_\la \geq \overline{B_\la}$.  The only restriction now is that $\la = B_\la N \Xi$ in \eqref{eq:lambdaDef} must be a positive integer.  Since we assume $\Xi \geq 2$ in Definition~\ref{def:freqEnDef} and $N \geq 1$, an appropriate choice of $B_\la$ exists in the interval $B_\la \in [\overline{B_\la}, 2 \overline{B_\la}]$.  Our construction is now fully specified once such a value is chosen.   

%
%

\subsection{Spatial derivative bounds}\label{sec:spatialBoundsStress}

First we claim that 
\[
\co{\nabla^k c_B} \leq C_k (N \Xi)^k K_1 e_J
\]
For this purpose, recall the definition \eqref{eq:decomposeRepComponents} and the bound~\eqref{eq:tilde:cA:bounds}. 
This above estimate holds since we have already verified the $C^0$ estimate \eqref{eq:coBoundcB}, and each spatial derivatives costs no more than a factor 
\[ |\nab| \leq C N \Xi. \]

The stress terms $R_T, R_L$, and $R_H$ each are contain a frequency localizing operator $P_{\approx \la}$, so that again, each spatial derivative costs at most $CN \Xi$, since the constant $B_\la$ has now been fixed, in the previous subsection.

The term $R_M^l$ is treated in the same fashion as the mollified stress term in \cite[Section 25.1]{isett}. The main ideas is that comparing the bound
\[ \co{ u - u_\ep } \leq \fr{e_v^{1/2}}{N}\]
which had been used to establish \eqref{eq:weAlreadyBoundedMoll} in Section~\ref{sec:specifyingParams}, to the estimate
\[ \co{ \nab u} + \co{ \nab u_\ep } \leq C \Xi e_v^{1/2} \]
we notice the cost is at most $C N \Xi$ upon taking a spatial derivative.

The $R_S^l$ stress is treated by writing $R_S^l = \sum_I (R_{SI,1}^l + R_{SI,2}^l)$. The estimate for $R_{SI,2}$ follows from the bounds established in Lemmas~\ref{lem:principAmpEst}--\eqref{lem:ampCorrectEst}. To treat $R_{SI,1}$ we need to observe that the spatial derivative costs at most $N \Xi$ when it is applied to the difference $\nab \xi_I - \nab \hat{\xi}_I$.  Comparing the bounds of Lemma~\ref{lem:transport:estimates} and \eqref{eq:closeToInitDat}
\ALI{
\co{ \nab \xi_I - \nab \hat{\xi}_I } &\leq \left( \fr{e_v^{1/2} }{e_R^{1/2} N} \right)^{1/2} \\
\co{ \nab^2 \xi_I } \leq C \Xi
}
gives a cost of $|\nab| \leq C N^{1/2} \Xi$, which is smaller than the threshold $N \Xi$.  All further derivatives of this term cost at most $C N^{1/2} \Xi$ according to Lemma~\ref{lem:transport:estimates}.

\subsection{Advective Derivative bounds} \label{sec:AdvecBoundsStress}

We now proceed to establish the advective derivative bounds \eqref{eq:wantAdvecCb}-\eqref{eq:wantAdvecR1} for the new frequency energy levels, which is more subtle than the spatial derivative estimates due to the improved regularity of the advective derivative.  As observed in~\cite[Proposition 24.1 and Proposition 24.2]{isett}, note that it suffices to check the bounds for the coarse scale advective derivative $\Ddt = \pr_t + u_\ep \cdot \nab$ after we write
\[ \pr_t + u_1 \cdot \nab = (\pr_t + u_\ep \cdot \nab) + (u - u_\ep) \cdot \nab + U \cdot \nab. \]
Having established spatial derivative estimates on all our error terms, the the bounds for the two error spatial derivative terms follow from the results of Section~\ref{sec:spatialBoundsStress}, the already established estimates on spatial derivatives for $\co{\nab^k ( u - u_\ep ) }$ which follow from \eqref{eq:basicConvEstu}, and the bounds on $\co{ \nab^k U }$, which follow from Propositon~\ref{prop:velocScalarCorrections}.

Since each term has been estimated already in $C^0$ by the energy level $e_J'$, our goal at this point is to check that the advective derivative never costs any more than
\ali{
 \Big| \Ddt \Big| &\unlhd C N \Xi e_R^{1/2}  \label{ineq:bigGoalForAdvec}
}
compared the estimates that were used to obtain the $C^0$ bound.

For most terms, the advective derivative costs $\tau^{-1}$, so it is useful to observe that our goal is also implied by a bound of the type
\ali{
 \Big| \Ddt \Big| \leq C \tau^{-1} \label{eq:tauInvIsOK}
}
from the fact that $\tau^{-1} \leq N \Xi e_R^{1/2}$.  For terms involving the difference between the phase gradients and their initial values, the following Lemma stating the cost of differentiating $\nab \xi_I - \nab \hat{\xi}_I$ is helpful
\begin{lem}\label{lem:phaseGradDiffEst}  For $k \geq 0$ and $0 \leq r \leq 2$, we have the following bounds
\ali{
\co{ \nab^k \left(\Ddt\right)^r ( \nab \xi_I - \nab \hat{\xi}_I) } &\leq C_k (N\Xi)^k \tau^{-r} b \\
b &= \left(\fr{e_v^{1/2}}{e_R^{1/2}N} \right)^{1/2}
}
\end{lem}
Lemma~\ref{lem:phaseGradDiffEst} follows from Lemma~\ref{lem:transport:estimates} after checking the relationships of the parameters using the condition $N \geq \fr{e_v}{e_R}$.
\begin{cor}\label{cor:mPhaseGradDiff} The bounds in Lemma~\ref{lem:phaseGradDiffEst} hold also for
\[m^l(\nab \xi_I) - m^l(\nab \hat{\xi}_I) = ( \nab \xi_I - \nab \hat{\xi}_I) \int_0^1 \pr_a m^l\Big( (1-\si)\nab \hat{\xi}_I + \si \nab \xi_I \Big) d\si  \]
\end{cor}

With these bounds in hand, we can now quickly verify \eqref{ineq:bigGoalForAdvec}.  

\paragraph{The term $c_B$.}

The term $c_B$ inherits the estimates for $R_\ep$ from its definition in \eqref{eq:decomposeRepComponents}.  These bounds are no worse than the bounds stated for $R_J$ in \eqref{eq:tilde:cA:bounds} as long as one takes no more than $2$ advective derivatives and no more than $L$ total spatial or advective derivatives (see \cite[Section 18]{isett}).  As a consequence, we obtain \eqref{eq:wantAdvecCb}.

\paragraph{The mollification term $R_M^l$.}

The mollification term \eqref{eq:listOfMollTermsFinal} is handled in the same way as in \cite[Sections 25.1, 25.2]{isett}.  Among these estimates, the most subtle are the terms
\[ (u - u_\ep) \Th + (\th - \th_\ep) U \]
For the purposes of proving our result and the main theorem in \cite{isett}, these terms can be estimated separately as
\[ \left\co{ \left(\Ddt u - \Ddt u_\ep\right) \Th \right} \leq C \left(\left\co{\Ddt u\right} + \left\co{\Ddt u_\ep\right}\right)\co{\Th} \]
at the cost of requiring the condition $N \geq \left( \fr{e_v}{e_R} \right)^{3/2}$.  However, as discussed in \cite[Section 25.1]{isett}, it appears that a scheme aimed at proving $1/3$ regularity might require this term to be estimated more delicately.  A more delicate commutator estimate would allow us to require instead that $N \geq \left( \fr{e_v}{e_R} \right)$.

\paragraph{The stress term $R_S$}

For the term $R_S$, the cost \eqref{ineq:bigGoalForAdvec} is obtained for every term appearing in \eqref{eq:rememberStressTerm} using the estimates of Lemmas~\ref{lem:principAmpEst}-\ref{lem:ampCorrectEst} and Corollary~\ref{cor:tildeBounds} for the amplitudes, and using Lemma~\ref{lem:phaseGradDiffEst} and Corollary~\ref{cor:mPhaseGradDiff} for the terms involving differences of phase gradients.

\paragraph{The terms $R_T$, $R_L$ and $R_H$}

The commutator estimates of Lemma~\ref{lem:commutatorEstimates} and the use of frequency localized waves make it especially simple to estimate the terms obtained by solving the divergence equation.  We list these terms here.
\ali{
\Ddt R_T^l &= \left[\Ddt,\pr^l \Delta^{-1} P_{\approx \la}\right] \left[ \Ddt \Theta\right] + \pr^l \Delta^{-1} P_{\approx \la} \left[ \DDdt \Th \right] \\
\Ddt R_L^l &=  \left[\Ddt,\pr^l \Delta^{-1} P_{\approx \la}\right] \left[ U^j \pr_j \theta_\epsilon\right] + \pr^l \Delta^{-1} P_{\approx \la} \Ddt \left[ U^j \pr_j \theta_\epsilon\right] \\
\Ddt R_H^l &= \sum_{J \neq \bar I} \left[\Ddt,\pr^l \Delta^{-1}  P_{\approx \lambda}\right] r_{H, IJ} + \pr^l \Delta^{-1}  P_{\approx \lambda} \Ddt r_{H,IJ} \\
r_{H,IJ} &= (i\la) e^{i \la (\xi_I + \xi_J)} \left( \th_J {\tilde \th}_I ( m^l(\nab \xi_J) - m^l(\nab \hat{\xi}_J) ) \pr_l \xi_I \right) \\
&+ (i \la) e^{i \la(\xi_I + \xi_J)} \left(\th_I {\tilde \th}_J m^l(\nab \hat{\xi}_J)( \pr_l \xi_I - \pr_l \hat{\xi}_I)+  \de u_J^l \pr_l \xi_I {\tilde \th}_I \right).
}
Combining Lemma~\ref{lem:commutatorEstimates} with Corollary~\ref{cor:mPhaseGradDiff} and all the bounds of Section~\ref{sec:estimatesCorrections}, we obtain a cost of \eqref{eq:tauInvIsOK} (and therefore \eqref{ineq:bigGoalForAdvec}) for the advective derivative.  Further spatial derivatives cost at most $C N \Xi$ as all the terms are in fact localized to frequencies of order $\la$.

This estimate concludes the proof of the Main Lemma.


\section{Proof of the Main Theorem} \label{sec:mainLemWorks}
In this Section, we explain how Theorem \ref{thm:mainThm} can be deduced from the Main Lemma, Lemma~\ref{lem:mainLemma}.  
More specifically, the Theorem we establish directly is the following:

\begin{thm}\label{thm:mainThm2} As in the hypotheses of Theorem \ref{thm:mainThm}, let $\a < 1/9$, let $\ep > 0$ be given, and let $f : \R \times \T^2 \to \R$ be any smooth function of compact support
\[ \supp f \subseteq I \times \T^2 \]
for which the integral
\ali{
 \int_{\T^2} f(t,x) dx &= 0, \qquad t \in \R 
}
remains constant in time.  Then there exists a function $\th : \R \times \T^2 \to \R$ with the following properties:
\begin{enumerate}
\item $\th$ satisfies the Active Scalar Equation \eqref{eq:activeScalar} in the sense of distributions.
\item The scalar field $\th$ and the drift velocity $u^l = T^l[\th]$ both belong to the H\"{o}lder class $\th, u^l \in C_{t,x}^\a$
\item $\th$ is supported in the time interval
\ali{ 
\supp \th \subseteq  I_\ep \times \T^2, \label{eq:suppOfTh}
}
where $I_\ep = [a_0 - \ep, b_0 + \ep]$ is an $\ep$-neighborhood of the interval $I = [a_0, b_0]$
\item $\th$ satisfies a uniform estimate
\ali{
\co{\th} &\leq C \label{eq:c0bdTh}
}
with $C$ depending only on $f$.
\item For any smooth function $\phi : \R \times \T^2 \to \C$, we have
\ali{
\left| \int_{\R \times \T^2} (\th - f) \phi~dt dx \right| &\leq C \ep \| \nab \phi \|_{L^1_{t,x}(I_\ep \times \T^2)} \label{ineq:weakApprox}
}
for some constant $C$ depending on $f$.
\end{enumerate}
\end{thm}
Theorem \ref{thm:mainThm} follows from Theorem \ref{thm:mainThm2} by a straightforward argument that is implicit in Section \ref{sec:proofOfCor} below.  

Our starting point is the observation that the function $f$ can be viewed as a solution to the scalar-stress equation
\ali{
\pr_t f + \pr_l( f u^l) &= \pr_l R^l \label{eq:scalStressForf} \\
u^l &= T^l[f] \notag \\
R^j &= \pr^j \De^{-1} [ \pr_t f + \pr_l( f u^l) ] \label{eq:0thstress}
}
thanks to the condition 
\[\int_{\T^2} \pr_t f dx = \fr{d}{dt} \int_{\T^2} f(t,x) dx = 0.\]  
Furthermore, the functions $(f, u^l, R^l)$ in \eqref{eq:scalStressForf} are all smooth functions on $\R \times \T^2$ with support contained in a finite time interval $I \times \T^2$.  In particular, the scalar function $\th_{(0)} = f$ can be viewed as part of a smooth, compactly supported solution $(f, u^l, c_A, R^l)$ to the compound scalar stress equation \eqref{eq:compScalStress} with $c_A = 0$. 

Our proof of Theorem \ref{thm:mainThm2} will be completed once we prove the following Claim.
\begin{Claim}\label{Properties}Under the assumptions of Theorem \ref{thm:mainThm2}, there exists a sequence sequence of scalar-stress fields $(\th_{(k)}, u_{(k)}^l , c_{A, (k)}, R_{J,(k)}^l )$  satisfying the following properties.
\begin{enumerate}
\item \label{item:alternating} For even indices $k = 0, 2, 4, \ldots$, $(\th_{(k)}, u_{(k)}^l , c_{A, (k)}, R_{J,(k)}^l )$ solves the Compound Scalar-Stress Equation \eqref{eq:compScalStress} with vector $A^l$, whereas for odd indices $k = 1, 3, 5, \ldots$, $(\th_{(k)}, u_{(k)}^l , c_{A, (k)}, R_{J,(k)}^l )$ solves the Compound Scalar-Stress Equation \eqref{eq:compScalStress} with vector $B^l$.  Here $A^l$ and $B^l$ are defined as in \eqref{eq:AlBlare}.
\item \label{item:vanishingStress} We have $\co{c_{A,(k)}} + \co{ R_{J,(k)} } \to 0$ as $k \to \infty$
\item \label{item:coBddness} The sequences $\th_{(k)}, u_{(k)}^l$ are Cauchy in $C_{t,x}^0$ with uniform bounds on $\co{\th_{(k)}}, \co{u_{(k)}}$ depending only on $f$.
\item \label{item:CalphCauchy} The sequences $\th_{(k)}, u_{(k)}^l$ are also Cauchy in $C_{t,x}^\a$.
\item \label{item:suppBounds} We have $\supp \th_{(k)} \subseteq I_\ep \times \T^2$ for all $k$. 
\item \label{item:bounds} The estimate \eqref{ineq:weakApprox} holds for $\th_{(k)}$ uniformly in $k$.
\end{enumerate}
\end{Claim}
The scalar stress fields described in Claim \ref{Properties} will be constructed by iteration of Lemma \ref{lem:mainLemma}.

\subsection{The Base Case}\label{sec:baseCase}

To initialize the construction, we set $\th_{(0)} = f$, $u_{(0)}^l = T^l[f]$, $c_{A,(0)} = 0$, and $R_{J,(0)}$ as in \eqref{eq:0thstress}.  We define $I_{(0)}$ to be the smallest closed inerval such that $\supp f \subseteq I_{(0)} \times \T^2$.  We set 
\[ e_{J,(0)} = \co{ R_{J,(0)}}. \] 
In order to be consistent with the iteration rules \eqref{eq:iterateevk}-\eqref{eq:iterateXik} below and to maintain the inequality $e_v \geq e_R \geq e_J$ during the iteration, we take
\[ e_{v,(0)} = e_{R,(0)} = K_1 e_{J,(0)} \]
where $K_1$ is the constant in Lemma~\ref{lem:mainLemma}.
Now let $\overline{\Xi}$ be a sufficiently large constant such that the bounds of Definition \ref{def:freqEnDef} hold with $L = 2$ for the frequency energy levels $(\Xi, e_v, e_R, e_J) = (\overline{\Xi}, e_{J,(0)}, e_{J,(0)}, e_{J,(0)})$. 

We will choose our initial frequency level $\Xi_{(0)}$ to be even larger than the parameter $\overline{\Xi}$.  More specifically, $\Xi_{(0)}$ will take the form
\ali{
\Xi_{(0)} &= Y \overline{\Xi} \label{eq:initXi0}
}
Here $Y \geq 1$ is a large parameter whose purpose will ultimately be to make sure that the time interval containing the support of the solution will be small without disturbing the required $C^0$ estimate.  In terms of the construction, choosing the parameter $Y$ to be large will imply that we perform the iteration with a large frequency parameter $\lambda$ and a small lifespan parameter $\tau$ when we iterate the Main Lemma.


\subsection{Choice of Parameters for \texorpdfstring{$k \geq 1$.}{k >= 1}}\label{sec:chooseParameters}

We will proceed with the proof by iteration of the Main Lemma, which requires us to specify a sequence of frequency energy levels $(\Xi_{(k)}, e_{v,(k)}, e_{R,(k)}, e_{J,(k)})$, a sequence of functions $e_{(k)}(t)$ prescribing the energy increment, a sequence of intervals $I_{(k)}$ containing the support of the compound scalar-stress fields, and a sequence of frequency growth factors $N_{(k)} \geq 2$.  The present section is devoted to choosing these parameters, and studying how these parameters grow or decay during the iteration.

We will choose our frequency energy levels so that the following iteration rules hold for all $k \geq 0$:
\ali{
e_{v,(k+1)} &= \; e_{R,(k)} \label{eq:iterateevk}\\
e_{R,(k+1)} &= K_1 e_{J,(k)} \\
e_{J,(k+1)} &= \fr{ e_{J,(k)}}{Z} \label{eq:iterateeJk} \\
\Xi_{(k+1)} &= C_0 N_{(k)} \Xi_{(k)} \label{eq:iterateXik}
}
The parameter $Z$ will be chosen in the proof to be a large constant satisfying $Z\geq K_1 \geq 1$.  From \eqref{eq:iterateevk}-\eqref{eq:iterateeJk} and the choices of Section~\ref{sec:baseCase}, 
the energy levels decay exponentially according to the following pattern:
\ali{
\begin{split}
(e_v, e_R, e_J)_{(0)} &= ( K_1 e_{J,(0)}, K_1 e_{J,(0)}, e_{J,(0)} ) \\
(e_v, e_R, e_J)_{(1)} &= \left( K_1 e_{J,(0)}, K_1 e_{J,(0)}, \fr{1}{Z} e_{J,(0)} \right) \\
(e_v, e_R, e_J)_{(k)} &= \left( \fr{K_1}{Z^{k-2}} e_{J,(0)}, \fr{K_1}{Z^{k-1}} e_{J,(0)}, \fr{1}{Z^k} e_{J,(0)} \right), \qquad k \geq 2
\end{split} \label{eq:expDecayEnergy}
}
The constant $C_0$ in \eqref{eq:iterateXik} is the constant $C$ appearing in line \eqref{eq:theNewEnergyLevel} of the Main Lemma.  Thus, $C_0$ will depend on how we construct our energy increment functions $e_{(k)}(t)$, which will be specified momentarily.  According to the bound of line \eqref{eq:theNewEnergyLevel}, we have that 
\ali{
e_{J,(k+1)} &= \left( \fr{e_{v,(k)}^{1/2}}{e_{R,(k)}^{1/2} N_{(k)}} \right)^{1/2} e_{R,(k)} \label{eq:NewEjk}
}
The iteration rules \eqref{eq:iterateevk}-\eqref{eq:iterateeJk} are therefore achieved by taking
\ali{
 N_{(k)} &= \left( \fr{e_{v,(k)}}{e_{R,(k)}} \right)^{1/2} \left( \fr{e_{R,(k)}}{e_{J,(k)}} \right)^2 Z^2   \label{eq:Nkchoice}
}
More specifically, recalling \eqref{eq:expDecayEnergy}, 
\ali{
\begin{split}
 N_{(0)} = K_1^2 Z^2, \quad  N_{(1)} = K_1^2 Z^4, \quad  N_{(k)} = K_1^2 Z^{9/2}, \qquad k \geq 2.
\end{split}\label{eq:choiceOfNk} 
}
As we always have $\left( \fr{e_{v,(k)}}{e_{R,(k)}} \right)^{3/2} \leq Z^{3/2} \leq N$, the assumption of line \eqref{eq:conditionsOnN2} is always satisfied, so this choice of $N_{(k)}$ is admissible.  
With this choice, iteration of \eqref{eq:iterateXik} results in exponential growth of the frequency levels 
\ali{
Z^{2k} \Xi_{(0)} \leq \Xi_{(k)} &\leq C_0^{k} K_1^{2k} Z^{(9/2) k} \Xi_{(0)} \label{eq:XiGrowsGeom}
}
for all $k \geq 0$.  

We will now specify how our sequence of energy functions $e_{(k)}(t)$ and time intervals $I_{(k)}$ will be chosen, beginning with stage $k = 0$.  Define 
\ali{
\hat{\tau}_{(k)} = \Xi_{(k)}^{-1} e_{v,(k)}^{-1/2}
}
to be the natural time scale associated to these frequency energy levels.  Let $I_{(0)}$ be the time interval containing the support of the initial scalar-stress field.  Let $\eta_\ep(t)$ be a standard, non-negative mollifying kernel in one variable, with support in $|t| \leq \ep$.  The initial energy function $e_{(0)}(t)$ is required to satisfy the lower bound $e_{(0)}(t) \geq K_0 e_{R,(0)}$ on the time interval $I_{(0)} \pm \hat{\tau}_{(0)}$ according to \eqref{eq:lowBoundEoftx}, and must have a square root $e_{(0)}^{1/2}(t)$ which satisfies bounds of the form \eqref{ineq:goodEnergy}.  We construct $e_{(0)}(t)$ by mollifying the characteristic function of $I_{(0)}$ according to the formula
\ali{
e_{(0)}^{1/2}(t) &= (2 K_0)^{1/2} e_{J,(0)}^{1/2} ~ \eta_{\hat{\tau}} \ast \chi_{I_{(0)} \pm 3 \hat{\tau}}(t) \label{eq:defineEo1half}
}
With this choice, the lower bound \eqref{eq:lowBoundEoftx} and the bounds \eqref{ineq:goodEnergy} are satisfied with $K = K_0$ and with $M$ being some absolute constant which arises from differentiating the mollifier.  Having constructed $e_{(0)}(t)$, we can apply Lemma \ref{lem:mainLemma} to obtain a solution $(\th_{(1)}, u_{(1)}^l, c_{A,(1)}, R_{J,(1)}^l)$ to the Compound Scalar-Stress equation with vector $B^l$ with support in the interval $I_{(1)} \times \T^2$, $I_{(1)} = I_{(0)} \pm 4 \hat{\tau}_{(0)}$.  

We now iterate this procedure to form a sequence of scalar stress fields $(\th_{(k)}, u_{(k)}^l, c_{A,(k)}, R_{J,(k)}^l)$ whose compound frequency energy levels obey the rules -\eqref{eq:iterateevk}-\eqref{eq:iterateXik} by choosing $N_{(k)} = Z^{9/2}$ according to \eqref{eq:choiceOfNk}.  We define
\[ e_{(k)}^{1/2}(t) = (2 K_0)^{1/2} e_{J,(k)}^{1/2} ~ \eta_{\hat{\tau}} \ast \chi_{I_{(k)} \pm 3 \hat{\tau}}(t) \]
so that the bounds on $e_{(k)}^{1/2}$ are consistent with the bounds on \eqref{eq:defineEo1half}, and Lemma \ref{lem:mainLemma} applies with the same constant $M$.  According to Lemma \ref{lem:mainLemma}, the time intervals containing the support of the scalar stress fields support grow according to the rule
\ali{
 I_{(k+1)} &= I_{(k)} \pm (4 \hat{\tau}_k) \label{eq:timeIntervalGrowth}
}
In \eqref{eq:timeIntervalGrowth} and below, we use the notation $I \pm \de$ to denote the $\de$-neighborhood of an interval $I$.  In other words, $I \pm \de = [a-\de, b+\de]$ if $I = [a,b]$.  During this iteration, the vector in the scalar-stress equation alternates between $A^l$ and $B^l$ as in Property \ref{item:alternating} of Claim \ref{Properties}.

We have now defined our iteration up to the choice of the parameters $Y$ and $Z$.  We will choose these parameters in the following Subsection to ensure that the properties listed in Claim \ref{Properties} are all satisfied.

\subsection{Verifying Claim \ref{Properties}  }\label{sec:verify:claim}
We now verify that the properties in Claim \ref{Properties} hold for sufficiently large values of $Y$ and $Z$.

\noindent {\bf Property \ref{item:alternating}} 
This property follows immediately from the construction.

\noindent {\bf Property \ref{item:vanishingStress} }
To check that the error terms converge uniformly to $0$, we observe that
\ALI{
\co{R_{J,(k)}} &\leq e_{J,(k)} = Z^{-k} e_{J,(0)}
}
from \eqref{eq:expDecayEnergy}, and the same type of estimate holds for $\co{ c_{A,(k+1)} }$.  Thus, both terms composing the stress error converge uniformly to $0$.

\noindent {\bf Property \ref{item:coBddness} }
Here we verify that the sequence $\th_{(n)}, u_{(n)}^l$ is Cauchy in $C^0$.  Recall that, for $n \geq 1$ we have
\ali{
\th_{(n)} = \th_{(0)} + \sum_{k = 0}^{n-1} \Th_{(k)}, \quad u_{(n)}^l = u_{(0)}^l + \sum_{k=0}^{n-1} U_{(k)} \label{eq:thnUnissum}
}
where the properties of $\Th_{(k)}$ and $U_{(k)}^l$ are as described in Lemma \ref{lem:mainLemma}.  The functions $\th_{(0)}$ and $u_{(0)}^l$ are smooth with compact support, and are therefore uniformly bounded.  Our estimates for the corrections have the form
\ali{
 \co{\Th_{(k)}} + \co{U_{(k)}} &\leq C e_{R,(k)}^{1/2} \label{eq:c0bdForThkUk}
}
From \eqref{eq:expDecayEnergy}, the bounds $e_{R,(k)}$ decays exponentially in $k$ for any choice of $Z \geq 2$, and both series in \eqref{eq:thnUnissum} therefore converge uniformly.  In particular, as $\th_{(0)} = f$, we have
\ali{
 \co{\th_{(k)}} + \co{u_{(k)}} &\leq \co{f} + C e_{R,(0)}^{1/2}, \qquad k \geq 0  \label{eq:unifBdthkuk}
}
where $C$ is proportional to the constant in Lemma~\ref{lem:mainLemma}.  In particular, the bound~\eqref{eq:unifBdthkuk} depends only on $f$, and does not depend on our choices of parameters $Y$ and $Z$.

\noindent {\bf Property \ref{item:CalphCauchy}} 
We now verify that the series \eqref{eq:thnUnissum} also converges in $C_{t,x}^\a$ once $Z$ is chosen large enough.  
We claim that the bounds from Lemma \ref{lem:mainLemma} give
\ali{
\co{ \nab \Th_{(k)}} + \co{ \nab U_{(k)} }  + \co{\pr_t \Th_{(k)}} + \co{\pr_t U_{(k)}} &\leq C N_{(k)} \Xi_{(k)} e_{R,(k)}^{1/2} \label{eq:nabThkUkbd}
}
The bounds on $\co{ \nab \Th_{(k)}} + \co{ \nab U_{(k)} } $ follow directly from Lemma \ref{lem:mainLemma}.  We obtain the same bound for the time derivatives by writing
\ali{
\pr_t \Th_{(k)} &= - u_{(k)} \cdot \nab \Th_{(k)} + (\pr_t + u_{(k)} \cdot \nab)\Th_{(k)} \label{eq:timeIsAdvecPlus}
}
and similarly for $U_{(k)}$.  As we have shown in \eqref{eq:unifBdthkuk} that the sequence $\co{u_{(k)}}$ is uniformly bounded by some constant, we have that the terms $- u_{(k)} \cdot \nab \Th_{(k)}$ and $- u_{(k)} \cdot \nab U_{(k)}$ both obey the estimate \eqref{eq:nabThkUkbd}. Lemma \ref{lem:mainLemma} also supplies the following bound on the advective derivative:
\ALI{
\co{ (\pr_t + u_{(k)} \cdot \nab) \Th_{(k)} } + &\co{(\pr_t + u_{(k)} \cdot \nab) U_{(k)}} \leq C {\bf b}_{(k)}^{-1/2} \Xi_{(k)}^{1/2} e_{v,(k)}^{1/2} e_{R,(k)}^{1/2} \\
{\bf b_{(k)}} &=  N_{(k)}^{-1} (e_{v,(k)}^{1/2}/e_{R,(k)}^{1/2})
}
Note that the parameter $N_{(k)} = K_1^2 Z^{9/2}$ and the ratio $\fr{e_{v,(k)}^{1/2}}{e_{R,(k)}^{1/2}} = Z^{1/2}$ are both independent of $k$ once $k \geq 2$, while $e_{v,(k)} = K_1^2 Z^{-(k-2)} e_{J,(0)}$ decays to $0$ exponentially.  Thus, the estimate for the advective derivative is even better than the bound \eqref{eq:nabThkUkbd}.  From \eqref{eq:timeIsAdvecPlus} we now conclude that \eqref{eq:nabThkUkbd} holds for the time derivative as well.  

Interpolation of \eqref{eq:nabThkUkbd} and \eqref{eq:c0bdForThkUk} gives
\ali{
\| \Th_{(k)} \|_{C_{t,x}^\a} + \| U_{(k)} \|_{C_{t,x}^\a} &\leq C [N_{(k)} \Xi_{(k)}]^\a e_{R,(k)}^{1/2} 
}
Applying \eqref{eq:XiGrowsGeom} and \eqref{eq:choiceOfNk}, we have
\ali{
\| \Th_{(k)} \|_{C_{t,x}^\a} + \| U_{(k)} \|_{C_{t,x}^\a} &\leq C_{Z, K_1, C_0} \Big(C_0^\a  K_1^{2 \a} Z^{\left(\fr{9}{2} \a - \fr{1}{2}\right)} \Big)^k ~ \Big( \Xi_{(0)}^\a e_{R,(0)}^{1/2} \Big) \label{eq:correctCtxaBd}
} 
As we have assumed $\a < 1/9$, we can take $Z$ large enough depending on $\a$, $K_1$ and $C_0$ so that 
\ali{
 Z^{\left(\fr{9}{2} \a - \fr{1}{2}\right)} &< K_1^{2\a} C_0^\a  \label{eq:pickABigZ}
}
Under this assumption, the right hand side of \eqref{eq:correctCtxaBd} tends to $0$ exponentially fast as $k \to \infty$, and it follows that the series \eqref{eq:thnUnissum} converges in $C_{t,x}^\a$.

\noindent {\bf Property \ref{item:suppBounds}}
To bound the support of $\th_{(k)}$, recall that $\supp \th_{(k)} \subseteq I_{(k)}$, where $I_{(0)}$ is the smallest time interval containing the support of $f$, and the intervals $I_{(k)}$ grow according to the rule \eqref{eq:timeIntervalGrowth}.  As a consequence,
we have (in terms of the notation introduced in \eqref{eq:timeIntervalGrowth})
\ALI{
 I_{(k)} \subseteq I_{(0)} \pm T , \qquad k = 0, 1, 2, \ldots  \\
 T = 4 \sum_{k = 0}^\infty \hat{\tau}_{(k)} = 4 \sum_{k = 0}^\infty \Xi_{(k)}^{-1} e_{v,(k)}^{-1/2}
}
We recall now that $e_{v,(0)} = e_{v,(1)} = e_{v,(2)} = K_1 e_{J,(0)}$ while $e_{v,(k)}$ decays exponentially as \eqref{eq:expDecayEnergy} for $k \geq 2$.  We also recall the lower bound in \eqref{eq:XiGrowsGeom} to obtain
\ali{
T &\leq 4 \Xi_{(0)}^{-1} e_{J,(0)}^{-1/2} \left( 2 + \fr{1}{1 - (C_0 Z^{5/2} )^{-1}} \right)
}
Recalling the definition \eqref{eq:initXi0} of $\Xi_{(0)}$, and noting that $C_0 \geq 2$ and $Z \geq 1$, we have
\ali{
T &\leq 8 Y^{-1} \overline{\Xi}^{-1} e_{J,(0)}^{-1/2}
}
Property \ref{item:suppBounds} is satisfied for the $\ep > 0$ given in \eqref{item:suppBounds} once $Y$ is chosen sufficiently large to ensure $T < \ep$.

\noindent {\bf Property \ref{item:bounds}}
For a smooth test function $\phi$ with compact support, we have
\ali{
\int_{\R \times \T^2} ( \th - f ) \phi(t,x) dt dx &= \int_{\R \times \T^2} ( \th - \th_{(0)} ) \phi(t,x) dt dx \\
&= \sum_{k = 0}^\infty \int \Th_{(k)} \phi(t,x) dt dx
}
According to Lemma \ref{lem:mainLemma}, we can write the corrections in divergence form $\Th_{(k)} = \pr_l W_{(k)}^l $ for some vector fields $W_{(k)}^l$ obeying the estimates \eqref{eq:Wco}-\eqref{eq:matWco}.  Integrating by parts, we have
\ali{
 \sum_{k = 0}^\infty \int \Th_{(k)} \phi(t,x) dt dx &= - \sum_{k = 0}^\infty \int W_{(k)}^l \pr_l \phi(t,x) dt dx 
}
Recalling \eqref{eq:expDecayEnergy}, \eqref{eq:XiGrowsGeom} and the definition \eqref{eq:initXi0} of $\Xi_{(0)}$, we obtain 
\ali{
\Big|\int_{\R \times \T^2} ( \th - f ) \phi(t,x) dt dx \Big| &\leq \sum_{k = 0}^\infty \co{ W_{(k)} } \| \nab \phi \|_{L^1_{t,x}(I_\ep \times \T^2)} \\
&\leq C \left(\sum_{k = 0}^\infty \fr{1}{N_{(k)}\Xi_{(k)}} e_{R,(k)}^{1/2} \right) \| \nab \phi \|_{L^1_{t,x}(I_\ep \times \T^2)} \\
&= \left(\sum_{k = 0}^\infty \fr{C_0}{\Xi_{(k+1)}} e_{R,(k)}^{1/2} \right)\| \nab \phi \|_{L^1_{t,x}(I_\ep \times \T^2)} \\
&\leq C \fr{C_0}{\Xi_{(1)}} e_{R,(0)}^{1/2} \| \nab \phi \|_{L^1_{t,x}(I_\ep \times \T^2)} \\
&\leq C \fr{1}{Z^2} \Xi_{(0)}^{-1} e_{J,(0)}^{1/2} \| \nab \phi \|_{L^1_{t,x}(I_\ep \times \T^2)} \\
&\leq C \fr{1}{Z^2 Y} \overline{\Xi}^{-1} e_{J,(0)}^{1/2} \| \nab \phi \|_{L^1_{t,x}(I_\ep \times \T^2)}
}
where $C_0$ above denotes the constant in the Main Lemma.  Taking $Y$ (or alternatively $Z$) to be large enough depending on $C$, $\overline{\Xi}$ and $e_{J,(0)}$, we obtain \eqref{ineq:weakApprox}.  This choice concludes the proof of Theorem \ref{thm:mainThm2}.

\section{Proof of Corollary \ref{cor:weakLims} } \label{sec:proofOfCor}

In this Section, we explain how Theorem \ref{thm:mainThm} (or alternatively Theorem \ref{thm:mainThm2}) 
can be applied to prove Corollary \ref{cor:weakLims}, which characterizes the closure of compactly supported solutions to the active scalar equations in the space $L^\infty$ endowed with the weak-* topology.  

\begin{proof}[Proof of Corollary \ref{cor:weakLims}]
As in the statement of Theorem \ref{thm:mainThm}, consider an Active Scalar Equation \eqref{eq:activeScalar} with a smooth multiplier that is not odd.  Let $I \subseteq \R$ be a nonempty, finite, open interval.  Let $\a < 1/9$ and let $S \subseteq L^\infty$ denote the set of all weak solutions $(\th, u^l)$ to the Active Scalar equation \eqref{eq:activeScalar} which have compact support in $I \times \T^2$, and which belong to the H\"{o}lder class $(\th, u^l) \in C_{t,x}^\a$.  Let $\overline{S}$ denote the closure of $S$ in $L^\infty$ with respect to the weak-* topology.  Corollary \ref{cor:weakLims} asserts that $\overline{S}$ is equal to the space of $f \in L^\infty(I \times \T^2)$ which satisfy the conservation law $\int f(t,x) dx = 0$ as a distribution in the variable $t$.  In other words, we assume that for every smooth test function $\eta(t) : I \to \R$ with compact support, we have
\ali{
 \int_{I \times \T^2} \eta(t) f(t,x) dt dx &= 0  \label{eq:consLawIntegralForm}
}
First, observe that any $f \in L^\infty$ which belongs to $\overline{S}$ must satisfy \eqref{eq:consLawIntegralForm} for all such $\eta(t)$, since the integration against $\eta(t)$ is continuous with respect to the weak-* topology, and because equality \eqref{eq:consLawIntegralForm} is satisfied by all of the elements $(\th, u^l) \in S$.  This conservation law is proven for each $(\th, u^l) \in S$ by writing the test function in \eqref{eq:consLawIntegralForm} as $\eta = \tilde{\eta} + ( \eta - \tilde{\eta})$, where $\tilde{\eta}$ is a smooth function whose support is disjoint support from that of $(\th, u^l)$ that satisfies
\[\int_I \tilde{\eta}(t) dt = \int_I \eta(t) dt \]  
This condition allows us to write $\eta - \tilde{\eta} = \fr{d}{dt} h(t)$, where $h(t)$ is smooth and compactly supported in $I$.  The definition of weak solution for \eqref{eq:activeScalar} then implies
\ALI{
\int_{I \times \T^2} \eta(t) \th(t,x) dt dx &= \int_{I \times \T^2} \tilde{\eta}(t) \theta(t,x) dt dx + \int_{I \times \T^2} (\eta - \tilde{\eta})(t) \theta(t,x) dt dx \\
&= \int \fr{\pr}{\pr t}h(t) \th(t,x) dt dx \\
&= \int u^l \fr{\pr}{\pr x^l} h(t) ~ \th(t,x) dt dx = 0 
}
We now show conversely that every $f \in L^\infty$ satisfying \eqref{eq:consLawIntegralForm} belongs to $\overline{S}$.  Let us assume by contradiction that $f \notin \overline{S}$.  By definition of the weak-* topology, there exists a finite collection $\{ \eta_1, \ldots, \eta_m \} \subseteq L^1(I \times \T^2)$ and $\de > 0$ such that for all $\th \in S$ the lower bound
\ali{
 \Big| \int (f(t,x) - \theta(t,x)) \eta_j(t,x) dt dx \Big| &\geq \de  \label{eq:badEta}
}
holds for at least one $\eta_j \in \{ \eta_1, \ldots, \eta_m \}$.

Let $\tilde{f} \in L^\infty(I \times \T^2)$ be a smooth function of compact support with  $\| \tilde{f} \|_{L^\infty} \leq \| f\|_{L^\infty}$ such that property \eqref{eq:consLawIntegralForm} holds for $\tilde{f}$ and for all such $\eta_j$, and we have the bound
\ali{
 \Big| \int (f(t,x) - \tilde{f}(t,x)) \eta_j(t,x) dt dx \Big| &\leq \de/4  \label{eq:closeFtildeForEta}
}
   Such a function $\tilde{f}$ can be constructed by first applying a smooth cutoff in time to restrict to a compact subset of $I \times \T^2$, and then convolving with a mollifier in time and space, noting that both operations preserve the property \eqref{eq:consLawIntegralForm} without enlarging the $L^\infty$ norm.   Inequality \eqref{eq:closeFtildeForEta} is established by duality, as the adjoint cutoff and mollifier operations converge strongly in $L^1$ when applied to each $\eta_j$.  We choose the mollification in such a way that the support of $\tilde{f}$ remains inside a time interval $\tilde{I}$ strictly smaller than $I$ with $\tilde{I} \pm \tau \subseteq I$ for some $\tau > 0$.


Now apply\footnote{At this point, Theorem \ref{thm:mainThm} would also suffice.} Theorem \ref{thm:mainThm2} for the function $\tilde{f}$ with $\ep = 1/n$ to obtain a sequence $(\th_n, u_n^l) \in S$ such that the bound $\| \th_n \|_{L^\infty} \leq A$ holds uniformly, and \eqref{ineq:weakApprox} holds for $\tilde{f}$ with $\ep = 1/n$.  We assume here that $n \geq \tau^{-1}$ is large enough to ensure $(\th_n, u_n^l)$ have compact support in $I \times \T^2$ thanks to the compact support of $\tilde{f}$ and  \eqref{eq:suppOfTh}.  Now let $\tilde{\eta}_j$ be smooth functions of compact support in $I \times \T^2$ with $\| \eta_j- \tilde{\eta}_j \|_{L^1(I \times \T^2)} \leq \fr{\de}{5(\|f\|_{L^\infty} + A)}$.  We obtain an upper bound on the left hand side of \eqref{eq:badEta}
\ALI{
\Big| \int (f(t,x) - \theta_n(t,x)) \eta_j(t,x) dt dx \Big| &\leq \fr{\de}{4} + \Big| \int ( \tilde{f} - \th_n ) \eta_j dt dx \Big| \\
&\leq \fr{\de}{4} + \fr{\de}{5} + \Big| \int ( \tilde{f} - \th_n ) \tilde{\eta}_j dt dx \Big| \\
&\leq \fr{\de}{4} + \fr{\de}{5} + \fr{1}{n} \| \nab \tilde{\eta}_j \|_{L^1_{t,x}}
}
Taking $n$ large enough contradicts inequality \eqref{eq:badEta} and concludes the proof.
\end{proof}

\section{Proof of Theorem~\ref{thm:arbitraryNonuniqueness}}\label{sec:arbitNonuniqueProof}

In this section, we outline how Lemma~\ref{lem:mainLemma} can also be applied to yield Theorem~\ref{thm:arbitraryNonuniqueness}.  The proof follows an idea of \cite[Section 12]{isett}.  The same argument below also shows that one can glue any two solutions which have the same integral\footnote{This observation is due to Sung-Jin Oh \cite{sungOhRemark}.}.

Let $\th$ be a smooth solution of \eqref{eq:activeScalar} on $(-T, T) \times \T^2$, with multiplier $m^l$ which is not odd.  Let $\bar{\th} = \fr{1}{|\T^2|} \int_{\T^2} \th(0,x) dx$ be the average value of $\th$, which is conserved by $\th$ along the flow.  Let $\psi(t)$ be a smooth cutoff function, equal to $1$ on $|t| \leq \fr{5T}{8}$ and equal to $0$ for $|t| \geq \fr{6T}{8} = \fr{3T}{4}$.  

Now consider the scalar field $\th_{(0)}(t,x) = \psi(t) \th(t,x) + (1-\psi(t) ) \bar{\th}$.  Then $\th_{(0)}$ is an integral-conserving scalar field (i.e. $\int_{\T^2} \pr_t \th_{(0)} dx = \fr{d}{dt} \int_{\T^2} \th_{(0)} dx = 0$), and therefore solves the scalar stress equation
\ali{
\pr_t \th_{(0)} + \pr_l( \th_{(0)} T^l[\th_{(0)}] ) &= \pr_l R^l \\
R^j &= \pr^j \De^{-1} [ \pr_t \th_{(0)} + \pr_l( \th_{(0)} T^l[\th_{(0)}] ) ]
}
Note also that, because both $\th$ and $\bar{\th}$ are solutions to \eqref{eq:activeScalar}, the support of $R^l$ is contained in the support of $\psi'(t)$, namely
\[ \supp R^l(t,x) \subseteq \{ \fr{5T}{8} \leq |t| \leq \fr{6T}{8} \} \times \T^2 \]
Repeating the argument of Sections~\ref{sec:baseCase}-\ref{sec:verify:claim}, we can now iterate Lemma~\ref{lem:mainLemma} to obtain a sequence of solutions $\th_{(k)}$ to the compound scalar stress equation, such that
\[ \supp (\th_{(k)} - \th_{(0)}) \subseteq \{ \fr{T}{2} \leq |t| \leq \fr{4T}{5} \} \times \T^2 \]
for all indices $k \geq 0$, and such that $\th_{(k)} \to \tilde{\th}$ converge in $C_{t,x}^\a$ to a solution of \eqref{eq:activeScalar}.  At this point, the main difference in the argument is that we choose energy functions $e_{(k)}(t)$ which are supported within pairs of intervals containing a small neighborhood of $\{ \fr{5T}{8} \leq |t| \leq \fr{6T}{8} \}$.  (In fact, the argument is simpler at this point because we do not need to achieve a weak approximation, and hence there is no need to introduce the parameter $Y$.)  As we can take this intervals of support to form an arbitrarily small neighborhood of $\{ \fr{5T}{8} \leq |t| \leq \fr{6T}{8} \}$, we can keep the support of the iteration contained within $\{ \fr{T}{2} \leq |t| \leq \fr{4T}{5} \}$, and thereby obtain Theorem~\ref{thm:arbitraryNonuniqueness}.


\section{Proof of Weak Rigidity for Odd Active Scalars} \label{sec:rididityOdd}
In this section we give the proof of Theorem~\ref{thm:theOddCase}. 
Let $\theta \in \{\theta_n \}_{n \geq 0}$ be a weak solution to \eqref{eq:activeScalar}, with associated velocity field $u^l = T^l[\theta]$. Also let $\phi \in \DD(I \times \T^2)$ be a fixed test function. The proof of the theorem is based on the following computation. For each fixed time $t$, let 
\begin{align} 
N_t[\th,\phi] = \int_{\T^2} \theta(t,x) \; u^l(t,x) \;\pr_l \phi(t,x) \;dx \label{eq:nonlinTermNotCpctYet}
\end{align}
denote the nonlinear term integrated over the time $t$ slice. Since $T^l$ is given by a Fourier multiplier, it commutes with differentiation, and upon integrating by parts several times we obtain
\begin{align} 
N_t[\th,\phi] &= \int_{\T^2} \pr^k \Delta^{-1} \pr_k \theta  \; \pr^j T^l[  \Delta^{-1} \pr_j\theta] \; \pr_l \phi \; dx \notag\\
&= - \int_{\T^2}  \Delta^{-1} \pr_k \theta  \; \pr^k \pr^j T^l[ \Delta^{-1} \pr_j  \theta] \; \pr_l \phi \;  dx - \int_{\T^2}  \Delta^{-1} \pr_k \theta  \;  \pr^j T^l[  \Delta^{-1}\pr_j \theta] \; \pr^k\pr_l \phi \;  dx\notag\\
&= \int_{\T^2} \pr_k \Delta^{-1}  \pr^j\theta  \; \pr^k  T^l[ \Delta^{-1} \pr_j  \theta] \; \pr_l \phi \; dx \notag\\
& \qquad + \int_{\T^2} \Delta^{-1} \pr_k \theta  \; \pr^k  T^l[ \Delta^{-1} \pr_j  \theta] \; \pr^j  \pr_l \phi \; dx - \int_{\T^2}  \Delta^{-1} \pr_k \theta  \;  \pr^j T^l[  \Delta^{-1}\pr_j \theta] \; \pr^k\pr_l \phi \; dx\notag\\
&= - \int_{\T^2}  \Delta^{-1}  \pr^j\theta  \; \pr_j  T^l[ \theta] \; \pr_l \phi \; dx - \int_{\T^2} \Delta^{-1}  \pr^j\theta  \; \pr^k  T^l[ \Delta^{-1} \pr_j  \theta] \; \pr_k \pr_l \phi \; dx\notag\\
& \qquad + \int_{\T^2} \Delta^{-1} \pr_k \theta  \; \pr^k  T^l[ \Delta^{-1} \pr_j  \theta] \; \pr^j  \pr_l \phi \; dx - \int_{\T^2}  \Delta^{-1} \pr_k \theta  \;  \pr^j T^l[  \Delta^{-1}\pr_j \theta] \; \pr^k\pr_l \phi \; dx.
\label{eq:N:rewrite:1}
\end{align}
At this stage we use that the Fourier multiplier $m^l(\xi)$ is odd in $\xi$, which implies that $\pr_j T^l$, given by the Fourier multiplier $i \xi_j m^l(\xi)$ which is even in $\xi$, is self-adjoint in $L^2(\T^2)$. We may thus write
\begin{align} 
& \int_{\T^2}  \Delta^{-1}  \pr^j\theta  \; \pr_j  T^l[ \theta] \; \pr_l \phi \; dx \notag\\
&\qquad = \int_{\T^2} \theta\;  \pr_j  T^l[ \Delta^{-1}  \pr^j\theta  \;  \pr_l \phi] \;  dx \notag\\
&\qquad = \int_{\T^2} \theta\;  \pr_j  T^l[ \Delta^{-1}  \pr^j\theta  ] \;  \pr_l \phi \; dx + \int_{\T^2} \theta \left(  \pr_j  T^l[ \Delta^{-1}  \pr^j\theta  \;  \pr_l \phi]  -  \pr_j  T^l[ \Delta^{-1}  \pr^j\theta  ] \;  \pr_l \phi  \right)  dx \notag\\
&\qquad = N_t[\th,\phi] + \int_{\T^2} \theta \; \left[\pr_j T^l, \pr_l \phi\right] \Delta^{-1} \pr^j \theta \; dx. 
\label{eq:N:rewrite:2}
\end{align}
Combining \eqref{eq:N:rewrite:1} and \eqref{eq:N:rewrite:2} we arrive at 
\begin{align} 
2 N_t[\th,\phi]
&= - \int_{\T^2} \theta \; \left[\pr_j T^l, \pr_l \phi\right] \Delta^{-1} \pr^j \theta \; dx - \int_{\T^2} \Delta^{-1}  \pr^j\theta  \; \pr^k  T^l[ \Delta^{-1} \pr_j  \theta] \; \pr_k \pr_l \phi \; dx\notag\\
& \qquad + \int_{\T^2} \Delta^{-1} \pr_k \theta  \; \pr^k  T^l[ \Delta^{-1} \pr_j  \theta] \; \pr^j  \pr_l \phi \; dx - \int_{\T^2}  \Delta^{-1} \pr_k \theta  \;  \pr^j T^l[  \Delta^{-1}\pr_j \theta] \; \pr^k\pr_l \phi \; dx .
\label{eq:N:rewrite:3}
\end{align}
From the H\"older inequality, and the bounds
\begin{align} 
\| \nabla T \Delta^{-1} \nabla \eta \|_{L^2}  &\leq C \|\eta\|_{L^2}, & \mbox{ for } \eta \in L^2(\T^2) \\
\| [ \nabla T^l, \pr_l \phi ] \eta \|_{L^2} &\leq  C \|\eta\|_{L^2} \|\phi\|_{H^{3+\epsilon}}, & \mbox{ for } \eta \in \dot{H}^1(\T^2) \label{ineq:commutator}
\end{align} 
we thus obtain from \eqref{eq:N:rewrite:3} that
\begin{align} 
|N_t[\th,\phi]| \leq C \| \theta(t, \cdot)\|_{L^2} \| \Delta^{-1} \nabla \theta(t,\cdot)\|_{L^2} \|\phi(t,\cdot) \|_{H^{3+\epsilon}}
\label{eq:N:bound}
\end{align}
for any $\epsilon>0$. The above estimate is a manifestation of the compactness inherent in $N_t$ in the spatial variables.

Since we have only assumed $\theta \in L^p(I; L^2(\T^2))$, compactness in the time variable must come from the active scalar equation.  Below we give two essentially equivalent approaches to obtaining this compactness.  The first proof is based on a variant of the Arzel\`{a}-Ascoli principle due to Aubin-Lions.  The second proof is a more direct argument in the spirit of \cite{isett2}, using Littlewood-Paley theory to extract regularity in time.

\paragraph{Time compactness via Aubin-Lions compactness lemma.}
At this stage we notice that for any weak solution $\theta \in L^p(I; L^2(\T^2))$ of \eqref{eq:activeScalar}, and any index $j$, we have we have that 
\begin{align} 
\partial_t (\Delta^{-1} \pr_j \theta) = \Delta^{-1} \pr_j \pr_l (\theta\; T^l[\theta])
\end{align}
holds in the sense of distributions, and thus 
\begin{align} 
\| \pr_t (\Delta^{-1} \pr_j \theta)\|_{L^{p/2}(I;H^{-2}(\T^2))} 
&= \| \Delta^{-1} \pr_j \pr_l (\theta\; T^l[\theta]) \|_{ L^{p/2}(I;H^{-2}(\T^2))} \notag \\
&\leq C \| \Delta^{-1} \pr_j \pr_l (\theta\; T^l[\theta]) \|_{L^{p/2}(I;L^1(\T^2))} \notag\\
&\leq C \| \theta\|_{L^p(I;L^2(\T^2))}^2
\label{eq:weak:Lipschitz}
\end{align}
in view of the compact embedding of $L^2(\T^2) \subset W^{2,1}(\T^2)$.

Now assume that 
\ali{
\theta^n \rightharpoonup f \in L^p(I; L^2(\T^2))
\label{eq:weak:conv:assume}
}
for some $p>2$.
The convergence of the mean 
\[ \int_{\T^2} \theta^n \; dx \to \int_{\T^2} f \; dx \qquad \mbox{ in } \DD'(I) \]
is automatic.
In view of the Sobolev embedding and \eqref{eq:weak:Lipschitz}, by \eqref{eq:weak:conv:assume} we have that 
\begin{align} 
&\Delta^{-1} \nabla \theta^n \mbox{ is uniformly bounded in } L^p(I; H^1(\T^2))\\
&\partial_t (\Delta^{-1} \nabla \theta^n) \mbox{ is uniformly bounded in } L^{p/2}(I;H^{-2}(\T^2)).
\end{align}
Therefore, applying the Aubin-Lions compactness lemma (see e.g.~\cite[Lemma 8.4]{CF88}), we obtain that there is a subsequence $\{\theta^{n_j}\}$ such that 
\begin{align} 
\Delta^{-1} \nabla \theta^{n_j} \to \Delta^{-1} \nabla f \in L^p(I; L^2(\T^2)),
\label{eq:strong:convergence}
\end{align}
i.e. the convergence is strong.  
To conclude, we integrate \eqref{eq:N:rewrite:3} in time, use \eqref{eq:N:bound} and \eqref{eq:strong:convergence}, and obtain that 
\[
\int_I \int_{\T^2} \theta^{n_j} \; T^l[\theta^{n_j}] \; \pr_l \phi dx dt \to \int_I \int_{\T^2} f\; T^l[f]\; \pr_l \phi dx dt
\]
for any test function $\phi$, since the product of a strong and a weak limit is a weak limit. The convergence holds in fact along any subsequence $n_j \to \infty$, and therefore holds also along the original sequence.

\paragraph{Time compactness via Littlewood-Paley theory.} 
We now give a more direct proof which illustrates the usefulness of Littlewood-Paley theory in extracting time regularity. 

Let us use the notation $P_{\leq I} \th$, $P_I \th$ and $P_{[a,b]} \th$ denote the standard, Littlewood-Paley projection operators.  Thus,
\[ \widehat{P_{\leq I} \th}(\xi) = \eta( 2^{-I} \xi) \hat{\th}(\xi), \quad  I = 0, 1, 2, \ldots \]
is a truncation of $\hat{\th}$ to frequencies of order $\supp \widehat{P_{\leq I} \th} \subseteq \{ |\xi| \leq 2^{I+1} \}$, $\eta$ is a smooth cutoff supported in $|\xi| \leq 2$ with $\eta(\xi) = 1$ for $|\xi| \leq 5/4$.  We let $P_I = P_{\leq I} - P_{\leq I-1} $ denote the Littlewood-Paley piece which occupies frequencies $\supp \widehat{P_I \th} \subseteq \{ 2^{I-1} \leq |\xi| \leq 2^{I+1} \}$.  We use the notation $P_{[a,b]} = \sum_{a \leq I \leq b} P_I$.

Now let $\phi \in C_0^\infty(I \times \T^2)$ be a smooth test function, and let $\th^n$ be a sequence of solutions to \eqref{eq:activeScalar} converging weakly to $\th_n \rightharpoonup f$ in $L^p(I; L^2(\T^2))$ for some $p > 2$ as in \eqref{eq:weak:conv:assume}.  Let $N = \int_\R N_t[\th, \phi] dt = \int_\R \int_{\T^2} \th u^l \pr_l \phi ~dx dt$ denote the full nonlinear term.

 We claim that $N[\th^n, \phi] \to N[f, \phi]$.  To simplify the calculation, a simple approximation argument allows us to assume that that $\hat{\phi}$ has compact support in $\supp \hat{\phi} \subseteq \{ |\xi| \leq 2^{r-1} \}$ for some $r \geq 0$.  
In this case, for all $\th \in \{ \th^n \}_{n \geq 0}$, we decompose the nonlinear term \eqref{eq:nonlinTermNotCpctYet} into dyadic frequency increments
\ali{
N[\th, \phi] &= N[P_{\leq r} \th, \phi] + \sum_{I = r+1}^\infty \de N_I[ \th, \phi] \label{eq:dyadicDecompN} \\
\de N_I[ \th, \phi] &= N[P_{\leq I+1} \th, \phi] - N[P_{\leq I} \th, \phi] \\
&= \int_\R \int_{\T^2} P_{I+1} \th P_{\leq I+1} u^l \pr_l \phi dx dt + \int_\R \int_{\T^2} P_{\leq I} \th P_{I+1} u^l \pr_l \phi dx dt \\
&= \int P_{I+1} \th P_{[ I - r, I + r]} u^l \pr_l \phi dx + \int P_{[I - r, I + r]} \th P_{I+1} u^l \pr_l \phi dx.
}
In the last line we take advantage of the compact support of $\hat{\phi}$ for convenience.  Using the commutator formulation \eqref{eq:N:rewrite:3}, each $\de N_I$ decomposes into several terms of the type
\[ \de N_I[\th,\phi] = \int_\R \int_{\T^2} \De^{-1} \pr_k P_{I+1} \th ~\pr^k T^l[\De^{-1} \pr_j P_{[I-r, I+r]} \th ] \pr^j \pr_l \phi ~dx dt + \tx{ other similar terms}  \]
From \eqref{ineq:commutator} and $\| \De^{-1} \nab P_I \th \|_{L_x^2} \leq C 2^{-I} \| \th \|_{L_x^2}$, each $\de N_I$ is bounded by
\ali{
| \de N_I[\th, \phi] | &\leq C_\phi 2^{-I} \| \th \|_{L_{t,x}^2}^2 \label{eq:spaceCpctnessEstimate}
}
for some constant $C_\phi$ depending on $\phi$.

We now show that the weak convergence \eqref{eq:weak:conv:assume} can in fact be upgraded to uniform convergence for each dyadic piece $P_I \th^n \to P_I f$, which implies the convergence of each term $\de N_I[\th^n , \phi] \to \de N_I[f, \phi]$.  The uniform convergence is obtained by compactness.  We start with the bounds
\[ \| P_I \th^n \|_{L_t^p L_x^\infty} + \| \nab P_I \th^n \|_{L_t^p L_x^\infty} \leq C_I \| \th^n \|_{L_t^p L_x^2} \]
Applying $P_I$ to \eqref{eq:activeScalar}, the equation $\pr_t P_I \th = - \pr_l P_I[ \th u^l]$ gives regularity in time
\[ \| \pr_t P_I \th^n \|_{L_t^{p/2} L_x^\infty} \leq C_I \| \th^n \|_{L_t^p L_x^2}^2 \]
As we have assumed $p > 2$ and uniform in $n$ bounds on $\| \th^n \|_{L_t^p L_x^2}^2$ from \eqref{eq:weak:conv:assume}, it follows by Sobolev embedding that the sequence $P_I \th^n$ for each $I$ is H\"{o}lder continuous in time and space, uniformly in $n$.  By Arzel\`{a}-Ascoli, there exists a uniformly convergent subsequence $P_I \th^{n_j}$ for each $I$.  From the weak convergence \eqref{eq:weak:conv:assume}, we have uniform convergence of $P_I \th^n \to P_I f$ on any subsequence, which implies that the original sequence $P_I \th^n \to P_I f$ converges uniformly.  

It now follows that $\de N_I[\th^n , \phi] \to \de N_I[f, \phi]$ for each index $I$ and that $N[P_{\leq r} \th^n, \phi] \to N[P_{\leq r} f, \phi]$.  We also have the estimate \eqref{eq:spaceCpctnessEstimate}, so the convergence of $N[\th^n, \phi] \to N[f, \phi]$ follows from the dominated convergence theorem applied to \eqref{eq:dyadicDecompN}.  

We remark that the same two arguments can be upgraded to prove compactness of solutions when we only assume weak convergence in $L_t^{p}H_x^s$ for some $p > 2$ and $s > -1/2$.  The main difference involves using the commutator formulation \eqref{eq:N:rewrite:3} to obtain an estimate for the time derivative from the lower regularity in space. 


\section{Conservation of the Hamiltonian for Odd Active Scalars} \label{sec:hamiltonianOdd}
In this Section we give the proof of Theorem~\ref{thm:odd:hamiltonian}. 
Recall that the symbol of the Fourier multiplier $L$ defined in \eqref{eq:L:def}, which defines the Hamiltonian is given by
\begin{align}
\hat{L}(\xi) = |\xi|^{-2} (i \xi_2 m_1(\xi) - i \xi_1 m_2(\xi) )
\end{align}
with the convention that $\hat{L}(0)=0$.
Since we are in two spatial dimensions and $\xi \cdot m(\xi) = 0$ for all nonzero vectors $\xi$, automatically we must have that
\begin{align}
m(\xi) = i \xi^\perp |\xi|^{-1} \ell(\xi)
\label{eq:ell:def}
\end{align}
for some {\em even}, zero-order homogenous, smooth on the unit sphere, {\em real-valued} {\em scalar} function $\ell(\xi)$. The fact that $\ell(\xi) \in \R$ follows from the fact that $\overline{\ell(\xi)} = \ell(-\xi) = \ell(\xi)$.
In the case of the SQG equation, $\ell(\xi) =1$.

In summary, we have that 
\begin{align}
\hat{L}(\xi) = |\xi|^{-1} \ell(\xi)
\label{eq:L:multiplier}
\end{align}
which reiterates that $L$ is a self-adjoint operator, which is smoothing of degree $-1$ when $\ell(\xi)$ is nonvanishing on the unit sphere. 
The Hamiltonian then is 
\begin{align}
H(t) = \int_{\T^2} \theta(t,x) L\theta(t,x) dx
\end{align}
or equivalently, in view of Plancherel's theorem,
\begin{align}
H(t) = \sum_{k \in \Z^2_\ast} |\hat{\theta}(t,k)|^2 |k|^{-1} \ell(k).
\end{align}

Let $\phi_\ep = \ep^{-2} \phi(x/\ep)$ be a standard mollifier on $\T^2$, and denote
\[
\cdot_\ep = \cdot \ast \phi_\ep, \qquad \cdot_{\ep,\ep} = \cdot \ast \phi_\ep \ast \phi_\ep.
\]

The conservation of the Hamiltonian $H$ for solutions of \eqref{eq:activeScalar} is implied by establishing that $\fr{d}{dt} H(t) = 0$ as a distribution in $t$.  Namely, we show that
\begin{align}
\lim_{\ep\to0}\int_{\R} \int_{\T^2} \eta'(t) \theta_\ep (t,x) L \theta_\ep (t,x) dx dt = 0
\label{eq:Hamiltonian:todo:1}
\end{align}
holds for every smooth function $\eta(t)$ which is supported in $I$.  Note that since mollification $\ast \phi_\ep$ is given by a Fourier multiplier, it commutes with $L$.

Considering the test function $\eta(t) L \theta_{\ep,\ep}$ in the weak formulation of \eqref{eq:activeScalar}, we arrive at
\begin{align}
\int_{\R} \int_{\T^2} \theta \pr_t (\eta L\theta_{\ep,\ep}) dx dt + \int_{\R}\int_{\T^2} \theta u^l \pr_l (\eta L\theta_{\ep,\ep}) dx dt = 0,
\label{eq:Hamiltonian:weak}
\end{align}
for every $\ep>0$.  Strictly speaking, the above test function is not smooth in time, but this restriction can be ignored after a time mollification argument, as in the proof of~\cite[Theorem 2.2]{IsettOhHeat13}.  The first term in \eqref{eq:Hamiltonian:weak} may now be rewritten as
\begin{align}
\int_{\R} \int_{\T^2} \theta \pr_t (\eta L\theta_{\ep,\ep}) dx dt 
&= \int_{\R} \int_{\T^2} \theta_{\ep} \pr_t (\eta L\theta_{\ep}) dx dt \notag\\
&= \int_{\R} \int_{\T^2} \theta_{\ep} \eta' L\theta_{\ep} dx dt + \int_{\R} \int_{\T^2} \theta_{\ep} \eta L\pr_t \theta_{\ep} dx dt \notag\\
&= \int_{\R} \int_{\T^2} \theta_{\ep} \eta' L\theta_{\ep} dx dt + \int_{\R} \int_{\T^2} L \theta_{\ep} \eta \pr_t\theta_{\ep} dx dt \notag\\
&= \int_{\R} \int_{\T^2} \theta_{\ep} \eta' L\theta_{\ep} dx dt -  \int_{\R} \int_{\T^2}\pr_t ( L \theta_{\ep,\ep} \eta) \theta dx dt.
\end{align}
Combining the above with \eqref{eq:Hamiltonian:weak} we see that establishing \eqref{eq:Hamiltonian:todo:1} is equivalent to establishing 
\begin{align}
\lim_{\ep \to 0}  \int_{\R} \eta \int_{\T^2} (\theta u^l)_{\ep} \pr_l L\theta_{\ep} dx dt = 0.
\label{eq:Hamiltonian:todo:2}
\end{align}
for $u^l = T^l[\th]$.

Up to this point, we have presented the proof of conservation of $H(t)$ analogously to the proof of energy conservation for Euler in the Onsager critical Besov space $L_t^3 B_{1/3, c(\N)}^3$ of \cite{ches}, but the remaining analysis turns out to be less subtle.  In particular, there is no need for a quadratic commutator estimate as in \cite{CET} (and the mollification above could also be simpler).

To proceed, we view the cubic term on the left hand side of \eqref{eq:Hamiltonian:todo:2} as the diagonal part of a family of trilinear operators
\ali{
Q_\ep[\th_{(1)}, \th_{(2)}, \th_{(3)}] &=\int_{\R} \eta \int_{\T^2} (\theta_{(1)} T^l[\th_{(2)}])_{\ep} (\pr_l L\theta_{(3)})_{\ep} dx dt
}
In this notation, equation \eqref{eq:Hamiltonian:todo:2} asks to show $\lim_{\ep \to 0} Q_\ep[\th, \th,\th] = 0$ for all $\th \in L_{t,x}^3$.  Observe first that the operators $Q_\ep$ satisfy the bound
\ali{
 |Q_\ep[ \th_{(1)}, \th_{(2)}, \th_{(3)} ]| &\leq \| \th_{(1)} \|_{L_{t,x}^3} \cdot \| T[\th_{(2)}] \|_{L_{t,x}^3} \cdot \| \nab L [\th_{(3)}] \|_{L_{t,x}^3} \notag \\
&\leq C \| \th_{(1)} \|_{L_{t,x}^3} \cdot \| \th_{(2)} \|_{L_{t,x}^3} \cdot \| \th_{(3)} \|_{L_{t,x}^3} \label{eq:trilinBound}
}
This bound follows from the fact that both operators $T$ and
\begin{align}
\nabla L = \nabla (-\Delta)^{-1/2} (R_2 T^1 - R_1 T^2)
\end{align}
are bounded as operators from $L^3_{x}$ to itself, thanks to the smoothness and degree $0$ homogeneity of $m$.

Because the operators $Q_\ep$ are trilinear, and the bound \eqref{eq:trilinBound} they satisfy is uniform in $\ep$, it suffices to prove \eqref{eq:Hamiltonian:todo:2} under the additional assumption that $\th$ is smooth with compact support by the density of such functions in $L_{t,x}^3$.  Assuming now that $\th$ is smooth, we may pass $\ep$ to $0$ in \eqref{eq:Hamiltonian:todo:2}, and it remains to show that
\[ \int_{\R} \eta \int_{\T^2} \theta T^l[\th] \pr_l L\theta dx dt = 0. \]

At this stage we recall that
\begin{align}
u^l = T^l [\theta] = (\pr^l)^\perp L\theta,
\end{align}
which may be seen on the Fourier side from \eqref{eq:ell:def} and \eqref{eq:L:multiplier}. As a result, we have 
\begin{align}
\int_{\T^2} \theta u^l \; \pr_l L\theta dx = \int_{\T^2} \theta (\pr^l)^\perp L \theta \; \pr_l L\theta dx = 0
\end{align}
which concludes the proof.

\section{Constraints on Weak Limits of Degenerate Active Scalars in Higher Dimensions} \label{sec:constraints}
In this Section, we give a proof of Theorem~\ref{thm:more:constraints}, which shows that the nondegeneracy condition in Theorem~\ref{thm:nDims} is necessary for the weak limit statement of Theorem~\ref{thm:mainThm}.

In this section, we assume that there is a nonzero frequency $\xi_{(0)} \in \widehat{\T^n} \setminus \{ 0\} = \Z_*^n$ in the dual lattice such that the image of the even part of the multiplier is contained in
\ali{
 \{ m(\xi) + m(-\xi) ~|~ \xi \in \widehat{\R^n} \} \subseteq \langle \xi_{(0)} \rangle^{\perp} \label{eq:imageInHyperplane}
}
In this case, we have the following restriction on weak limits of solutions to the active scalar equation, which bears resemblance to a new conservation law. 
\begin{lem}\label{lem:conditionIfIntegerPerp} Consider the active scalar equation \eqref{eq:activeScalar} on $I \times \T^n$ and suppose that the image of the even part of the multiplier is contained in the hyperplane \eqref{eq:imageInHyperplane}.  Let $T_{0}^l$ denote the Fourier multiplier with symbol
\[ \widehat{T_0^l[\th]}(\xi) = \fr{1}{2}( m^l(\xi) - m^l(-\xi) ) \hat{\th}(\xi)\]
Suppose that $\phi \in C_0^\infty(I \times \T^n)$ has the property that its spatial gradient takes values in the direction $\xi_{(0)}$
\ali{
 \nab \phi(t,x) &\in \langle \xi_{(0)} \rangle \label{eq:gradValues}
}
Suppose that $f \in L^\infty(I \times \T^n)$ can be realized as a weak-* limit $\th_{(k)} \rightharpoonup f$ in $L^\infty$ of some sequence of solutions $\th_{(k)}$ to \eqref{eq:activeScalar}.  Then
\ali{
 \int_{I \times \T^n} f \pr_t \phi + f T_0^l[f] \pr_l \phi dx dt &= 0 \label{eq:orthogCondition}
}
\end{lem}
\begin{proof}  Consider the sequence of solutions $\th_{(k)}$ to \eqref{eq:activeScalar} converging to $f$ in the $L^\infty$ weak-* topology.  Decompose the operator $T^l$ as $T^l = T_0^l + T_e^l$, where the term $T_e^l$ of the operator is the Fourier multiplier with symbol
\ali{
 \widehat{T_e^l[\th]}(\xi) &= \fr{1}{2}( m^l(\xi) + m^l(-\xi) ) \hat{\th}(\xi)
}
By equation \eqref{eq:activeScalar}, we have for all indices $k$ that
\ali{
\int_{I \times \T^n} ( \th_{(k)} \pr_t \phi + \th_{(k)} T_0^l[\th_{(k)}] \pr_l \phi ) dx dt &= - \int_{I \times \T^n} \th_{(k)} T_e^l[\th_{(k)}] \pr_l \phi dx dt = 0 \notag
}
by the condition \eqref{eq:imageInHyperplane}.  By the proof of Theorem~\ref{thm:theOddCase}, the nonlinear term is continuous with respect to weak-* limits in $L^\infty$ when restricted to active scalar fields, giving \eqref{eq:orthogCondition}.  To make this conclusion, it is important to note that, in the proof of compactness for the nonlinear term, it was not important that the operator in the nonlinear term was identical to the operator appearing in the active scalar equation.  The proof used only the oddness of the multiplier in the nonlinear term, and certain time regularity estimates from the active scalar equation coming from the boundedness properties of the operator in the equation.
\end{proof}
Assuming that the hyperplane containing the image of the even part of $m$ is in the dual lattice $\xi_{(0)} \in \widehat{\T^n}$, it is now not so hard to design a test function $\phi$ obeying \eqref{eq:gradValues} and an integral-conserving function $f$ which fails to satisfy \eqref{eq:orthogCondition}.  As a first attempt, we can let $\zeta(t)$ be a smooth cutoff in time, and take
\ali{
\phi(t,x) &= \zeta(t) \cos( \xi_{(0)} \cdot x) \\
f(t,x) &= \zeta'(t) \cos( \xi_{(0)} \cdot x) \label{eq:firstFTry}
}
Then \eqref{eq:gradValues} is satisfied, and we also have
\ali{
\int f(t,x) \pr_t \phi(t,x) dx dt &= \int (\zeta'(t))^2 \cos^2(\xi_{(0)} \cdot x) dx dt > 0 \label{eq:firstPositive}
}
is strictly positive.

The positivity of \eqref{eq:firstPositive} does not necessarily imply the failure of \eqref{eq:orthogCondition}.  However, if the equality \eqref{eq:orthogCondition} holds for this function $f$, then \eqref{eq:orthogCondition} cannot hold for the function $2f$, because the linear term (which is positive by \eqref{eq:firstPositive}) scales linearly, whereas the quadratic term scales quadratically.  Thus, at least one of $f$ or $2f$ fails to satisfy \eqref{eq:orthogCondition}, and we have Theorem~\ref{thm:more:constraints}.

\section{Concluding Discussion}\label{sec:conclusions}
Active scalar equations arise naturally in fluid dynamics in several asymptotic regimes, and as model equations for the full fluid systems.  The problem of constructing active scalar fields for which the energy $\|\theta\|_{L^2_x}$ fails to be conserved is a natural generalization of Onsager's conjecture for the Euler equations. This problem, however, encounters several additional difficulties when compared to Euler.  Most importantly, a suitable analogue of Beltrami flows, which provide an essential ingredient for obtaining regularity up to $1/5$ in the case of Euler, are unavailable.

For active scalars with multipliers that are not odd, we obtain nonuniqueness of weak solutions and even $h$-principles among integral-conserving functions for weak solutions with H\"{o}lder regularity up to $1/9$ (Theorem~\ref{thm:mainThm}, Theorem~\ref{thm:arbitraryNonuniqueness}, and Corollary~\ref{cor:weakLims}).  
Our proof is based on the observation that the interference terms which arise due to self-interactions between individual waves must vanish to leading order.  This observation allows for an approach in the spirit of the isometric embedding equations, where we eliminate one component of the error in each stage of the iteration using one-dimensional oscillations.  Our observation is general, and applies in arbitrary dimensions even to the case of the Euler equations, giving a new approach to solutions in that case as well.  However, our inability to remove more than one component of the error leads to further losses in regularity.


These results however should not be expected for multipliers which are odd.  For odd symbols, the Hamiltonian is conserved at the level of $\theta \in L^3_{t,x}$ (Theorem~\ref{thm:odd:hamiltonian}), and the nonlinearity exhibits a weak rigidity which makes it impossible to obtain an $h$-principle type result (Theorem~\ref{thm:theOddCase}).  In higher dimensions, the presence of conservation laws and other rigidity properties of weak solutions can even be sensitive to more subtle algebraic properties of the multiplier, and our method applies in a generality which is essentially optimal (Theorems~\ref{thm:nDims},~\ref{thm:more:constraints}).  

Several related questions remain open.  Part of our proof does not apply to the nonperiodic setting and some new idea could be required to produce nonperiodic solutions (currently even $L^\infty_{t,x}$ solutions have not been constructed in this case).  Other significant questions include
\begin{enumerate}
 \item In the case of SQG, exhibit a weak solution $\theta \in L^p_t L^2_x$, that does not conserve energy.
 \item In the case of SQG, exhibit a weak solution $\theta \in C^0_{x,t}$ that does not conserve energy, but does conserve the Hamiltonian.
 \item In the case of IPM, or more generally for not odd symbols, exhibit weak solutions $\theta \in C^{\alpha}_{t,x}$, with $\alpha \in (1/9,1/3)$ that do not conserve energy.
\end{enumerate}
We believe that answering these questions may shed some light into the field of two dimensional turbulence. 

Finally, further sharpening approaches which do not rely on the use of Beltrami flows may be found useful in resolving Onsager's conjecture.  The current approaches introduce anomalous time scales in the construction which are incompatible with the time regularity bounds held by more regular solutions.  With this obstruction, it seems unlikely at this time that an approach based on stationary solutions alone will go beyond the exponent $1/5$ even for $L^2$-based function spaces.  Although our construction shares in this deficiency, the cancellation of self-interference terms that lies at the heart of our proof is a general observation that arises from the structure of the equations and remains available even at longer time scales.  It is important to investigate whether further, more dynamical methods of construction can be developed.

\bibliographystyle{plain}
\bibliography{scalarDraft}

\end{document}